\documentclass[12pt]{amsart}

\usepackage{amsmath}
\usepackage{amsfonts}
\usepackage{latexsym}
\usepackage{graphicx}
\usepackage{amssymb}
\usepackage{amsthm}
\usepackage[margin=2cm]{geometry}
\usepackage{url}
\usepackage{color}
\usepackage{enumerate}
\usepackage[shortlabels]{enumitem} 
\usepackage[small]{caption}
\usepackage{cite}       
\usepackage{pgfplots}
\usepackage{longtable}

\usepackage{accents}

\usepackage{stmaryrd}

\usepackage{array}           

\usepackage{algorithm}     
\usepackage[noend]{algpseudocode}
\algrenewcommand\algorithmicrequire{\textbf{Precondition:}}
\algrenewcommand\algorithmicensure{\textbf{Postcondition:}}

\usepackage{todonotes}

\usepackage{hyperref}

\usepackage{tikz}
\usetikzlibrary{patterns}
\usetikzlibrary{shapes} 
\usetikzlibrary{decorations.footprints}

\usepackage{pifont}
\DeclareFontFamily{U}{dancers}{}
\DeclareFontShape{U}{dancers}{m}{n}{<-> dancers}{}

\usepackage{marvosym} 

\usepackage{tikz-cd}

\usepackage[T1]{fontenc}     
\usepackage{lmodern}         
\usepackage[utf8]{inputenc}  

\newtheorem{definition}{Definition}[section]
\newtheorem{lemma}[definition]{Lemma}
\newtheorem{proposition}[definition]{Proposition}

\newtheorem{corollary}[definition]{Corollary}
\newtheorem{conjecture}[definition]{Conjecture}
\newtheorem{theorem}[definition]{Theorem}
\newtheorem{remark}[definition]{Remark}

\newcommand{\N}{\mathbb{N}}
\newcommand{\Z}{\mathbb{Z}}
\newcommand{\Q}{\mathbb{Q}}
\newcommand{\R}{\mathbb{R}}

\newcommand{\A}{\mathcal{A}}
\newcommand{\B}{\mathcal{B}}
\newcommand{\Bcal}{\mathcal{B}}

\newcommand{\Fcal}{\mathcal{F}}
\newcommand{\Gcal}{\mathcal{G}}

\newcommand{\Kcal}{\mathcal{K}}
\newcommand{\Lcal}{\mathcal{L}}
\newcommand{\Pcal}{\mathcal{P}}
\newcommand{\Qcal}{\mathcal{Q}}

\newcommand{\Xcal}{\mathcal{X}}
\newcommand{\T}{\mathcal{T}}
\newcommand{\U}{\mathcal{U}}

\newcommand{\X}{\mathcal{X}}

\newcommand{\zero}{{\boldsymbol{0}}}
\newcommand{\UN}{{\boldsymbol{1}}}
\newcommand{\ba}{{\boldsymbol{a}}}
\newcommand{\be}{{\boldsymbol{e}}}

\newcommand{\bk}{{\boldsymbol{k}}}
\newcommand{\bn}{{\boldsymbol{n}}}
\newcommand{\bm}{{\boldsymbol{m}}}
\newcommand{\bp}{{\boldsymbol{p}}}

\newcommand{\bv}{{\boldsymbol{v}}}
\newcommand{\bx}{{\boldsymbol{x}}}

\newcommand{\by}{{\boldsymbol{y}}}

\newcommand{\balpha}{{\boldsymbol{\alpha}}}
\newcommand{\bbeta}{{\boldsymbol{\beta}}}

\newcommand{\dist}{\mathrm{dist}}
\newcommand{\GL}{\textrm{GL}}
\newcommand{\sccode}{\textsc{Code}}
\newcommand{\scReturnWord}{\textsc{ReturnWord}}
\newcommand{\scConfig}{\textsc{Config}}
\newcommand{\Zrange}[1]{\llbracket0,#1\rrbracket}
\newcommand{\ZZrange}[2]{{\llbracket#1,#2\rrbracket}}
\newcommand{\shape}{\textsc{shape}}
\newcommand{\defn}[1]{\textbf{#1}}
\newcommand{\generictorus}{\mathbf{T}}
\newcommand{\torusI}{\mathbb{T}}

\def\p{1.61803398874989}   

\pgfdeclaredecoration{footprintsR}{right}
{
  \state{left}[width=\pgfkeysvalueof{/pgf/decoration/stride length}/2,next state=right]
  {
    \pgftransformyshift{\pgfkeysvalueof{/pgf/decoration/foot sep}/2}
    \pgfmathparse{\pgfkeysvalueof{/pgf/decoration/foot length}}
    \pgftransformscale{\pgfmathresult}
    \pgftransformrotate{\pgfkeysvalueof{/pgf/decoration/foot angle}}
    \csname pgf@lib@foot@of@\pgfkeysvalueof{/pgf/decoration/foot of}\endcsname
  }
  \state{right}[width=\pgfkeysvalueof{/pgf/decoration/stride length}/2,next state=left]
  {
    \pgftransformyscale{-1}
    \pgftransformyshift{\pgfkeysvalueof{/pgf/decoration/foot sep}/2}
    \pgfmathparse{\pgfkeysvalueof{/pgf/decoration/foot length}}
    \pgftransformscale{\pgfmathresult}
    \pgftransformrotate{\pgfkeysvalueof{/pgf/decoration/foot angle}}
    \csname pgf@lib@foot@of@\pgfkeysvalueof{/pgf/decoration/foot of}\endcsname
  }
}

\keywords{Rauzy induction \and Markov partition \and self-induced \and
self-similar \and SFT \and polyhedron exchange transformation \and aperiodic
tiling \and Sturmian}

\subjclass[2020]{Primary 37A05; Secondary 37B51, 52C23}

\thanks{
The author acknowledges financial support from the Laboratoire International
Franco-Québécois de Recherche en Combinatoire (LIRCO), the Agence Nationale de
la Recherche through the project CODYS (ANR-18-CE40-0007) and the Horizon  2020
European  Research  Infrastructure  project OpenDreamKit (676541).}

\begin{document}

\title[Rauzy induction of polygon partitions and toral $\mathbb{Z}^2$-rotations]
{Rauzy induction of polygon partitions\\and toral $\mathbb{Z}^2$-rotations}
\author[S.~Labb\'e]{S\'ebastien Labb\'e}
\address[S.~Labb\'e]{Univ. Bordeaux, CNRS,  Bordeaux INP, LaBRI, UMR 5800, F-33400, Talence, France}
\email{sebastien.labbe@labri.fr}

\date{\today}

\begin{abstract}
    We extend the notion of Rauzy induction of interval exchange transformations to the case of toral $\mathbb{Z}^2$-rotation, i.e., $\mathbb{Z}^2$-action defined by rotations on a 2-torus. If $\mathcal{X}_{\mathcal{P},R}$ denotes the symbolic dynamical system corresponding to a partition $\mathcal{P}$ and $\mathbb{Z}^2$-action $R$ such that $R$ is Cartesian on a sub-domain $W$, we express the 2-dimensional configurations in $\mathcal{X}_{\mathcal{P},R}$ as the image under a $2$-dimensional morphism (up to a shift) of a configuration in $\mathcal{X}_{\widehat{\mathcal{P}}|_W,\widehat{R}|_W}$ where $\widehat{\mathcal{P}}|_W$ is the induced partition and $\widehat{R}|_W$ is the induced $\mathbb{Z}^2$-action on $W$. 

We focus on one example $\mathcal{X}_{\mathcal{P}_0,R_0}$ for which we obtain an eventually periodic sequence of 2-dimensional morphisms. We prove that it is the same as the substitutive structure of the minimal subshift $X_0$ of the Jeandel-Rao Wang shift computed in an earlier work by the author. As a consequence, $\mathcal{P}_0$ is a Markov partition for the associated toral $\mathbb{Z}^2$-rotation $R_0$.  It also implies that the subshift $X_0$ is uniquely ergodic and is isomorphic to the toral $\mathbb{Z}^2$-rotation $R_0$ which can be seen as a generalization for 2-dimensional subshifts of the relation between Sturmian sequences and irrational rotations on a circle.  Batteries included: the algorithms and code to reproduce the proofs are provided.

\end{abstract}

\maketitle

\setcounter{tocdepth}{1}
\tableofcontents


\section{Introduction}

A person
is walking on a sidewalk made of
alternating dark bricks of size 1 and light bricks of size $\alpha>0$ where
step zero is made at position $p\in\R$. They walk from left to right with steps
of length 1 and uses 
their right foot ($R$) on dark bricks and
their left foot ($L$) on light bricks,
\begin{figure}[h]
\begin{center}
    \includegraphics[width=0.9\linewidth,trim={1cm 2mm 0cm 6mm},clip]{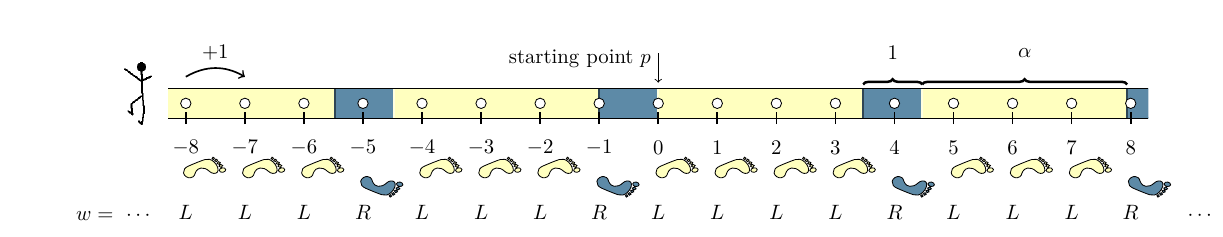}
\end{center}
    \caption{
    The shifted lattice $p+\Z$ is coded as a bi-infinite binary 
    sequence $w\in\{L,R\}^\Z$ which is a symbolic representation of $p$.}
    \label{fig:side-walk}
\end{figure}
thus constructing 
a bi-infinite binary sequence $w$ in 
$\{L,R\}^\Z$ (see Figure~\ref{fig:side-walk}).
For a fixed $\alpha$, let $X_\alpha\subset\{L,R\}^\Z$ be 
the set of sequences obtained when $p\in\R$.
When the brick length $\alpha$ is irrational, 
the partition of the circle $\R/(1+\alpha)\Z$
into the two colored bricks yields a symbolic representation of the map $x\mapsto x+1$
having a nice combinatorial interpretation
known as \emph{Sturmian sequences} \cite[Chap.~2]{MR1905123}:
\begin{enumerate}
    \item[(C1)] every sequence $w\in X_\alpha$ is obtained from a unique starting
        point in $\R/(1+\alpha)\Z$,
    \item[(C2)] $X_\alpha$ is
        the set of sequences in $\{L,R\}^\Z$ 
        whose language has $n+1$
        patterns of length $n$ for all $n\geq0$ and
        whose letter frequencies exist and satisfies 
        $\lim_{n\to\infty}\frac{\#\{-n\leq i< n\mid w_i=L\}}
                               {\#\{-n\leq i< n\mid w_i=R\}}=\alpha$.
\end{enumerate}
For example, the \emph{patterns} of length from 1 to 4 that we see in $w$ in the figure are
    $\{L,R\}$
    $\{LL,LR,RL\}$,
    $\{LLL,LLR,LRL,RLL\}$,
    $\{LLLL,LLLR,LLRL,LRLL,RLLL\}$
    and one can show in general that $w$ has $n+1$ patterns of length $n$.
    The converse is the difficult part of the proof:
any bi-infinite sequence whose letter frequencies exist and having $n+1$ distinct patterns
of size $n$ belongs to $X_\alpha$ for
some $\alpha\in\R_{>0}\setminus\Q$ and for some starting point $p\in\R/(1+\alpha)\Z$.

This result is a fundamental result in symbolic dynamics
known as Morse-Hedlund's theorem
\cite{MR0000745}, although the connection with the number of fixed-length patterns was done in
\cite{MR0322838}. Its proof uses important tools from dynamical systems and
number theory and is 
explained nowadays in terms of $S$-adic developments,
first return maps (Rauzy induction),
continued fraction expansion of real numbers,
and Ostrowski numeration system \cite{MR1970391}.

A generalization of the result of Morse and Hedlund was provided by Rauzy
for a single example \cite{MR667748}.
Based on the right-infinite sequence
often called the Tribonacci word
\[
T = 1213121121312121312112131213121121312121...
\]
which is fixed by $1 \mapsto 12,2 \mapsto 13,3 \mapsto 1$,
Rauzy 
proved that the system $(X_T, \sigma)$,
where $\sigma$ is the shift action,
is measurably conjugate to 
the toral rotation $(\torusI^2, x\mapsto x+(\beta^{-1},\beta^{-2}))$
where $\beta$ is the real root of $x^3-x^2-x-1$, the characteristic polynomial
of the incidence matrix of the substitution.
The coding of the toral translation is made through the partition into three
parts of a fundamental domain of $\torusI^2$ known as the Rauzy fractal \cite[\S
7.4]{MR1970385}.
Proving that this holds for all Pisot substitution is known as the Pisot
Conjecture \cite{MR3381478}, an important and still open question.

Finding further generalizations was coined the term of
\emph{Rauzy program} in \cite{MR2180244}, a survey divided into three parts:
the \emph{good}
coding of $k$-interval exchange transformations (IETs);
the \emph{bad}
coding of a rotation on $\torusI^k$;
and the \emph{ugly}
coding of two rotations on $\torusI^k$ for $k=1$.
The IETs are the good part since they behave well with induced transformations and 
admit continued fraction algorithms \cite{MR543205,MR644019,MR2261103,MR2299743}.
The bad part was much improved since then with
various recent results using multidimensional continued fraction algorithms
including Brun's algorithm \cite{MR0111735} which provides measurable-theoretic
conjugacy with symbolic systems for almost every toral rotations on $\torusI^2$
\cite{MR3986918,zbMATH07287525}.
As the authors wrote in \cite{MR2180244}, 
the term ugly ``\textit{refers to some esthetic
difficulties in building two-dimensional sequences by iteration of patterns}''.
Indeed, digital planes \cite{MR1888763,MR1906478,MR2180244,MR2577965,MR3124511}
are typical objects that are described by the coding of two rotations on
$\torusI^1$ and they are not built by rectangular shaped substitutions.
In this article, we want to add some shine on the ugly part by providing a
method based on induced transformations on $\torusI^2$ to prove that a toral
partition is a Markov partition for a toral $\Z^2$-rotation.





\subsection*{Markov partitions for automorphisms of the torus}

While Morse-Hedlund's theorem deals with the coding of irrational rotations,
other kinds of dynamical systems admit a symbolic representation.
Hyperbolic automorphisms of the torus are one example
\cite{MR1369092,MR1484730}.
Suppose that one starts at some position $v\in\R^2$ and moves according to the
successive images under the application of the map $v\mapsto Mv$ with
$M=\left(\begin{smallmatrix}
    1 & 1 \\
1 & 0\end{smallmatrix}\right)$ as shown in Figure~\ref{fig:automorphism}.
\begin{figure}[h]
\begin{center}
    \includegraphics[scale=.90]{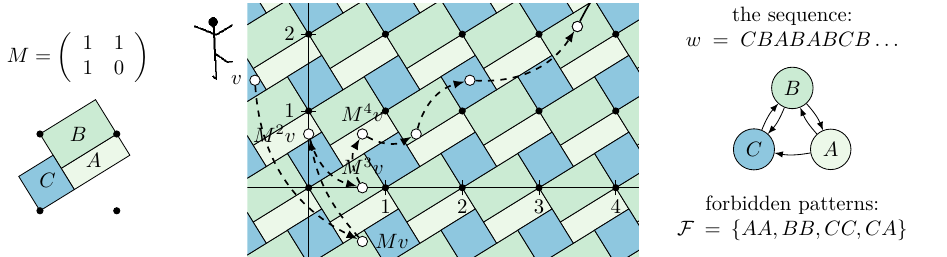}
\end{center}
    \caption[No table here]{
    The automorphism of $\R^2/\Z^2$
    defined as
    $v\mapsto Mv$ 
    admits a Markov partition.}
\label{fig:automorphism}
\end{figure}
The map $v\mapsto Mv$ is an automorphism of $\R^2/\Z^2$
which is \emph{hyperbolic} since $M$ has no eigenvalue of modulus~1.
It allows one to code the orbit $(M^kv)_{k\in\Z}$ as a sequence
in $\{A,B,C\}^\Z$ according to a well-chosen partition $\Pcal$
of a fundamental domain 
of $\R^2/\Z^2$
into three rectangles 
indexed by letters in the set $\{A,B,C\}$.
        In Figure~\ref{fig:automorphism}, the positive orbit $(M^kv)_{k\geq0}$ 
        of the starting point
    $v=\left(\frac{-7}{10},\frac{14}{10}\right)^T$
        is coded by
the sequence $CBABABCB \ldots$ which avoids the patterns
in $\Fcal=\{AA,BB,CC,CA\}$.
We denote the set of obtained sequences as $\Xcal_{\Pcal,M}$.
The partition of $\R^2/\Z^2$ is a Markov partition for the automorphism
because it has two important properties \cite[\S 6.5]{MR1369092}:
\begin{enumerate}
    \item[(C1)] every sequence in $\Xcal_{\Pcal,M}$ is obtained from a unique
        starting point in $\R^2/\Z^2$,
    \item[(C2')] the set $\Xcal_{\Pcal,M}$ is a shift of finite type (SFT),
        i.e., there exists a finite set $\Fcal$ of patterns such that
        $\Xcal_{\Pcal,M}$ is the set of sequences in $\{A,B,C\}^\Z$ which
        avoids the patterns in $\Fcal$.
\end{enumerate}
Such Markov Partitions exist for all hyperbolic automorphisms of the torus
\cite{MR0257315,MR0233038,MR0442989} and various kinds of diffeomorphisms \cite{MR277003},
see also \cite{MR1351521,MR1477538,MR1619562}.
Surprisingly, it turns out that Markov partitions also exist for toral $\Z^2$-rotations $R$
and $2$-dimensional subshifts $\Xcal_{\Pcal,R}$.

\subsection*{Results}
In this article, we propose a method for proving that
a 2-dimensional toral partition is a Markov partition
for a given toral $\Z^2$-rotation.
The method is inspired from the proof of Morse-Hedlund's theorem and its link with
continued fraction expansion and induced transformations.
We extend the notion of Rauzy induction of IETs
to the case of $\Z^2$-actions and we introduce the notion of induced partitions. We
apply the method on one example related to the golden mean and Jeandel-Rao aperiodic
Wang shift.

\begin{figure}[h]
\begin{center}
    \includegraphics[trim={2cm 0cm 0cm 0cm},clip]{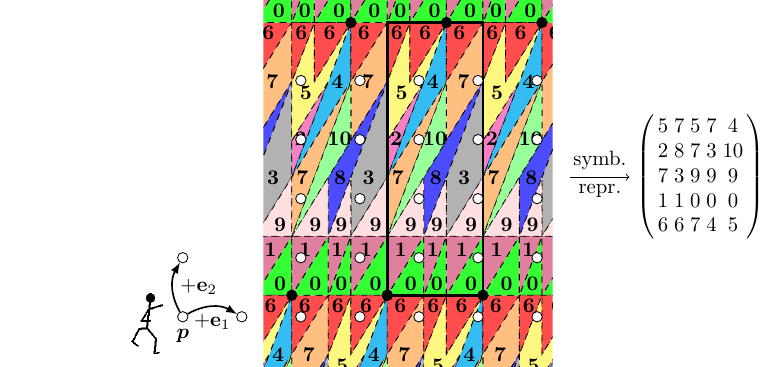}
\end{center}
    \caption{
    For every starting point $p\in\R^2$,
    the coding of the shifted lattice $p+\Z^2$
    under the polygon partition $\Pcal_0$
    yields a configuration which is a symbolic representation of $p$.
    We show that the set of such configurations is a shift of finite type (SFT)
    and hence that $\Pcal_0$ is a Markov partition for the toral
    $\Z^2$-rotation $R_0$.}
    \label{fig:2d-walk}
\end{figure}

Let 
\[
    \Gamma_0=\langle (\varphi,0), (1,\varphi+3) \rangle_\Z
\]
be a lattice
in $\R^2$
with $\varphi=\frac{1+\sqrt{5}}{2}$.
We consider the dynamical system defined by the
following toral $\Z^2$-rotation, i.e., 
a $\Z^2$-action on the $2$-dimensional torus $\R^2/\Gamma_0$:
\[
\begin{array}{rccl}
    R_0:&\Z^2\times\R^2/\Gamma_0 & \to & \R^2/\Gamma_0\\
    &(\bn,\bx) & \mapsto &\bx+\bn.
\end{array}
\]
A polygon partition $\Pcal_0$ of $\R^2/\Gamma_0$
indexed by integers from the set $\{0,1,2,\dots,10\}$
was introduced in \cite{labbe_markov_2021}, see
Figure~\ref{fig:2d-walk}.
The set $\Xcal_{\Pcal_0,R_0}\subset\Zrange{10}^{\Z^2}$ is the set of
$2$-dimensional configurations obtained by coding 
the orbits of points in $\R^2/\Gamma_0$ under the $\Z^2$-action $R_0$
by the atoms of the partition $\Pcal_0$.
It is a subshift and, in particular, it is closed
under the shift action $\sigma:\Z^2\times\Xcal_{\Pcal_0,R_0}\to\Xcal_{\Pcal_0,R_0}$
which is defined as $(\sigma^\bn(w))_\bk=w_{\bk+\bn}$
for every $\bn,\bk\in\Z^2$
and $w\in\Xcal_{\Pcal_0,R_0}$.
It was shown that $\Pcal_0$ gives a symbolic representation of $R_0$,
thus the 
$2$-dimensional subshift $\Xcal_{\Pcal_0,R_0}$ forms a
symbolic dynamical system 
$(\Xcal_{\Pcal_0,R_0},\Z^2,\sigma)$
that satisfies condition (C1).
It was also shown that $\Xcal_{\Pcal_0,R_0}$ is a
\textit{strict subset} of the Jeandel-Rao Wang shift $\Omega_0$ \cite{jeandel_aperiodic_2021},
but proving that $\Xcal_{\Pcal_0,R_0}$ is itself a SFT 
satisfying condition (C2') was left open.

\begin{theorem}\label{thm:XP0R0-markov-partition}
    $\Pcal_0$ is a Markov partition for the dynamical system $(\R^2/\Gamma_0, \Z^2, R_0)$.
\end{theorem}

It may sounds counter-intuitive for the reader since Markov partitions are usually
associated with hyperbolic systems and not with rotations. 
In this article, we use a more inclusive interpretation of the conditions (C1) and (C2')
by considering higher dimensional subshifts of finite type
as it was done already in \cite{MR1632169}.
See the definition and a discussion in Section~\ref{sec:Markov-partition}.

The fact that
$\Xcal_{\Pcal_0,R_0}\subset\Omega_0$ 
corresponds to the easy direction in the proof of Morse-Hedlund's
theorem, namely that codings of irrational rotations have pattern complexity $n+1$.
Proving the converse, i.e., that a configuration in the Jeandel-Rao Wang shift
is obtained as the coding of a shifted lattice $p+\Z^2$ for some 
$p\in\R^2/\Gamma_0$ is harder, for
the same reason that proving that a Sturmian sequence is a coding of an
irrational rotation from some starting point is more involved.

The proof of Theorem~\ref{thm:XP0R0-markov-partition}
is inspired from the proof of Morse-Hedlund's theorem and Rauzy induction of
irrational rotations \cite{MR1970391}.
It builds on top of
results proved previously on Jeandel-Rao aperiodic Wang shift
and on the
substitutive structure of $\Xcal_{\Pcal_0,R_0}$ which is computed
herein from the induction of $\Z^2$-actions and of toral partitions.
The substitutive structure of $\Xcal_{\Pcal_0,R_0}$ is described in the next result.

In the statement,
we use the following notation
\begin{equation*}
    {\overline{X}}^{\sigma} 
    = \bigcup_{\bk\in\Z^2}\sigma^\bk X
    = \bigcup_{\bk\in\Z^2}\{\sigma^\bk(x) \mid x\in X\}
\end{equation*}
for the closure of a set $X\subset\A^{\Z^2}$ under the shift $\sigma$.
Moreover, the substitutive structure of subshifts is given in terms of $2$-dimensional morphisms.
The definition of $2$-dimensional morphism can be found in 
Section~\ref{sec:d-dim-morphism}.
Let us define three notions which appear in the statement
and whose formal definitions can be found in
Section~\ref{sec:self-similar-subshifts}
on self-similar subshifts.
We say that a $2$-dimensional morphism $\omega:\A\to\A^{*^2}$ is
expansive if for every $a\in\A$ the $2$-dimensional word $\omega^m(a)$ gets
arbitrarily large in width and in height as $m\in\N$ goes to infinity.
A subshift $X\subseteq\A^{\Z^2}$ 
is self-similar
if there exists an expansive
$2$-dimensional morphism $\omega:\A\to\A^{*^2}$ such that
$X=\overline{\omega(X)}^\sigma$.
Finally and as in the one-dimensional case, we say that $\omega$ is primitive
if there exists $m\in\N$ such that
for every $a,b\in\A$ the letter $b$ occurs in $\omega^m(a)$.

\begin{theorem}\label{thm:XP0R0-subst-structure}
    Let $\Xcal_{\Pcal_0,R_0}$ be the symbolic dynamical system associated
    to $\Pcal_0$, $R_0$.
    There exist
    lattices $\Gamma_{i}\subset\R^2$,
    alphabets $\A_{i}$,
    $\Z^2$-actions $R_{i}:\Z^2\times \R^2/\Gamma_{i}\to\R^2/\Gamma_{i}$
    and topological partitions $\Pcal_{i}$ of $\R^2/\Gamma_{i}$
    indexed by letters from the alphabet $\A_i$
    that provide the substitutive structure of $\Xcal_{\Pcal_0,R_0}$.
    More precisely,
\begin{enumerate}[\rm (i)]
    \item
    There exists a $2$-dimensional morphism $\beta_0:\A_{1}\to\A_0^{*^2}$
    \begin{equation*}
        \Xcal_{\Pcal_0,R_0} \xleftarrow{\beta_0}
        \Xcal_{\Pcal_1,R_1}
    \end{equation*}
    that is onto up to a shift, i.e.,
    $\Xcal_{\Pcal_0,R_0}=\overline{\beta_0(\Xcal_{\Pcal_1,R_1})}^\sigma$.
\item 
    There exists a shear conjugacy 
    \begin{equation*}
        \Xcal_{\Pcal_1,R_1} \xleftarrow{\beta_1}
        \Xcal_{\Pcal_2,R_2}
    \end{equation*}
    shearing configurations by the action of the matrix
$M=\left(\begin{smallmatrix}
        1 & 1 \\
        0 & 1
\end{smallmatrix}\right)$, i.e.,
    satisfying
    $\sigma^{M\bk}\circ\beta_1=\beta_1\circ\sigma^\bk$
    for every $\bk\in\Z^2$.

    \item
    There exist $2$-dimensional morphisms
    $\beta_2$, $\beta_3$, $\beta_4$, $\beta_5$, $\beta_6$ and $\beta_7$:
    \begin{equation*}
    \Xcal_{\Pcal_2,R_2} \xleftarrow{\beta_2}
    \Xcal_{\Pcal_3,R_3} \xleftarrow{\beta_3}
    \Xcal_{\Pcal_4,R_4} \xleftarrow{\beta_4}
    \Xcal_{\Pcal_5,R_5} \xleftarrow{\beta_5}
    \Xcal_{\Pcal_6,R_6} \xleftarrow{\beta_6}
    \Xcal_{\Pcal_7,R_7} \xleftarrow{\beta_7}
    \Xcal_{\Pcal_8,R_8} 
    \end{equation*}
    that are onto up to a shift, i.e.,
    $\Xcal_{\Pcal_i,R_i}=\overline{\beta_{i}(\Xcal_{\Pcal_{i+1},R_{i+1}})}^\sigma$
    for each $i\in\{2,3,4,5,6,7\}$.
\item 
    The subshift $\Xcal_{\Pcal_8,R_8}$ is self-similar
    satisfying
        $\Xcal_{\Pcal_8,R_8}=\overline{\beta_8\beta_9\tau(\Xcal_{\Pcal_8,R_8})}^\sigma$.
        More precisely,
    there exist two $2$-dimensional morphisms $\beta_8$, $\beta_9$
    and a bijection $\tau:\A_{8}\to \A_{10}$
    \begin{equation*}
    \Xcal_{\Pcal_8,R_8} \xleftarrow{\beta_8}
    \Xcal_{\Pcal_9,R_9} \xleftarrow{\beta_9}
    \Xcal_{\Pcal_{10},R_{10}}\xleftarrow{\tau}
    \Xcal_{\Pcal_{8},R_{8}}
    \end{equation*}
    that are onto up to a shift, i.e.,
    $\Xcal_{\Pcal_8,R_8}=\overline{\beta_{8}(\Xcal_{\Pcal_{9},R_{9}})}^\sigma$,
    $\Xcal_{\Pcal_9,R_9}=\overline{\beta_{9}(\Xcal_{\Pcal_{10},R_{10}})}^\sigma$ and
    $\Xcal_{\Pcal_{10},R_{10}}=\tau(\Xcal_{\Pcal_{8},R_{8}})$
    and the product $\beta_8\beta_9\tau$
    is an expansive and primitive self-similarity.
\item 
    The subshift $\Xcal_{\Pcal_8,R_8}$ is topologically conjugate to the
        subshift $\Xcal_{\Pcal_\U,R_\U}$ introduced in \cite{labbe_markov_2021}
        as there exists a bijection $\zeta:\U\to\A_{8}$ such that
        $\zeta(\Xcal_{\Pcal_\U,R_\U})=\Xcal_{\Pcal_{8},R_{8}}$.
\end{enumerate}
\end{theorem}

Theorem~\ref{thm:XP0R0-subst-structure} must be compared with
the main result of \cite{MR4226493},
recalled herein as Theorem~\ref{thm:article2},
    giving the substitutive structure
    of a minimal subshift $X_0$ of the Jeandel-Rao Wang shift $\Omega_0$.
That result proved the existence
of sets of Wang tiles $\{\T_{i}\}_{1\leq i\leq 12}$
together with their associated Wang shifts $\{\Omega_{i}\}_{1\leq i\leq 12}$
and $2$-dimensional morphisms $\omega_i:\Omega_{i+1}\to\Omega_i$
that provide the substitutive structure of Jeandel-Rao Wang shift,
see Figure~\ref{fig:common-subst-struct}. 
In fact, the consequence of the two theorems is that 
the subshifts $X_0$ and $\Xcal_{\Pcal_0,R_0}$ have
the \textit{exact same} substitutive structure given as the inverse limit of the same
eventually periodic sequence of $2$-dimensional morphisms.

\begin{figure}[h]
\begin{center}
\begin{tikzpicture}[auto,xscale=1.30,every node/.style={scale=1.0}]
    \node (0) at (0,2) {$\Omega_0$};
    \node (1) at (1,2) {$\Omega_1$};
    \node (2) at (2,2) {$\Omega_2$};
    \node (3) at (3,2) {$\Omega_3$};
    \node (4) at (4,2) {$\Omega_4$};
    \node (X0) at (0,1) {$X_0$};
    \node (X1) at (1,1) {$X_1$};
    \node (X2) at (2,1) {$X_2$};
    \node (X3) at (3,1) {$X_3$};
    \node (X4) at (4,1) {$X_4$};
    \node (5) at (5,1) {$\Omega_5$};
    \node (6) at (6,1) {$\Omega_6$};
    \node (7) at (7,1) {$\Omega_7$};
    \node (12) at (10,1) {$\Omega_{12}$};
    \node (13) at (11.5,1) {$\Omega_\U$};
    \node (chi0) at (0,0) {$\Xcal_{\Pcal_0,R_0}$};
    \node (chi1) at (4,0) {$\Xcal_{\Pcal_1,R_1}$};
    \node (chi2) at (5.5,0) {$\Xcal_{\Pcal_2,R_2}$};
    \node (chi3) at (7,0) {$\Xcal_{\Pcal_3,R_3}$};
    \node (chi8) at (10,0) {$\Xcal_{\Pcal_8,R_8}$};
    \node (chiU) at (11.5,0) {$\Xcal_{\Pcal_\U,R_\U}$};
    \draw (X0) edge[draw=none] node[auto=false,sloped] {$\subset$} (0); 
    \draw (X1) edge[draw=none] node[auto=false,sloped] {$\supset$} (1);
    \draw (X2) edge[draw=none] node[auto=false,sloped] {$\supset$} (2);
    \draw (X3) edge[draw=none] node[auto=false,sloped] {$\supset$} (3);
    \draw (X4) edge[draw=none] node[auto=false,sloped] {$\supset$} (4);
    \draw (chi0) edge[draw=none] node[auto=false,sloped] {$=$} (X0);
    \draw (chi1) edge[draw=none] node[auto=false,sloped] {$=$} (X4);
    \draw (chi3) edge[draw=none] node[auto=false,sloped] {$=$} (7);
    \draw (chi8) edge[draw=none] node[auto=false,sloped] {$=$} (12);
    \draw (chiU) edge[draw=none] node[auto=false,sloped] {$=$} (13);
    \draw[to-,very thick] (chi0) to node {$\beta_0$} (chi1);
    \draw[to-,very thick] (chi1) to node {$\beta_1$} (chi2);
    \draw[to-,very thick] (chi2) to node {$\beta_2$} (chi3);
    \draw[to-,very thick] (X0) to node {$\omega_0$} (X1);
    \draw[to-,very thick] (X1) to node {$\omega_1$} (X2);
    \draw[to-,very thick] (X2) to node {$\omega_2$} (X3);
    \draw[to-,very thick] (X3) to node {$\omega_3$} (X4);
    \draw[to-,very thick] (0) to node {$\omega_0$} (1);
    \draw[to-,very thick] (1) to node {$\omega_1$} (2);
    \draw[to-,very thick] (2) to node {$\omega_2$} (3);
    \draw[to-,very thick] (3) to node {$\omega_3$} (4);
    \draw[to-,very thick] (X4) to node {$\jmath$} (5);
    \draw[to-,very thick] (5) to node {$\eta$} (6);
    \draw[to-,very thick] (6) to node {$\omega_6$} (7);
    \draw[to-,very thick] (7) to node {$\omega_7\omega_8\omega_9\omega_{10}\omega_{11}$} (12);
    \draw[to-,very thick] (chi3) to node {$\beta_3\beta_4\beta_5\beta_{6}\beta_{7}$} (chi8);
    \draw[to-,very thick] (chi8) to node {$\zeta$} (chiU);
    \draw[to-,very thick] (chiU) to [loop below] node {$\zeta^{-1}\beta_8\beta_9\tau\zeta$} (chiU);
    \draw[to-,very thick] (chi8) to [loop below] node {$\beta_8\beta_9\tau$} (chi8);
    \draw[to-,very thick] (12) to node {$\rho$} (13);
    \draw[to-,very thick] (12) to [loop above] node {$\rho\,\omega_\U\rho^{-1}$} (12);
    \draw[to-,very thick] (13) to [loop above] node {$\omega_\U$} (13);
\end{tikzpicture}
\end{center}
\caption{We prove that the subshifts $X_0\subset\Omega_0$ and $\Xcal_{\Pcal_0,R_0}$ are
equal since they have a 
common substitutive structure. 
The substitutive structure of $X_0$ computed in \cite{MR4226493}
and the substitutive structure of $\Xcal_{\Pcal_0,R_0}$ 
satisfy $\beta_0=\omega_0\omega_1\omega_2\omega_3$,
    $\beta_1\beta_2=\jmath\,\eta\,\omega_6$,
    $\beta_3 \beta_4 \beta_5 \beta_6 \beta_7=
     \omega_7 \omega_8 \omega_9 \omega_{10} \omega_{11}$,
     $\zeta=\rho$
     and $\beta_8\,\beta_9\,\tau=\rho\,\omega_\U\,\rho^{-1}$.
     We deduce that  
     $\Xcal_{\Pcal_8,R_8}=\Omega_{12}$,
     $\Xcal_{\Pcal_3,R_3}=\Omega_{7}$,
     $\Xcal_{\Pcal_1,R_1}=X_{4}$ and finally
     $\Xcal_{\Pcal_0,R_0}=X_{0}$.}
\label{fig:common-subst-struct}
\end{figure}

\begin{theorem}\label{thm:substitutionequality}
    The symbolic dynamical system $\Xcal_{\Pcal_0,R_0}$
    and the minimal subshift $X_0\subset\Omega_0$ 
    of the Jeandel-Rao Wang shift
    have the same substitutive structure
    in the sense that the following equalities hold:
    \[
        \setlength{\arraycolsep}{4mm}
    \begin{array}{lllll}
        \beta_0=\omega_0\,\omega_1\,\omega_2\,\omega_3,
        &\beta_1\beta_2=\jmath\,\eta\,\omega_6,
        &\beta_3=\omega_7 ,
        &\beta_4=\omega_8,\\
        \beta_5=\omega_9,
        &\beta_6=\omega_{10},
        &\beta_7=\omega_{11},
        &\zeta=\rho,
        &\beta_8\,\beta_9\,\tau=\rho\,\omega_\U\,\rho^{-1},
    \end{array}
    \]
where $\omega_0$, $\omega_1$, $\omega_2$, $\omega_3$, $\jmath$, $\eta$,
    $\omega_6$, $\omega_7$, $\omega_8$, $\omega_9$, $\omega_{10}$, $\omega_{11}$,
    $\rho$ were computed in \cite{MR4226493}
    and $\omega_\U$ was first defined in \cite{MR3978536}.
\end{theorem}

\begin{remark}\label{rem:military-ordering}
    To obtain equalities between substitutions computed from totally
    different objects,
    we use a common convention
    for the definition of the $\beta_i$ from the induction of toral partitions,
    and in \cite{MR4226493} for the 
    construction of the $\omega_i$ from sets of Wang tiles.
    When constructing the substitutions, 
    short images of letter come before longer ones, 
    and words of the same length are sorted lexicographically.
    See Definition~\ref{def:definition-radix-order}.
\end{remark}

The description of
the symbolic dynamical system $\Xcal_{\Pcal_0,R_0}$
and the minimal subshift $X_0$ of Jeandel-Rao aperiodic subshift $\Omega_0$
by their substitutive structure
allows to prove their topological conjugacy.

\begin{corollary}\label{cor:equality-X0=XP0R0}
    The subshifts $\Xcal_{\Pcal_\U,R_\U}$ and $\Omega_\U$ are topologically conjugate
    and are equal to the minimal aperiodic substitutive subshift $\X_{\omega_\U}$.
    The subshifts $\Xcal_{\Pcal_8,R_8}$ and $\Omega_{12}$ are topologically conjugate.
    The subshifts $\Xcal_{\Pcal_0,R_0}$ and $X_0$ are topologically conjugate,
    and the same holds for intermediate subshifts:
\begin{equation*}
    \Xcal_{\Pcal_1,R_1}=\jmath(\Omega_{5}),\;
    \Xcal_{\Pcal_3,R_3}=\Omega_{7},\;
    \Xcal_{\Pcal_4,R_4}=\Omega_{8},\;
    \Xcal_{\Pcal_5,R_5}=\Omega_{9},\;
    \Xcal_{\Pcal_6,R_6}=\Omega_{10}\text{ and }
    \Xcal_{\Pcal_7,R_7}=\Omega_{11}.
\end{equation*}
\end{corollary}

The fact that $X_0$ and $\Xcal_{\Pcal_0,R_0}$ are equal implies 
Theorem~\ref{thm:XP0R0-markov-partition} since
it was proved in \cite{MR4226493}
that $X_0$ is a shift of finite type.
It also implies the following corollary which can be seen as a generalization
of what happens for Sturmian sequences.

\begin{corollary}\label{cor:isomorphic}
    The dynamical system $(X_0,\Z^2,\sigma)$ is strictly
    ergodic and
    the measure-preserving dynamical system $(X_0,\Z^2,\sigma,\nu)$
    is isomorphic 
    to the toral $\Z^2$-rotation $(\R^2/\Gamma_0,\Z^2,R_0,\lambda)$ 
    where $\nu$ is the unique shift-invariant probability measure on
    $X_0$ and $\lambda$ is the Haar measure on $\R^2/\Gamma_0$.
\end{corollary}

Although Jeandel-Rao Wang shift $\Omega_0$ is not minimal as it contains the
proper minimal subshift $\Xcal_{\Pcal_0,R_0}=X_0$, we believe that it is
uniquely ergodic.

\begin{conjecture}\label{conjecture:uniquely-erogodic}
    The Jeandel-Rao subshift $\Omega_0$ is uniquely ergodic.
\end{conjecture}

That conjecture is equivalent to prove that $\Omega_0\setminus X_0$ has
measure 0 for any shift-invariant probability measure on $\Omega_0$
which was stated as a conjecture in \cite{MR4226493}
and where some progress was done.
That would also imply the existence 
    of an isomorphism
    of measure-preserving dynamical systems
    between Jeandel-Rao Wang shift $(\Omega_0,\Z^2,\sigma,\nu)$ and
    $(\R^2/\Gamma_0,\Z^2,R_0,\lambda)$
    where $\nu$ would be the unique shift-invariant probability measure on
    $\Omega_0$ and $\lambda$ is the Haar measure on $\R^2/\Gamma_0$.


This calls for a general theory of $d$-dimensional subshifts of finite type
coded by Markov partitions of the $d$-dimensional torus and admitting induced
subsystems.

\subsection*{Structure of the article}

The article is divided into three parts.
Part~\ref{part:preliminaries} 
gathers the preliminary notions including
two-dimensional words and morphisms,
toral $\Z^2$-rotations and
polyhedron exchange transformations (PETs).
Part~\ref{part:induction-z2-actions}
defines
the induction of toral $\Z^2$-rotations,
induced toral partitions and renormalization schemes.
In Part~\ref{part:subst-struct-XP0R0},
we induce the partition $\Pcal_0$ introduced
in \cite{labbe_markov_2021} to get a sequence of substitutions
until we reach a self-induced partition. 
We prove that the substitutive structure obtained from the induction procedure
is the same as the substitutive structure computed directly from the
Jeandel-Rao Wang tiles in \cite{MR4226493}.
Some remarks about further research can be found in the conclusion.

\subsection*{Batteries included}
Two algorithms are given in Section~\ref{sec:algorithms}
at the end of Part~\ref{part:induction-z2-actions}
to compute induced partitions and induced PETs.
They are implemented in the optional package \texttt{slabbe}
\cite{labbe_slabbe_0_6_2_2020} of SageMath \cite{sagemathv9.2}.
The code to reproduce the proof of
Theorem~\ref{thm:XP0R0-subst-structure}
is included directly in the proofs of intermediate results stated in
Part~\ref{part:subst-struct-XP0R0}.
The code present in this article is also contained in the file
\texttt{demos/arXiv\_1906\_01104.rst} of the \texttt{slabbe} package.
It allows to make sure it remains reproducible in the future with new
versions of SageMath with the command
\texttt{sage -t demos/arXiv\_1906\_01104.rst}
which should return something like
\texttt{[87 tests, 8.19 s]} and \texttt{All tests passed!}.

\subsection*{Acknowledgments}

I am thankful to Vincent Delecroix for many helpful discussions at LaBRI in
Bordeaux during the preparation of this article which allowed me to improve my
knowledge on dynamical systems and Rauzy induction of interval exchange
transformations.
I also want to thank the referee for her/his valuable comments
which led to many improvements in the article.

\part{Preliminaries}\label{part:preliminaries}

This part is divided into four sections on
dynamical systems, subshifts and Wang shift;
two-dimensional words and morphisms;
toral $\Z^2$-rotations and polygonal exchange transformations;
symbolic representation of toral $\Z^2$-rotations.

\section{Dynamical systems and subshifts}
\label{sec:preliminaries}


In this section, we introduce dynamical systems,
subshifts and shifts of finite type.
We let $\Z=\{\dots,-1,0,1,2,\dots\}$ denote the integers and
$\N=\{0,1,2,\dots\}$ be the nonnegative integers.

\subsection{Topological dynamical systems}


Most of the notions introduced here can be found in \cite{MR648108}.
A \defn{dynamical system} is
a triple $(X,G,T)$, where $X$ is a topological space, $G$ is a topological
group and $T$ is a continuous function $G\times X\to X$ defining a left action
of $G$ on $X$:
if $x\in X$, $e$ is the identity element of $G$ and $g,h\in G$, then using
additive notation for the operation in $G$ we have $T(e,x)=x$
and $T(g+h,x)=T(g,T(h,x))$.
In other words, if one denotes the transformation $x\mapsto T(g,x)$
by $T^g$, then $T^{g+h}=T^g T^h$.
In this work, we consider the Abelian group $G=\Z\times\Z$.

If $Y\subset X$, let $\overline{Y}$ denote the topological closure of $Y$ and
let $T(Y):=\cup_{g\in G}T^g(Y)$ denote the $T$-closure of $Y$.
Alternatively, we also use the notation $\overline{Y}^T=T(Y)$ to denote the $T$-closure of $Y$.
A subset $Y\subset X$ is \defn{$T$-invariant} if $T(Y)=Y$.
A dynamical system $(X,G,T)$ is called \defn{minimal} if $X$ does
not contain any nonempty, proper, closed $T$-invariant subset.
The left action of $G$ on $X$ is \defn{free}
if $g=e$ whenever there exists $x\in X$ such that $T^g(x)=x$.


Let $(X,G,T)$ and $(Y,G,S)$ be two dynamical systems with
the same topological group $G$.
A \defn{homomorphism} $\theta:(X,G,T)\to(Y,G,S)$ is a continuous
function $\theta:X\to Y$ satisfying the commuting property
that $S^g\circ\theta=\theta\circ T^g$ for every $g\in G$.
A homomorphism $\theta:(X,G,T)\to(Y,G,S)$ is called an \defn{embedding}
if it is one-to-one, a \defn{factor map} if it is onto, and a \defn{topological
conjugacy} if it is both one-to-one and onto and its inverse map is continuous.
If $\theta:(X,G,T)\to(Y,G,S)$ is a factor map,
then $(Y,G,S)$ is called a \defn{factor} of $(X,G,T)$
and $(X,G,T)$ is called an \defn{extension} of $(Y,G,S)$.
Two subshifts are \defn{topologically conjugate} if there is a topological
conjugacy between them.


Let $(X,G,T)$ and $(Y,G,S)$ be two dynamical systems with
the same topological group $G=\Z^d$.
We say that a \defn{$\GL_d(\Z)$-homomorphism} 
$\theta:(X,\Z^d,T)\to(Y,\Z^d,S)$ is a continuous function
$\theta:X\to Y$ equipped with some matrix $M\in \GL_d(\Z)$
satisfying the commuting property
that 
$S^{M\bk}\circ\theta=\theta\circ T^\bk$ 
for every $\bk\in\Z^d$.
A $\GL_d(\Z)$-homomorphism is called a \defn{$\GL_d(\Z)$-conjugacy} if it is
both one-to-one and onto and its inverse map is continuous.
A $\GL_d(\Z)$-conjugacy is called a \defn{shear conjugacy}
if the matrix $M\in\GL_d(\Z)$ is a shear matrix.
The notion of $\GL_d(\Z)$-conjugacy corresponds to \defn{flip-conjugacy} when
$d=1$ 
\cite{MR1613140}
and to \defn{extended symmetry} when $X=Y$
\cite{MR3721878}.


A \defn{measure-preserving dynamical system} is defined as a system
$(X,G,T,\mu,\Bcal)$, where $\mu$ is a probability measure defined on
the Borel $\sigma$-algebra $\Bcal$ of subsets of $X$,
and $T^g:X\to X$ is a measurable map
which preserves the measure $\mu$ for all $g\in G$, that is,
$\mu(T^g(B))=\mu(B)$ for all $B\in\Bcal$. The measure $\mu$ is said to be
\defn{$T$-invariant}.
In what follows, 
when it is clear from the context,
we omit the Borel $\sigma$-algebra $\Bcal$ of subsets of $X$ and write $(X,G,T,\mu)$ 
to denote a measure-preserving dynamical system.

The set of all $T$-invariant probability measures of a dynamical
system $(X,G,T)$ is denoted by $\mathcal{M}^T(X)$.
A $T$-invariant probability measure on $X$ is called \defn{ergodic} if for every set
$B\in\Bcal$ such that $T^{g}(B)=B$ for all $g\in G$, we have that $B$ has either
zero or full measure. A
dynamical system $(X,G,T)$ is \defn{uniquely ergodic}
if it has only one invariant probability measure, i.e., $|\mathcal{M}^T(X)|=1$.
A dynamical system $(X,G,T)$ is said \defn{strictly ergodic}
if it is uniquely ergodic and minimal.

Let $(X,G,T,\mu,\Bcal)$
and $(X',G,T',\mu',\Bcal')$ be two measure-preserving dynamical systems.
We say that the two systems are
\defn{isomorphic} 
if there exist measurable sets $X_0\subset X$ and $X'_0\subset X'$
of full measure (i.e., $\mu(X\setminus X_0)=0$
and $\mu'(X'\setminus X'_0)=0$) with
$T^g(X_0)\subset X_0$, $T'^g(X'_0)\subset X'_0$ for all $g\in G$
and there exists a map $\phi:X_0\to X'_0$, called an \defn{isomorphism},
that is one-to-one and onto and such that for all $A\in\Bcal'(X'_0)$,
\begin{itemize}
    \item $\phi^{-1}(A)\in\Bcal(X_0)$,
    \item $\mu(\phi^{-1}(A))=\mu'(A)$, and
    \item $\phi\circ T^g(x)=T'^g\circ\phi(x)$ for all $x\in X_0$ and $g\in G$.
\end{itemize}
The role of the set $X_0$ is to make precise the fact that the properties of
the isomorphism need to hold only on a set of full measure.

\subsection{Subshifts and shifts of finite type}


In this section, we introduce multidimensional subshifts,
a particular type of dynamical systems 
\cite[\S 13.10]{MR1369092},
\cite{MR1861953,MR2078846,MR3525488}.
Let $\A$ be a finite set, $d\geq 1$, and let $\A^{\Z^d}$ be the set of all maps
$x:\Z^d\to\A$, equipped with the compact product topology. 
An element $x\in\A^{\Z^d}$ is called \defn{configuration}
and we write it as $x=(x_\bm)=(x_\bm:\bm\in\Z^d)$,
where $x_\bm\in\A$ denotes the value of $x$ at $\bm$. 
The topology on $\A^{\Z^d}$ is compatible with the metric defined for all
configurations $x,x'\in\A^{\Z^d}$ by $\dist(x,x')=2^{-\min\left\{\Vert\bn\Vert\,:\,
x_\bn\neq x'_\bn\right\}}$
where $\Vert\bn\Vert = |n_1| + \dots + |n_d|$.
The \defn{shift action} $\sigma:\bn\mapsto
\sigma^\bn$ of $\Z^d$ on $\A^{\Z^d}$ is defined by
\begin{equation}\label{eq:shift-action}
    (\sigma^\bn(x))_\bm = x_{\bm+\bn}
\end{equation}
for every $x=(x_\bm)\in\A^{\Z^d}$ and $\bn\in\Z^d$. 
If $X\subset \A^{\Z^d}$,
let $\overline{X}$ denote the topological closure of $X$
and let $\overline{X}^\sigma:=\{\sigma^\bn(x)\mid x\in X, \bn\in\Z^d\}$
denote the shift-closure of $X$.
A subset $X\subset
\A^{\Z^d}$ is \defn{shift-invariant} if 
$\overline{X}^\sigma=X$. A closed, shift-invariant subset
$X\subset\A^{\Z^d}$ is a \defn{subshift}. 
If $X\subset\A^{\Z^d}$ is a subshift we write
$\sigma=\sigma^X$ for the restriction of the shift action
\eqref{eq:shift-action} to $X$. 
When $X$ is a subshift,
the triple $(X,\Z^d,\sigma)$ is a dynamical system
and the notions presented in the previous section hold.

A configuration $x\in X$ is \defn{periodic} if there is a nonzero vector
$\bn\in\Z^d\setminus\{\zero\}$ such that $x=\sigma^\bn(x)$
and otherwise it is said \defn{nonperiodic}.
We say that a nonempty subshift $X$ is \defn{aperiodic}
if the shift action $\sigma$ on $X$ is free.

For any subset $S\subset\Z^d$ let $\pi_S:\A^{\Z^d}\to\A^S$ denote the
projection map which restricts every $x\in\A^{\Z^d}$ to $S$. 
A \defn{pattern} is a function $p\in\A^S$ for some finite subset
$S\subset\Z^d$.
To every pattern $p\in\A^S$ corresponds
a subset $\pi_S^{-1}(p)\subset\A^{\Z^d}$ called \defn{cylinder}.
A nonempty set $X\subset\A^{\Z^d}$ is a
\defn{subshift} if and only if there exists a set $\Fcal$
of \defn{forbidden} patterns such that
\begin{equation}\label{eq:SFT}
    X = \{x\in\A^{\Z^d} \mid \pi_S\circ\sigma^\bn(x)\notin\Fcal
    \text{ for every } \bn\in\Z^d \text{ and } S\subset\Z^d\},
\end{equation}
see \cite[Prop.~9.2.4]{MR3525488}.
A subshift $X\subset\A^{\Z^d}$ is a 
\defn{shift of finite type} (SFT) if there exists a finite set $\Fcal$ such that \eqref{eq:SFT} holds.
In this article, we consider shifts of finite type on $\Z\times\Z$, that is, the case
$d=2$.

\section{Two-dimensional words and morphisms}

In this section, we reuse the notations from
\cite{MR3978536} and \cite{MR4226493}
on $d$-dimensional words, morphisms and self-similar subshifts.

\subsection{$d$-dimensional word}

In this section, we recall the definition of $d$-dimensional word that appeared
in \cite{MR2579856} and we keep the notation $u\odot^i v$ they proposed
for the concatenation. 

We denote by
$\{\be_k|1\leq k\leq d\}$ the canonical
basis of $\Z^d$ where $d\geq1$ is an integer.
If $i\leq j$ are integers, then $\llbracket i, j\rrbracket$ denotes the
interval of integers $\{i, i+1, \dots, j\}$.
Let $\bn=(n_1,\dots,n_d)\in\N^d$ and $\A$ be an alphabet.
We denote by $\A^{\bn}$ the set of functions
\begin{equation*}
    u:
    \ZZrange{0}{n_1-1}
\times
\cdots
\times
    \ZZrange{0}{n_d-1}
\to\A.
\end{equation*}
An element $u\in\A^\bn$ is called a
\defn{$d$-dimensional word $u$ of shape $\bn=(n_1,\dots,n_d)\in\N^d$}
on the alphabet~$\A$.
We use the notation $\shape(u)=\bn$ when necessary.
The set of all finite $d$-dimensional words is 
$\A^{*^d}=\{\A^\bn\mid\bn\in\N^d\}$.
A $d$-dimensional word of shape $\be_k+\sum_{i=1}^d\be_i$ is called a
\defn{domino in the direction $\be_k$}.
When the context is clear, we write $\A$ instead of $\A^{(1,\dots,1)}$.
When $d=2$, we represent a $d$-dimensional word $u$ of shape $(n_1,n_2)$ as a
matrix with Cartesian coordinates:
\begin{equation*}
    u=
    \left(\begin{array}{ccc}
        u_{0,n_2-1} &\dots   & u_{n_1-1,n_2-1} \\
        \dots   &\dots   & \dots \\
        u_{0,0} &\dots   & u_{n_1-1,0}
    \end{array}\right).
\end{equation*}
Let $\bn,\bm\in\N^d$ and $u\in\A^\bn$ and $v\in\A^\bm$.
If there exists an index $i$ such that the
shapes $\bn$ and $\bm$ are equal except maybe at index $i$,
then the \defn{concatenation of $u$ and $v$ in the direction $\be_i$} 
is \defn{defined}: it is
the 
$d$-dimensional word $u\odot^i v$ of shape $(n_1,\dots,n_{i-1},n_i+m_i,n_{i+1},\dots,n_d)\in\N^d$
given as
\begin{equation*}
    (u\odot^i v) (\ba) = 
\begin{cases}
    u(\ba)          & \text{if}\quad 0 \leq a_i < n_i,\\
    v(\ba-n_i\be_i) & \text{if}\quad n_i \leq a_i < n_i+m_i.
\end{cases}
\end{equation*}
If the shapes $\bn$ and $\bm$ are not equal except at
index $i$, we say that the concatenation of $u\in\A^\bn$ and $v\in\A^\bm$ in
the direction $\be_i$ is \defn{not defined}.
The following equation illustrates the concatenation of words 
in the direction $\be_2$ when $d=2$:
\[
    \arraycolsep=2.5pt
\left(\begin{array}{ccccc}
4 & 5 \\
10 & 5
\end{array}\right)
\odot^2
\left(\begin{array}{ccccc}
3 & 10 \\
9 & 9 \\
0 & 0 \\
\end{array}\right)
    =
\left(\begin{array}{ccccc}
3 & 10 \\
9 & 9 \\
0 & 0 \\
4 & 5 \\
10 & 5
\end{array}\right)
\]
in the direction $\be_1$ when $d=2$:
\[
    \arraycolsep=2.5pt
\left(\begin{array}{ccccc}
2 & 8 & 7  \\
7 & 3 & 9 \\
1 & 1 & 0 \\
6 & 6 & 7 \\
7 & 4 & 3 
\end{array}\right)
\odot^1
\left(\begin{array}{ccccc}
3 & 10 \\
9 & 9 \\
0 & 0 \\
4 & 5 \\
10 & 5
\end{array}\right)
    =
\left(\begin{array}{ccccc}
2 & 8 & 7 & 3 & 10 \\
7 & 3 & 9 & 9 & 9 \\
1 & 1 & 0 & 0 & 0 \\
6 & 6 & 7 & 4 & 5 \\
7 & 4 & 3 & 10 & 5
\end{array}\right).
\]

Let $\bn,\bm\in\N^d$ and $u\in\A^\bn$ and $v\in\A^\bm$.
We say that $u$ \defn{occurs in $v$ at position $\bp\in\N^d$} if
$v$ is large enough, i.e., $\bm-\bp-\bn\in\N^d$ and
\[
    v(\ba+\bp) = u(\ba)
\]
for all $\ba=(a_1,\dots,a_d)\in\N^d$ such that 
$0\leq a_i<n_i$ with $1\leq i\leq d$.
If $u$ occurs in $v$ at some position, then we say that $u$ is a
$d$-dimensional \defn{subword} or \defn{factor} of $v$.

\subsection{$d$-dimensional language}

A subset $L\subseteq\A^{*^d}$ is called a $d$-dimensional \defn{language}. The
\defn{factorial closure} of a language $L$ is
\begin{equation*}
    \overline{L}^{Fact}
    = \{u\in\A^{*^d} \mid u\text{ is a $d$-dimensional subword of some } 
                           v\in L\}.
\end{equation*}
A language $L$ is \defn{factorial} if $\overline{L}^{Fact}=L$.
All languages considered in this contribution are factorial.
Given a configuration $x\in\A^{\Z^d}$, the language $\Lcal(x)$ defined by $x$ is
\begin{equation*}
    \Lcal(x) = \{u\in\A^{*^d} \mid u\text{ is a $d$-dimensional subword of } x\}.
\end{equation*}
The language of a subshift $X\subseteq\A^{\Z^d}$ is
    $\Lcal_X = \cup_{x\in X} \Lcal(x)$.
Conversely, given a factorial language $L\subseteq\A^{*^d}$ we define the subshift
\begin{equation*}
    \Xcal_L = \{x\in\A^{\Z^d}\mid \Lcal(x)\subseteq L\}.
\end{equation*}
A $d$-dimensional subword $u\in\A^{*^d}$ is \defn{allowed} in a 
subshift $X\subset\A^{\Z^d}$ if $u\in\Lcal_X$
and it is \defn{forbidden} in $X$ if $u\notin\Lcal_X$.
A language $L\subseteq\A^{*^d}$ is forbidden in a 
subshift
$X\subset\A^{\Z^d}$ if $L\cap\Lcal_X=\varnothing$.

\subsection{$d$-dimensional morphisms}\label{sec:d-dim-morphism}

In this section, we generalize the definition of $d$-dimensional morphisms
\cite{MR2579856} to the case where the domain and codomain are different as for
$S$-adic systems \cite{MR3330561}.

Let $\A$ and $\Bcal$ be two alphabets.
Let $L\subseteq\A^{*^d}$ be a factorial language.
A function $\omega:L\to\Bcal^{*^d}$ is a \defn{$d$-dimensional
morphism} if for every
$i$ with $1\leq i\leq d$,
and every $u,v\in L$ such that $u\odot^i v$ is defined and is in $L$ 
we have
that the concatenation $\omega(u)\odot^i \omega(v)$
in direction $\be_i$ is defined and
\begin{equation*}
    \omega(u\odot^i v) = \omega(u)\odot^i \omega(v).
\end{equation*}
Note that the left-hand side of the equation is defined since
$u\odot^i v$ belongs to the domain of $\omega$.
A $d$-dimensional morphism $L\to\Bcal^{*^d}$ is thus completely defined from the
image of the letters in $\A$, so we sometimes denote
a $d$-dimensional morphism as a rule $\A\to\Bcal^{*^d}$ 
when the language $L$ is unspecified.

Given a language $L\subseteq\A^{*^d}$ of $d$-dimensional words and 
a $d$-dimensional morphism $\omega:L\to\Bcal^{*^d}$, we define the image of the
language $L$ under $\omega$ as the language
\begin{equation*}
\overline{\omega(L)}^{Fact}
    = \{u\in\Bcal^{*^d} \mid u\text{ is a $d$-dimensional subword of }
                  \omega(v) \text{ with } v\in L\}
    \subseteq \Bcal^{*^d}.
\end{equation*}


Let $L\subseteq\A^{*^d}$ be a factorial language
and $\Xcal_L\subseteq\A^{\Z^d}$ be the subshift generated by $L$.
A $d$-dimensional morphism 
$\omega:L \to\Bcal^{*^d}$ 
can be extended to a continuous map 
$\omega:\Xcal_L\to\Bcal^{\Z^d}$
in such a way that the origin of $\omega(x)$ is at zero position
in the word $\omega(x_\zero)$
for all $x\in\Xcal_L$. More precisely, the image
under $\omega$ of the configuration $x\in\Xcal_L$ is
\[
    \omega(x) =
    \lim_{n\to\infty}\sigma^{f(n)}\omega\left(\sigma^{-n\UN}(x|_{\llbracket-n\UN,n\UN\llbracket})\right)
    \in\Bcal^{\Z^d}
\]
where $\UN=(1,\dots,1)\in\Z^d$,
$f(n)=\shape\left(\omega(\sigma^{-n\UN}(x|_{\llbracket-n\UN,\zero\llbracket}))\right)$
for all $n\in\N$ and
$\llbracket\bm,\bn\llbracket
=
\ZZrange{m_1}{n_1-1}
\times \cdots \times
\ZZrange{m_d}{n_d-1}$.

In general, the closure under the shift of the image of a subshift $X\subseteq\A^{\Z^d}$ under $\omega$
is the subshift
\begin{equation*}
\overline{\omega(X)}^{\sigma}
    = \{\sigma^\bk\omega(x)\in\Bcal^{\Z^d} \mid \bk\in\Z^d, x\in X\}
    \subseteq \Bcal^{\Z^d}.
\end{equation*}
The next lemma states that $d$-dimensional morphisms preserve minimality of subshifts.

\begin{lemma}\label{lem:minimal-implies-minimal}
    \cite{MR4226493}
    Let $\omega:X\to\Bcal^{\Z^d}$ be a $d$-dimensional morphism for some 
    $X\subseteq\A^{\Z^d}$.
    If $X$ is a minimal subshift, then $\overline{\omega(X)}^{\sigma}$ is
    a minimal subshift.
\end{lemma}

\subsection{Self-similar subshifts}\label{sec:self-similar-subshifts}

In this section, we consider languages and subshifts defined from morphisms
leading to self-similar structures. 
In this situation, the domain and codomain
of morphisms are defined over the same alphabet. 
Formally, we consider the case of $d$-dimensional morphisms
$\A\to\Bcal^{*^d}$ where $\A=\Bcal$.

The definition of self-similarity depends on the notion of expansiveness.
It avoids the presence of lower-dimensional self-similar structure by having
expansion in all directions.
We say that a $d$-dimensional morphism $\omega:\A\to\A^{*^d}$ is
\defn{expansive}
if for every $a\in\A$ and $K\in\N$,
there exists $m\in\N$ such that 
$\min(\shape(\omega^m(a)))>K$.

\begin{definition}
A subshift $X\subseteq\A^{\Z^d}$ 
is \defn{self-similar}
if there exists an expansive
$d$-dimensional morphism $\omega:\A\to\A^{*^d}$ such that
$X=\overline{\omega(X)}^\sigma$.
The map $\omega$ is called the \defn{self-similarity} of $X$.
\end{definition}

Respectively, 
a language $L\subseteq\A^{*^d}$
is self-similar
if there exists an expansive
$d$-dimensional morphism $\omega:\A\to\A^{*^d}$ such that
$L=\overline{\omega(L)}^{Fact}$.
Self-similar languages and subshifts can be constructed by iterative
application of a morphism $\omega$ starting with the letters.
The language $\Lcal_\omega$ defined by an expansive $d$-dimensional
morphism $\omega:\A\to\A^{*^d}$ is
\begin{equation*}
    \Lcal_\omega = \{u\in\A^{*^d} \mid u\text{ is a $d$-dimensional subword of }
    \omega^n(a) \text{ for some } a\in\A\text{ and } n\in\N \}.
\end{equation*}
It satisfies
$\Lcal_\omega=\overline{\omega(\Lcal_\omega)}^{Fact}$
and thus is self-similar.
The \defn{substitutive subshift} $\Xcal_\omega=\Xcal_{\Lcal_\omega}$
defined from the language of $\omega$ is a self-similar subshift
since $\Xcal_\omega=\overline{\omega(\Xcal_\omega)}^{\sigma}$ holds.

Substitutive shift obtained from expansive and primitive morphisms are
interesting for their properties.
As in the one-dimensional case, we say that $\omega$ is \defn{primitive}
if there exists $m\in\N$ such that
for every $a,b\in\A$ the letter $b$ occurs in $\omega^m(a)$.

\begin{lemma}\label{lem:xcal_omega_minimal}
    Let $\omega:\A\to\A^{*^d}$ be an expansive and primitive $d$-dimensional
    morphism. Then the self-similar substitutive subshift $\Xcal_\omega$ is minimal.
\end{lemma}

\begin{proof}
    The substitutive shift of $\omega$ is well-defined since $\omega$ is expansive
    and it is minimal since $\omega$ is primitive 
    using standard arguments \cite[\S 5.2]{MR2590264}.
\end{proof}

\section{Toral $\Z^2$-rotations and polygon exchange transformations (PETs)}

Let $\Gamma$ be a \defn{lattice} in $\R^2$, i.e., a discrete subgroup of the
additive group $\R^2$ with $2$ linearly independent generators.
This defines a $2$-dimensional torus $\generictorus=\R^2/\Gamma$. 
By analogy with the rotation $x\mapsto x+\alpha$ on the circle $\R/\Z$ for
some $\alpha\in\R/\Z$, we use the terminology of \emph{rotation}
to denote the following $\Z^2$-action defined on a 2-dimensional torus.

\begin{definition}\label{def:Z2-rotation}
For some $\balpha,\bbeta\in\generictorus$, we consider
the dynamical system $(\generictorus, \Z^2, R)$ where
$R:\Z^2\times\generictorus\to\generictorus$ 
is the continuous $\Z^2$-action on $\generictorus$
defined by
\[
R^\bn(\bx):=R(\bn,\bx)=\bx + n_1\balpha + n_2\bbeta
\]
for every $\bn=(n_1,n_2)\in\Z^2$.
We say that the $\Z^2$-action $R$ is a 
    \defn{toral $\Z^2$-rotation} or a
    \defn{$\Z^2$-rotation on} $\generictorus$
    which defines a dynamical system $(\generictorus,\Z^2,R)$.
\end{definition}

When the $\Z^2$-action $R$ is a $\Z^2$-rotation on the torus $\generictorus$,
the maps $R^{\be_1}$ and $R^{\be_2}$ can be seen as polygon exchange
transformations \cite{MR3010377,MR3186232}
on a fundamental domain of~$\generictorus$.

\begin{definition}{\rm\cite{alevy_kenyon_yi}}
Let $X$ be a polygon together with
two topological partitions of $X$ into polygons
\[
    X=\bigcup_{k=0}^N P_k
     =\bigcup_{k=0}^N Q_k
\] 
such that for each $k$, $P_k$ and $Q_k$ are translation equivalent, i.e.,
there exists $v_k\in\R^2$ such that $P_k=Q_k+v_k$.
A \defn{polygon exchange transformation (PET)} is the piecewise translation
on $X$ defined for $x\in P_k$ by
$T(x) = x+v_k$.
The map is not defined for points $x\in\bigcup_{k=0}^N\partial P_k$.
\end{definition}

The fact that a rotation on a circle can be seen as a exchange of two
intervals is well-known as noticed for example in \cite{MR551341}.
It generalizes in higher dimension where a generic translation on a
$d$-dimensional torus is a polyhedron exchange transformation defined by the
exchange of at most $2^d$ pieces on a fundamental domain having for shape a
$d$-dimensional parallelotope.
We state a 2-dimensional version of this lemma restricted
to the case of rectangular fundamental domain because we use this connection
several times in the following sections to prove that induced
$\Z^2$-actions are again $\Z^2$-rotations on a torus.

\begin{lemma}\label{lem:PET-iff-toral-rotation}
    Let $\Gamma=\ell_1\Z\times\ell_2\Z$ be a lattice in $\R^2$
    and its rectangular fundamental domain
    $D=[0,\ell_1)\times[0,\ell_2)$.
    For every $\alpha=(\alpha_1,\alpha_2)\in D$,
    the dynamical system $(\R^2/\Gamma, \Z, \bx\mapsto\bx+\alpha)$
    is measurably conjugate to
    the dynamical system $(D, \Z, T)$ where $T:D\to D$
    is the polygon exchange transformation shown
    in Figure~\ref{fig:PET-toral-rotation}.
\end{lemma}

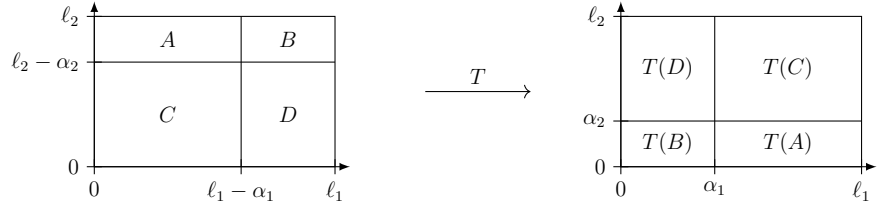
\begin{figure}[h]
\begin{center}
\begin{tikzpicture}[auto,scale=2.0,every node/.style={scale=0.7}]
    \def\aI{.624}
    \def\aII{.304}
    \def\lengthI{1.6}
    \def\lengthII{1}
\begin{scope}
    \draw[-latex] (0,0) -- (\lengthI+.1,0);
    \foreach \x/\y in
    {0/0, \lengthI/\ell_1, \lengthI-\aI  /\ell_1-\alpha_1}
    \draw (\x,.05) -- (\x,-.05) node[below] {$\y$};
    \draw[-latex] (0,0) -- (0,\lengthII+.1);
    \foreach \x/\y in
    {0/0, \lengthII/\ell_2, \lengthII-\aII  /\ell_2-\alpha_2}
    \draw (.05,\x) -- (-.05,\x) node[left] {$\y$};
    \draw (0,0) -- (0,\lengthII) -- (\lengthI,\lengthII) -- (\lengthI,0) --
    cycle; 
    \draw (\lengthI-\aI,0) -- (\lengthI-\aI, \lengthII);
    \draw (0,\lengthII-\aII) -- (\lengthI, \lengthII-\aII);
    \node at (\lengthI*.5-\aI*.5,  \lengthII-\aII*.5)    {$A$};
    \node at (\lengthI-\aI*.5,  \lengthII-\aII*.5)       {$B$};
    \node at (\lengthI*.5-\aI*.5,  \lengthII*.5-\aII*.5) {$C$};
    \node at (\lengthI-\aI*.5,  \lengthII*.5-\aII*.5)    {$D$};
\end{scope}
\begin{scope}[xshift=2.2cm,yshift=5mm]
    \draw[->] (0,0) -- node[above]{$T$} (.7,0);
\end{scope}
\begin{scope}[xshift=3.5cm]
    \draw[-latex] (0,0) -- (\lengthI+.1,0);
    \foreach \x/\y in
    {0/0, \lengthI/\ell_1, \aI  /\alpha_1}
    \draw (\x,.05) -- (\x,-.05) node[below] {$\y$};
    \draw[-latex] (0,0) -- (0,\lengthII+.1);
    \foreach \x/\y in
    {0/0, \lengthII/\ell_2, \aII  /\alpha_2}
    \draw (.05,\x) -- (-.05,\x) node[left] {$\y$};
    \draw (0,0) -- (0,\lengthII) -- (\lengthI,\lengthII) -- (\lengthI,0) --
    cycle; 
    \draw (\aI,0) -- (\aI, \lengthII);
    \draw (0,\aII) -- (\lengthI, \aII);
    \node at (\aI*.5,              \lengthII*.5+\aII*.5)  {$T(D)$};
    \node at (\lengthI*.5+\aI*.5,  \lengthII*.5+\aII*.5)  {$T(C)$};
    \node at (\aI*.5,              \aII*.5)               {$T(B)$};
    \node at (\lengthI*.5+\aI*.5,  \aII*.5)               {$T(A)$};
\end{scope}
\end{tikzpicture}
\end{center}
\caption{The polygon exchange transformation $T$ of the rectangle
    $[0,\ell_1)\times[0,\ell_2)$ as defined on the figure can be seen as a
    toral translation by the vector $(\alpha_1,\alpha_2)$ on the torus
    $\R^2/(\ell_1\Z\times\ell_2\Z)$.}
    \label{fig:PET-toral-rotation}
\end{figure}

\begin{proof}
    It follows from the fact that toral rotations and such polygon exchange
    transformations are the Cartesian product of circle rotations and exchange
    of two intervals.
\end{proof}

\section{Symbolic dynamical systems for toral $\Z^2$-rotations}

\subsection{Symbolic dynamical systems}

Let $\Gamma$ be a lattice in $\R^2$ and
$\generictorus=\R^2/\Gamma$ be a $2$-dimensional torus.
Let $(\generictorus,\Z^2,R)$ be the dynamical system 
given by a $\Z^2$-rotation $R$ on $\generictorus$.
For some finite set $\A$,
a \defn{topological partition} of $\generictorus$ is a finite
collection $\{P_a\}_{a\in\A}$ of disjoint open sets $P_a\subset\generictorus$
such that 
    $\generictorus = \bigcup_{a\in\A} \overline{P_a}$.
If $S\subset\Z^2$ is a finite set,
we say that a pattern $w\in\A^S$
is \defn{allowed} for $\Pcal,R$ if
\begin{equation}\label{eq:allowed-if-nonempty}
    \bigcap_{\bk\in S} R^{-\bk}(P_{w_\bk}) \neq \varnothing.
\end{equation}
Let $\Lcal_{\Pcal,R}$ be the collection of all allowed patterns for $\Pcal,R$.
The set $\Lcal_{\Pcal,R}$ is the language of a subshift 
$\Xcal_{\Pcal,R}\subseteq\A^{\Z^2}$ defined as follows,
see \cite[Prop.~9.2.4]{MR3525488},
\[
    \Xcal_{\Pcal,R} = 
    \{x\in\A^{\Z^2} \mid \pi_S\circ\sigma^\bn(x)\in\Lcal_{\Pcal,R}
    \text{ for every } \bn\in\Z^2 \text{ and finite subset } S\subset\Z^2\}.
\]

\begin{definition}
We call $\Xcal_{\Pcal,R}$ the \defn{symbolic dynamical
system corresponding to $\Pcal,R$}.
\end{definition}

For each $w\in\Xcal_{\Pcal,R}\subset\A^{\Z^2}$ and $n\geq 0$ there is a corresponding nonempty open set
\[
    D_n(w) = \bigcap_{\Vert\bk\Vert\leq n} R^{-\bk}(P_{w_\bk}) \subseteq \generictorus.
\]
The closures $\overline{D}_n(w)$ of these sets are compact
and decrease with $n$, so that
$\overline{D}_0(w)\supseteq
\overline{D}_1(w)\supseteq
\overline{D}_2(w)\supseteq
\dots$.
It follows that $\cap_{n=0}^{\infty}\overline{D}_n(w)\neq\varnothing$.
In order for points in
$\Xcal_{\Pcal,R}$
to correspond to points in $\generictorus$, this intersection should contain only one point.
This leads to the following definition.

\begin{definition}
A topological partition $\Pcal$ of $\generictorus$ \defn{gives a symbolic representation}
of $(\generictorus,\Z^2,R)$ if for every $w\in\Xcal_{\Pcal,R}$ the intersection
$\cap_{n=0}^{\infty}\overline{D}_n(w)$ consists of exactly one
point $\bx\in\generictorus$.
We call $w$ a \defn{symbolic representation of $\bx$}.
\end{definition}

The set
\begin{equation*}\label{eq:boundaries}
    \Delta_{\Pcal,R}:=\bigcup_{\bn\in\Z^2}R^\bn
         \left(\bigcup_{a\in\A}\partial P_a\right)
         \subset \generictorus
\end{equation*}
is the set of points whose orbit under the $\Z^2$-action of $R$ intersect
the boundary of the topological partition
$\Pcal=\{P_a\}_{a\in\A}$.
From Baire Category Theorem \cite[Theorem 6.1.24]{MR1369092}, the set
$\generictorus\setminus\Delta_{\Pcal,R}$ is dense in $\generictorus$.

A topological partition $\Pcal=\{P_a\}_{a\in \A}$ of $\generictorus=\R^2/\Gamma$
is associated to a coding map
\[
\begin{array}{rccl}
    \sccode:& \generictorus\setminus\left(\bigcup_{a\in \A}\partial
    P_a\right) &\to & \A \\
    &\bx &\mapsto & a \quad\text{ if and only if }
    \quad \bx\in P_a.
\end{array}
\]
For every starting point 
$\bx\in \generictorus\setminus \Delta_{\Pcal,R}$, 
the coding of its orbit under the
$\Z^2$-action of $R$ is a configuration
$\scConfig^{\Pcal,R}_{\bx}\in\A^{\Z^2}$ defined by
\[
    \scConfig^{\Pcal,R}_{\bx}(\bn)=\sccode(R^\bn(\bx)).
\]
for every $\bn\in\Z\times\Z$.

\begin{lemma}\label{lem:closure-of-tilings}
The symbolic dynamical system $\Xcal_{\Pcal,R}$
    corresponding to $\Pcal,R$ is the topological closure of the set of configurations:
\[
    \Xcal_{\Pcal,R}
    = \overline{\left\{\scConfig^{\Pcal,R}_\bx\mid
                       \bx\in\generictorus\setminus \Delta_{\Pcal,R}\right\}}.
\]
\end{lemma}

\begin{proof}
    ($\supseteq$) 
    Let $\bx\in\generictorus\setminus \Delta_{\Pcal,R}$.
    The patterns appearing in the configuration
    $\scConfig^{\Pcal,R}_\bx$ are in $\Lcal_{\Pcal,R}$.
    Thus $\scConfig^{\Pcal,R}_\bx\in\Xcal_{\Pcal,R}$.
    The topological closure of such configurations is in $\Xcal_{\Pcal,R}$ 
    since $\Xcal_{\Pcal,R}$ is topologically closed.

    ($\subseteq$)
    Let $w\in\A^S$ be a pattern with finite support $S\subset\Z^2$ appearing
    in $\Xcal_{\Pcal,R}$.
    Then $w\in\Lcal_{\Pcal,R}$ and from Equation~\eqref{eq:allowed-if-nonempty}
    there exists $\bx\in\generictorus\setminus\Delta_{\Pcal,R}$
    such that
        $\bx\in\bigcap_{\bk\in S} R^{-\bk}(P_{w_\bk})$.
        The pattern $w$ appears in the configuration
    $\scConfig^{\Pcal,R}_\bx$.
    Any configuration
    in $\Xcal_{\Pcal,R}$ is 
    the limit of a sequence $(w_n)_{n\in\N}$ of patterns covering
    a ball of radius $n$ around the origin,
    thus equal to
    some limit $\lim_{n\to\infty}\scConfig^{\Pcal,R}_{\bx_n}$
    with $\bx_n\in\generictorus\setminus\Delta_{\Pcal,R}$ for every $n\in\N$.
\end{proof}

\subsection{Factor map}

An important consequence of the fact that
a partition $\Pcal$ gives a symbolic representation of the dynamical system
$(\generictorus,\Z^2,R)$ is the existence of a
factor map $f:\Xcal_{\Pcal,R}\to\generictorus$ which commutes the
$\Z^2$-actions.
In the spirit of \cite[Prop.~6.5.8]{MR1369092} for $\Z$-actions,
we have the following proposition whose
proof can be found in \cite{labbe_markov_2021}.

\begin{proposition}\label{prop:factor-map}
    {\rm\cite[Prop.~5.1]{labbe_markov_2021}}
    Let $\Pcal$ give a symbolic representation of the dynamical system
    $(\generictorus,\Z^2,R)$.
    Let $f:\Xcal_{\Pcal,R}\to\generictorus$ be defined 
    such that $f(w)$ is the unique point
    in the intersection $\cap_{n=0}^{\infty}\overline{D}_n(w)$.
    The map $f$ is a factor map from
            $(\Xcal_{\Pcal,R},\Z^2,\sigma)$ to $(\generictorus,\Z^2,R)$
            such that $R^\bk\circ f = f\circ\sigma^\bk$
    for every $\bk\in\Z^2$.
    The map $f$ is one-to-one on
    $f^{-1}(\generictorus\setminus\Delta_{\Pcal,R})$.
\end{proposition}

Using the factor map,
one can prove the following lemma that we use herein in the proof of
Corollary~\ref{cor:equality-X0=XP0R0}.

\begin{lemma}\label{lem:minimal-aperiodic}
    {\rm\cite[Lemma~5.2]{labbe_markov_2021}}
    Let $\Pcal$ give a symbolic representation of the dynamical system
    $(\generictorus,\Z^2,R)$. Then
\begin{enumerate}[\rm (i)]
    \item if $(\generictorus,\Z^2,R)$ is minimal,
        then $(\Xcal_{\Pcal,R},\Z^2,\sigma)$ is minimal,
    \item if $R$ is a free $\Z^2$-action on $\generictorus$,
        then $\Xcal_{\Pcal,R}$ aperiodic.
\end{enumerate}
\end{lemma}

Of course, Lemma~\ref{lem:minimal-aperiodic},
whose proof can also be found in \cite{labbe_markov_2021},
does not hold
if $\Pcal$ does not give a symbolic representation of $(\generictorus,\Z^2,R)$.
For example, consider
the $\Z^2$-rotation on $\torusI^2$
defined by
$R(\bn,\bx)=\bx + (n_1\alpha, n_2\beta)$
for every $\bn=(n_1,n_2)\in\Z^2$
for some fixed irrational $\alpha,\beta\in\R\setminus\Q$.
If $\Pcal$ is a partition of the fundamental domain $[0,1)^2$
consisting of horizontal rectangles cutting all the way
across the domain, all configurations of the shift space 
$\Xcal_{\Pcal,R}$ will be fixed by the horizontal shift.
Eventhough $R$ is a free $\Z^2$-action,
$\Xcal_{\Pcal,R}$ contains periodic configuration.
This is not a contradiction since the partition into horizontal rectangles does
not give a symbolic representation of $(\torusI^2,\Z^2,R)$.
Indeed, for each $w\in\Xcal_{\Pcal,R}$,
the intersection $\cap_{n=0}^{\infty}\overline{D}_n(w)$ is an
horizontal segment in the fundamental domain.
In general, the existence of an atom of the partition of the torus $\generictorus$ which is invariant only under the trivial
translation is a sufficient condition for
the partition to give a symbolic representation of a minimal $\Z^2$-rotation on $\generictorus$,
see \cite[Lemma~3.4]{labbe_markov_2021}.

\subsection{Markov partitions for toral $\Z^2$-rotations}\label{sec:Markov-partition}

Markov partitions were originally defined for one-dimensional dynamical
systems $(\generictorus,\Z,R)$ and were extended to $\Z^d$-actions by automorphisms of
compact Abelian group in \cite{MR1632169}.
Following \cite{labbe_markov_2021},
we allow ourselves to use the same terminology 
and extend the definition proposed in \cite[\S 6.5]{MR1369092}
for dynamical systems defined by higher-dimensional actions by rotations.

\begin{definition}\label{def:Markov}
A topological partition $\Pcal$ of $\generictorus$ is a \defn{Markov partition} for
$(\generictorus,\Z^2,R)$ if 
\begin{itemize}
    \item $\Pcal$ gives a symbolic representation of $(\generictorus,\Z^2,R)$ and 
    \item $\Xcal_{\Pcal,R}$ is a shift of finite type (SFT).
\end{itemize}
\end{definition}

In this article, we consider Markov partitions associated with
aperiodic subshifts of finite type over $\Z^2$ coded by toral rotations (thus of zero entropy). 
This may seem counter-intuitive for the reader since Markov partitions are usually
associated with hyperbolic systems (thus with positive entropy).
Moreover, the coding of an irrational rotation on the circle leads to aperiodic
Sturmian sequences which are not SFT.
Our opinion is that positive entropy and all associated intuitions follows
from the restriction of Definition~\ref{def:Markov} to the case of $\Z$-actions,
but not from the notion of Markov partition itself.
In this article, we made a choice by using the terminology of \textit{Markov partitions}
in the unusual setup of $\Z^2$-rotations.

Of course, SFTs over $\Z^2$ are much
different then SFTs over $\Z$.
The emptyness of $\Z^2$-SFTs is undecidable \cite{MR0216954}
and the
possible entropies achievable by a
$\Z^2$-SFT are exactly the
non-negative numbers obtainable as the limit of computable decreasing sequences
of rationals \cite{MR2680402},
as opposed to be given by an algebraic characterization
in the case of $\Z$-SFT, see \cite[\S 4]{MR1369092}.
In particular, there exist aperiodic $\Z^2$-SFTs of zero entropy
which is not possible in the one-dimensional case,
since infinite $\Z$-SFTs have positive entropy and contain a periodic configuration.

\part{Induction of $\Z^2$-actions and induced partitions of the 2-torus}
\label{part:induction-z2-actions}

This part is divided into four sections
on 
induced $\Z^2$-actions;
toral partitions induced by $\Z^2$-rotations,
renormalization schemes and
algorithms.

\section{Induced $\Z^2$-actions}

Renormalization schemes also known as \emph{Rauzy induction} were originally
defined for dynamical system including interval exchange transformations
(IETs) \cite{MR543205}. A natural way to generalize it to higher
dimension is to consider polygon exchange transformations
\cite{MR3010377,alevy_kenyon_yi} or even polytope exchange
transformations \cite{MR3186232,schwartz_outer_2011} where only one map is
considered. But more dimensions also allows to induce two or more (commuting)
maps at the same time.

In this section, we define the induction of $\Z^2$-actions on a sub-domain.
We consider the torus $\generictorus=\R^2/\Gamma$
where $\Gamma$ is a lattice in $\R^2$.
Let $(\generictorus,\Z^2,R)$ be a minimal dynamical system 
given by a $\Z^2$-action $R$ on $\generictorus$.
For every $\bn\in\Z^2$, the toral translation $R^{\bn}$ can be seen as a polygon
exchange transformation on a fundamental domain of $\generictorus$.

The \defn{set of return times} of $\bx\in\generictorus$ 
to $W\subset\generictorus$
under the $\Z^2$-action $R$ is the subset of $\Z\times\Z$ defined as:
\[
    \delta_W(\bx) = \{ \bn\in\Z\times\Z\mid R^\bn(\bx)\in W \}.
\]
\begin{definition}\label{def:Cartesian-on-a-window}
Let $W\subset\generictorus$.
We say that the $\Z^2$-action $R$ is \defn{Cartesian on $W$} if
the set of return times $\delta_W(\bx)$ can be expressed as a Cartesian product,
that is,
for all $\bx\in\generictorus$
there exist two strictly increasing sequences $r_\bx^{(1)},r_\bx^{(2)}:\Z\to\Z$ 
such that
\[
\delta_W(\bx)=r_\bx^{(1)}(\Z)\times r_\bx^{(2)}(\Z).
\]
We always assume that the sequences are
shifted in such a way that
\[
r_\bx^{(i)}(n)\geq 0 \iff n\geq0
\qquad
\text{ for }
i\in\{1,2\}.
\]
In particular, if $\bx\in W$, then $(0,0)\in\delta_W(\bx)$, so that
$r_\bx^{(1)}(0)=r_\bx^{(2)}(0)=0$.
\end{definition}

When the $\Z^2$-action $R$ is Cartesian on $W\subset\generictorus$,
we say that the tuple
\begin{equation}\label{eq:first-return-time-tuple}
    (r_\bx^{(1)}(1),r_\bx^{(2)}(1))
\end{equation}
is
the \defn{first return time} of a starting point $\bx\in\generictorus$ 
to $W\subset\generictorus$ under the action $R$.
When the $\Z^2$-action $R$ is Cartesian on $W\subset\generictorus$,
we may consider its return map on $W$ and we prove in the next lemma that this
induces a $\Z^2$-action on $W$. 

\begin{lemma}\label{lem:Cartesian-induced-action}
If the $\Z^2$-action $R$ is Cartesian on $W\subset\generictorus$, then
the map
    $\widehat{R}|_W:\Z^2\times W\to W$ defined by
\[
    (\widehat{R}|_W)^\bn(\bx):=\widehat{R}|_W(\bn,\bx)=
    R^{(r_\bx^{(1)}(n_1),r_\bx^{(2)}(n_2))}(\bx)
\]
for every $\bn=(n_1,n_2)\in\Z^2$ is a well-defined $\Z^2$-action.
\end{lemma}

We say that $\widehat{R}|_W$ is the \defn{induced $\Z^2$-action} of the $\Z^2$-action $R$ on $W$.

\begin{proof}
    Let $\bx\in W$.
    We have that
    \[
    \widehat{R}|_W(\zero,\bx)
    = R^{(r_\bx^{(1)}(0),r_\bx^{(2)}(0))}(\bx)
    = R^{(0,0)}(\bx)
    = \bx.
    \]
    Firstly, using
    \[
    r_\bx^{(i)}(k+n) = r_{R^{\be_i\cdot r_\bx^{(i)}(n)}(\bx)}(k)+r_\bx^{(i)}(n),
    \]
    and skipping few details, we get
    \[
    \widehat{R}|_W(k\be_i+n\be_i,\bx)
    = \widehat{R}|_W\left(k\be_i,
    \left(\widehat{R}|_W(n\be_i,\bx)\right)\right).  
    \]
    Secondly, using the fact that
    \[
        r_\bx^{(1)}(k_1) = r_\by^{(1)}(k_1)
    \]
    whenever $y=R^{(0,r_\bx^{(2)}(k_2))}\bx=\widehat{R}|_W(k_2\be_2,\bx)$, we
    obtain
    \begin{align*}
    \widehat{R}|_W(\bk,\bx)
    &= R^{(r_\bx^{(1)}(k_1),r_\bx^{(2)}(k_2))}(\bx)
    = R^{(r_\bx^{(1)}(k_1),0)} R^{(0,r_\bx^{(2)}(k_2))}(\bx)\\
    &= R^{(r_\bx^{(1)}(k_1),0)} \widehat{R}|_W(k_2\be_2,\bx)
    = \widehat{R}|_W\left(k_1\be_1,
      \widehat{R}|_W(k_2\be_2,\bx)\right).  
    \end{align*}
    Therefore, for every $\bk,\bn\in\Z^2$, we have
    \begin{align*}
        (\widehat{R}|_W)^{\bk+\bn}(\bx)
    &= (\widehat{R}|_W)^{(k_1+n_1)\be_1} 
        (\widehat{R}|_W)^{(k_2+n_2)\be_2}(\bx)\\
     &= (\widehat{R}|_W)^{k_1\be_1} (\widehat{R}|_W)^{n_1\be_1}
        (\widehat{R}|_W)^{k_2\be_2} (\widehat{R}|_W)^{n_2\be_2}(\bx)\\
     &= (\widehat{R}|_W)^{k_1\be_1} 
        (\widehat{R}|_W)^{(n_1,k_2)} (\widehat{R}|_W)^{n_2\be_2}(\bx)\\
     &= (\widehat{R}|_W)^{k_1\be_1} (\widehat{R}|_W)^{k_2\be_2}
        (\widehat{R}|_W)^{n_1\be_1} (\widehat{R}|_W)^{n_2\be_2}(\bx)\\
     &= (\widehat{R}|_W)^\bk (\widehat{R}|_W)^\bn(\bx),
    \end{align*}
    which shows that $\widehat{R}|_W$ is a $\Z^2$-action on $W$.
\end{proof}

A consequence of the lemma is that
the induced $\Z^2$-action $\widehat{R}|_W$
is generated by two commutative maps
\[
    (\widehat{R}|_W)^{\be_1}(\bx)
    = R^{(r_\bx^{(1)}(1),0)}(\bx)
    \quad\text{ and }\quad
    (\widehat{R}|_W)^{\be_2}(\bx)
    = R^{(0,r_\bx^{(2)}(1))}(\bx)
\]
which are the first return maps of $R^{\be_1}$ and $R^{\be_2}$
to $W$:
\[
    (\widehat{R}|_W)^{\be_1}(\bx)
    = \widehat{R^{\be_1}}|_W(\bx)
    \quad\text{ and }\quad
    (\widehat{R}|_W)^{\be_2}(\bx)
    = \widehat{R^{\be_2}}|_W(\bx).
\]
Recall that the \defn{first return map}
$\widehat{T}|_W$ of a dynamical system $(X,T)$ maps a point $\bx\in W\subset X$ 
to the first point in the forward orbit of
$T$ lying in $W$, i.e.
\[
    \widehat{T}|_W(\bx) = T^{r(\bx)}(\bx) \quad\text{ where }
    r(\bx) = \min\{k\in\Z_{>0} : T^k(\bx)\in W\}.
\]
From Poincar\'e's recurrence theorem, 
if $\mu$ is a finite $T$-invariant measure on $X$, then
the first return map $\widehat{T}|_W$ is well
defined for $\mu$-almost all $\bx\in W$. 
When $T$ is a translation on a torus, if the subset $W$ is
open, then the first return map is well-defined for every point $\bx\in W$.
Moreover if $W$ is a polygon, then the first return map $\widehat{T}|_W$
is a polygon exchange transformation.
An algorithm to compute the induced transformation 
$\widehat{T}|_W=\widehat{R^{\be_i}}|_W$ of the sub-action $R^{\be_i}$
is provided in Section~\ref{sec:algorithms}.

\section{Toral partitions induced by toral $\Z^2$-rotations}

As for IETs, the domain on which we define the Rauzy induction of
$\Z^2$-rotations is given by a union of atoms from the partition which defines
the IET itself or its inverse. But in the examples that follow, we code the
orbits of toral $\Z^2$-rotations by partitions which carries more information
than the natural partition expressing a $\Z^2$-rotation as a pair of polygon
exchange transformations on a fundamental domain. The partition that we use
are non-trivial refinements of the natural partition involving well-chosen
vertices and sloped lines. Thus we also need to define the Rauzy induction of
the involved coding partitions and not only of the $\Z^2$-rotations.

Let $\Gamma$ be a lattice in $\R^2$ and
$\generictorus=\R^2/\Gamma$ be a $2$-dimensional torus.
Let $(\generictorus,\Z^2,R)$ be the dynamical system 
given by a $\Z^2$-rotation $R$ on $\generictorus$.
Assuming the $\Z^2$-rotation $R$ is Cartesian on a window $W\subset\generictorus$,
then there exist
two strictly increasing sequences $r_\bx^{(1)},r_\bx^{(2)}:\Z\to\Z$ 
are such that
\[
    (r,s) = \left(r_\bx^{(1)}(1),r_\bx^{(2)}(1)\right)
\]
is the first return time of a starting point $\bx\in\generictorus$ to the
window $W$ under the action $R$, see
Equation~\eqref{eq:first-return-time-tuple}.
It allows to define the \defn{return word map} as
\[
\begin{array}{rccl}
    \scReturnWord:&W & \to & \A^{*^2}\\
    &\bx&\mapsto &
    \left(\begin{array}{ccc}
        \sccode(R^{(0,s-1)}\bx) & \cdots & 
        \sccode(R^{(r-1,s-1)}\bx)\\
        \cdots & \cdots & \cdots\\
        \sccode(R^{(0,0)}\bx) & \cdots & \sccode(R^{(r-1,0)}\bx)
    \end{array}\right)
\end{array},
\]
where $r,s\geq 1$ both obviously depend on $\bx$.

The image $\Lcal=\scReturnWord(W)\subset \A^{*^2}$ is a language called the \defn{set of return words}.
When the return time to $W$ is bounded, the set of return words $\Lcal$ is
finite.
Let $\Lcal=\{w_b\}_{b\in\Bcal}$ be an enumeration
of $\Lcal$ for some finite set $\Bcal$.
    The way the enumeration of $\Lcal$ is done influences the substitutions
    which are obtained afterward. To obtain a canonical ordering when the words
    in $\Lcal$
    are 1-dimensional, we use the radix order on $\Lcal$. It
    is used at Line~\ref{algo:line:prec} of Algorithm~\ref{alg:induced-partition}.

\begin{definition}[radix order]\label{def:definition-radix-order}
    The total order $(\A^*,\prec)$ is
    defined by $u\prec v$ if $|u|<|v|$ or $|u|=|v|$ and $u<_{lex}v$
    for every $u,v\in\A^*$.
\end{definition}

The \defn{induced partition} of $\Pcal$ by the action of $R$ on the sub-domain
$W$ is a topological partition of $W$ defined as the set of preimage sets
under $\scReturnWord$:
\[
    \widehat{\Pcal}|_W=\{\scReturnWord^{-1}(w_b)\}_{b\in\Bcal}.
\]
This yields the \defn{induced coding} on $W$
\[
\begin{array}{rccl}
    \sccode|_W:&W & \to &\Bcal\\
    &\by & \mapsto & b \quad \text{ if and only if } 
                       \quad\by\in\scReturnWord^{-1}(w_b).
\end{array}
\]
A \defn{natural substitution} comes out of this induction procedure:
\begin{equation}\label{eq:induced-substitution}
\begin{array}{rccl}
    \omega:&\Bcal & \to & \A^{*^2}\\
    &b & \mapsto & w_b.
\end{array}
\end{equation}
The partition $\widehat{\Pcal}|_W$ of $W$ can be effectively computed by the
refinement of the partition $\Pcal$ with translated copies of the sub-domain $W$
under the action of $R$. 
In Section~\ref{sec:algorithms},
we propose Algorithm~\ref{alg:induced-partition} 
to compute the induced partition $\widehat{\Pcal}|_W$ 
and substitution $\omega$.
The next result shows that
the coding of the orbit under the $\Z^2$-rotation $R$ is
the image under the
2-dimensional substitution $\omega$ of the coding of the orbit under
the $\Z^2$-action $\widehat{R}|_W$.

\begin{lemma}\label{lem:desubstitute}
If the $\Z^2$-action $R$ is Cartesian on a window $W\subset\generictorus$,
    then $\omega$ is a $2$-dimensional morphism, and
    for every $\bx\in W$ we have
    \[
        \scConfig^{\Pcal,R}_{\bx}
        = \omega\left(
        \scConfig^{\widehat{\Pcal}|_W,\widehat{R}|_W}_{\bx}
        \right).
    \]
\end{lemma}

\begin{proof}
    Let $\bx\in W$.
    By hypothesis,
    there exists $P,Q\subset\Z$ such that
    the set of returns times satisfies
    $\delta_W(\bx)=P\times Q$
    and we may write $P=\{r_i\}_{i\in\Z}$
    and $Q=\{s_j\}_{j\in\Z}$ as increasing sequences such that
    $r_0=s_0=0$.
    Therefore, $\scConfig^{\Pcal,R}_{\bx}$ may be decomposed into
    a lattice of rectangular blocks.
    More precisely, for every $i,j\in\Z$, the following block is the image of
    a letter under $\omega$:
    \begin{align*}
        &\left(\begin{array}{ccc}
            \sccode(R^{(r_i,s_{j+1}-1)}\bx) & \cdots 
                & \sccode(R^{(r_{i+1}-1,s_{j+1}-1)}\bx)\\
            \cdots & \cdots & \cdots\\
            \sccode(R^{(r_i,s_j)}\bx) & \cdots 
                & \sccode(R^{(r_{i+1}-1,s_j)}\bx)
        \end{array}\right)\\
        &\qquad= \scReturnWord(R^{(r_i,s_j)}\bx)
        =w_{b_{ij}}=\omega(b_{ij})
    \end{align*}
    for some letter $b_{ij}\in\Bcal$.
    Moreover,
    \[
        b_{ij}
        = \sccode|_W(R^{(r_i,s_j)}\bx)
        = \sccode|_W\left((\widehat{R}|_W)^{(i,j)}\bx\right).
    \]
    Since the adjacent blocks have matching dimensions, 
    for every $i,j\in\Z$, the following concatenations
    \begin{align*}
        & \omega\left(b_{ij}\odot^1 b_{(i+1)j}\right) = 
        \omega\left(b_{ij}\right)\odot^1 \omega\left(b_{(i+1)j}\right)
        \qquad\text{ and }
        \\
        & \omega\left(b_{ij}\odot^2 b_{i(j+1)}\right) = 
        \omega\left(b_{ij}\right)\odot^2 \omega\left(b_{i(j+1)}\right)
    \end{align*}
    are well defined. Thus $\omega$ is a $2$-dimensional morphism
    on the set
    $\left\{ \scConfig^{\widehat{\Pcal}|_W,\widehat{R}|_W}_{\bx}
            \,\middle|\, \bx\in W\right\}$
    and we have
    \[
        \scConfig^{\Pcal,R}_{\bx}
        =
        \omega\left(\scConfig^{\widehat{\Pcal}|_W,\widehat{R}|_W}_{\bx}\right)
    \]
    which ends the proof.
    Note that the domain of $\omega$ can be extended to
    its topological closure.
\end{proof}

\begin{proposition}\label{prop:orbit-preimage}
    Let $\Pcal$ be a topological partition of $\generictorus$.
If the $\Z^2$-action $R$ is Cartesian on a window $W\subset\generictorus$,
    then $\Xcal_{\Pcal,R}=\overline{\omega(\Xcal_{\widehat{\Pcal}|_W,\widehat{R}|_W})}^\sigma$.
\end{proposition}

\begin{proof}
    Let
    \[
        Y = \left\{ \scConfig^{\Pcal,R}_{\bx}\,\middle|\,
                    \bx\in\generictorus\right\}
        \quad\text{ and }\quad
        Z = \left\{ \scConfig^{\widehat{\Pcal}|_W,\widehat{R}|_W}_{\bx}
               \,\middle|\, \bx\in W\right\},
    \]
    ($\supseteq$).
        Let $\bx\in W$.
        From Lemma~\ref{lem:desubstitute},
        $\omega\left(
         \scConfig^{\widehat{\Pcal}|_W,\widehat{R}|_W}_{\bx}\right)
         = \scConfig^{\Pcal,R}_{\bx}$ with $\bx\in W\subset\generictorus$.
    ($\subseteq$).
    Let $\bx\in\generictorus$. There exists $k_1,k_2\in\N$ such that
    $\bx'=R^{-(k_1,k_2)}(\bx)\in W$. 
    Therefore, we have $\bx=R^{(k_1,k_2)}(\bx')$ where $0\leq k_1<r(\bx')$ and
    $0\leq k_2<s(\bx')$.
    Thus the shift $\bk=(k_1,k_2)\in\Z^2$ is bounded by the maximal return time
    of $R^{\be_1}$ and $R^{\be_2}$ to $W$. We have
    \begin{align*}
        \scConfig^{\Pcal,R}_{\bx}
        &= \scConfig^{\Pcal,R}_{R^\bk\bx'}
        = \sigma^{\bk} \circ
        \scConfig^{\Pcal,R}_{\bx'}\\
        &= \sigma^{\bk} \circ \omega
        \left(\scConfig^{\widehat{\Pcal}|_W,\widehat{R}|_W}_{\bx'}\right)
    \end{align*}
    where we used
    Lemma~\ref{lem:desubstitute} with $\bx'\in W$.
    We conclude that $Y=\overline{\omega(Z)}^\sigma$.
    The result follows from Lemma~\ref{lem:closure-of-tilings}
    by taking the topological closure on both sides.
\end{proof}

\subsection*{Including the first vs the last letter}

We finish with a remark on the definition of the return word map.
There is a choice to be made in its
definition whether we include the first or the last letter of
the orbit that comes back to the window (vertically and horizontally). For
example, another option is to define return words that includes the last letter
instead of the first letter. In general, for each value of
$\epsilon_1,\epsilon_2\in\{0,1\}$, we may define
\[
\begin{array}{rccl}
    \scReturnWord_{\epsilon_1,\epsilon_2}:&W & \to & \A^{*^2}\\
    &\bx&\mapsto &
    \left(\begin{array}{ccc}
        \sccode(R^{(\epsilon_1,\epsilon_2+s-1)}\bx) & \cdots & 
        \sccode(R^{(\epsilon_1+r-1,\epsilon_2+s-1)}\bx)\\
        \cdots & \cdots & \cdots\\
        \sccode(R^{(\epsilon_1,\epsilon_2)}\bx) & \cdots & 
        \sccode(R^{(\epsilon_1+r-1,\epsilon_2)}\bx)
    \end{array}\right)
\end{array}
\]
where $r=r_\bx^{(1)}(\bx)$ and $s=r_\bx^{(2)}(\bx)$.
An alternative is to consider $\scReturnWord$ for the 
$\Z^2$-actions
$R^{(-n_1,n_2)}$,
$R^{(n_1,-n_2)}$ or 
$R^{(-n_1,-n_2)}$
   and taking horizontal
and/or vertical mirror images of the obtained return words.

\section{Renormalization schemes}

In what follows, we define a renormalization scheme on a topological partition
of the domain $X$ together with a dynamical systems $(X,\Z^2,R)$ defined by a
$\Z^2$-action $R$ on $X$.

Sometimes,
when $R$ is a $\Z^2$-action on $X$,
the induced $\Z^2$-action $\widehat{R}|_{W}$ can be renormalized as a
$\Z^2$-action on the original domain $X$.
This leads to the definition of renormalization schemes.
The following definition is inspired from \cite{alevy_kenyon_yi} and adapted
to the case dynamical systems defined by $\Z^2$-actions.

\begin{definition}
    A dynamical system $(X,\Z^2,R)$
    has a \defn{renormalization scheme} if there exists 
    a proper subset $W\subset X$,
    a homeomorphism $\phi:W\to X$,
    a \defn{renormalized} dynamical system $(X,\Z^2,R')$
    such that
    \[
        (R')^\bn = \phi \circ (\widehat{R}|_{W})^\bn \circ \phi^{-1}
    \]
for all $\bn\in\Z^2$,
meaning that the following diagram commute:
\[ 
\begin{tikzcd}
    W \arrow{r}{\phi} \arrow[swap]{d}{(\widehat{R}|_W)^\bn}
    & X \arrow{d}{(R')^\bn} \\%
    W \arrow{r}{\phi}
    & X
\end{tikzcd}
\]
    The dynamical system $(X,\Z^2,R)$ is
    \defn{self-induced} if it is equal to the renormalized one,
    that is, if $R=R'$.
\end{definition}

Now we define the renormalization of a topological partition
$\Pcal=\{P_a\}_{a\in\A}$ of $X$.
If a dynamical system $(X,\Z^2,R)$ has a renormalization scheme
with the same notations as above,
then the induced partition $\widehat{\Pcal}|_W=\{Q_b\}_{b\in\Bcal}$ of $W$ can be
\defn{renormalized} as a partition
\[
    \phi(\widehat{\Pcal}|_W)=\left\{\phi(Q_b)\right\}_{b\in\Bcal}
\]
of $X$.
If the dynamical system $(X,\Z^2,R)$ is self-induced
and there exists a bijection $\pi:\Bcal\to\A$ such that
$P_{\pi(b)}=\phi(Q_b)$ for all $b\in\Bcal$, then we say that
the topological partition $\Pcal$ of $X$ is \defn{self-induced}.

\begin{proposition}\label{prop:self-induced-implies-self-similar}
    Let $\Pcal$ be a topological partition of $\generictorus$
    and suppose that the $\Z^2$-action $R$ is Cartesian on a window
    $W\subset\generictorus$.
If $\Pcal$ is self-induced with bijection $\pi:\Bcal\to\A$, then
    $\Xcal_{\Pcal,R}$
    is self-similar satisfying
    $\Xcal_{\Pcal,R}=\overline{\omega\pi^{-1}(\Xcal_{\Pcal,R})}^\sigma$
    where $\omega$ is the $2$-dimensional morphism defined in
    Eq.~\eqref{eq:induced-substitution}.
\end{proposition}

\begin{proof}
From Proposition~\ref{prop:orbit-preimage}, we have
$\Xcal_{\Pcal,R}=\overline{\omega(\Xcal_{\widehat{\Pcal}|_W,\widehat{R}|_W})}^\sigma$.
It remains to show
    $\Xcal_{\widehat{\Pcal}|_W,\widehat{R}|_W}=\pi^{-1}\left(\Xcal_{\Pcal,R}\right)$.
Since $\Pcal$ is self-induced, we have
$\pi\circ\sccode|_W= \sccode \circ \phi$. 
Thus for every $m,n\in\Z$ and every $\bx\in W$, we have
\begin{align*}
    \pi\left(\scConfig^{\widehat{\Pcal}|_W,\widehat{R}|_W}_{\bx}(\bn)\right)
    &= \pi\circ\sccode|_W \circ (\widehat{R}|_W)^\bn (\bx)
    = \sccode \circ \phi \circ (\widehat{R}|_W)^\bn (\bx)\\
    &= \sccode \circ R^\bn \circ \phi(\bx)
    = \scConfig^{\Pcal,R}_{\phi(\bx)} (\bn).
\end{align*}
Since $\phi:W\to\generictorus$ is a homeomorphism,
we obtain
    \begin{align*}
\Xcal_{\Pcal,R}
    &= \overline{\left\{ \scConfig^{\Pcal,R}_{\bx}\,\middle|\,
                \bx\in\generictorus\right\}}
        = \overline{\left\{ \scConfig^{\Pcal,R}_{\phi(\bx)}\,\middle|\,
                \bx\in W\right\}}\\
        &= \overline{\left\{
            \pi\left(\scConfig^{\widehat{\Pcal}|_W,\widehat{R}|_W}_{\bx}\right) \,\middle|\,
                    \bx\in W\right\}}\\
        &= \pi\left(\overline{\left\{
                     \scConfig^{\widehat{\Pcal}|_W,\widehat{R}|_W}_{\bx} \,\middle|\,
                    \bx\in W\right\}}\right)
        =\pi\left(\Xcal_{\widehat{\Pcal}|_W,\widehat{R}|_W}\right)
\end{align*}
and the conclusion follows.
\end{proof}

\section{Algorithms}
\label{sec:algorithms}

\begin{algorithm}[h]
    \caption{Compute the induced partition $\widehat{\Pcal}|_W$ and substitution $\omega$
             associated to the induced transformation $\widehat{T}|_W$
             (we use it when $T=R^{\be_i}$ for some $i$.)}
    \label{alg:induced-partition}
  \begin{algorithmic}[1]
      \Require $T$ is a polytope exchange transformation (PET)
      on a convex domain $D\subset\R^d$ 
      and $\Gcal$ is a partition of $D$ into convex polytopes
      such that the restriction of $T$ 
      on each atom of $\Gcal$ is continuous;
      $\bv\in\R^{d+1}$ defines a half space 
             $H_\bv=\{\bx\in\R^d\mid v_0 + \sum_{i=1}^{d}v_ix_i\geq0\}$ 
             such that $D\cap H_\bv=W$;
             $\Pcal$ is a list of pairs $(a,p)$ such that
             $\{p\mid (a,p)\in\Pcal\}$ is a partition of $D$ into convex polytopes
             indexed by the alphabet $\A=\{a\mid (a,p)\in\Pcal\}$.
      \Function{InducedPartition}{$T$, $\bv$, $\Pcal$}
        \State $\Qcal\gets \{(\varepsilon, W)\}$
                                    \Comment{$\varepsilon\in\A^*$ is the empty word and $W=D\cap H_\bv$}
        \State $\Kcal\gets T(\Pcal\wedge\Gcal)=\{(a, T(p\cap g))\mid (a,p)\in \Pcal\text{ and }
                                                 g\in\Gcal\text{ and }
                                                 p\cap g\neq\varnothing\}$
        \State $S\gets \Call{EmptyList}$ 
        \While{$\Qcal$ not empty} \label{algo:line:while-not-empty}
            \State $\Qcal\gets T^{-1}(\Qcal\wedge\Kcal)=
                               \{(au,T^{-1}(q\cap k)) \mid (u,q)\in\Qcal \text{ and }
                                                        (a,k)\in\Kcal       \text{ and }
                                                        q\cap k\neq\varnothing\}$
            \State $S \gets S\cup \left(\Qcal\wedge H_\bv\right)=
                S\cup \left\{\left(u,q\cap H_\bv\right)\mid (u,q) \in\Qcal\text{ and }
                                                      q\cap H_\bv\neq\varnothing\right\}$
            \State $\Qcal\gets \Qcal\wedge \left(\R^d\setminus H_\bv\right)=
                         \left\{\left(u,q\cap \left(\R^d\setminus H_\bv\right)\right)\mid 
                                    (u,q) \in\Qcal\text{ and }
                                    q\cap \left(\R^d\setminus H_\bv\right)\neq\varnothing\right\}$
        \EndWhile
        \State $\Lcal\gets\{u\in\A^*\mid (u,q)\in S\}$
                                                \Comment{the set of return words}
        \State $\Bcal\gets\{0,1,\dots,\#\Lcal-1\}$
                                                    \Comment{the new alphabet}
        \State $\omega\gets$ bijection $\Bcal\to\Lcal$ 
                \Comment{s.t. $i<j$ if and only if $\omega(i)\prec\omega(j)$,
                         see Definition~\ref{def:definition-radix-order}}
                \label{algo:line:prec}
        \State $\Pcal'\gets \{(\omega^{-1}(u),q) \mid (u,q)\in S\}$
                \Comment{the induced partition labeled by $\Bcal$}
        \State\Return $(\Pcal', \omega)$
      \EndFunction
      \Ensure 
             $\{q \mid (b,q)\in \Pcal'\}=\widehat{\Pcal}|_W$ is the induced partition 
             of $W$ into convex polytopes;
             $\Pcal'$ is a list of pairs $(b,q)$ such that
             $\sccode|_W(\bx)=b$ for every $\bx\in q$;
             the map $\omega:\Bcal\to\A^*$ extends to a morphism of monoid 
      $\omega:\Bcal^*\to\A^*$ satisfying the following equation for one-dimensional subshifts:
      $\Xcal_{\Pcal,T}=\overline{\omega\left(\Xcal_{\widehat{\Pcal}|_W,
                                         \widehat{T}|_W}\right)}^\sigma$;
             if $T=R^{\be_i}$ for some $i$ and
             the $\Z^2$-action $R$ is Cartesian on the window $W$, then
             $\omega:\Bcal\to\A^*$ defines a $d$-dimensional morphism in the direction $\be_i$
             satisfying
      $\Xcal_{\Pcal,R}=\overline{\omega\left(\Xcal_{\widehat{\Pcal}|_W,
                                         \widehat{R}|_W}\right)}^\sigma$.
  \end{algorithmic}
\end{algorithm}

\begin{algorithm}[h]
    \caption{Compute the induced transformation $\widehat{T}|_W$}
    \label{alg:induced-transformation}
  \begin{algorithmic}[1]
      \Require $T$ is a polytope exchange transformation (PET)
      on a convex domain $D\subset\R^d$ given as 
      a pair $(\Pcal,h)$ where
      $\Pcal$ is a list of pairs $(a,p)$ such that
      $\{p\mid (a,p)\in\Pcal\}$ is a partition of $D$ into convex polytopes,
      $\A=\{a\mid (a,p)\in\Pcal\}$ is some alphabet and
      $h:\A\to\R^d$ is a map such that
      $(a,p)\in\Pcal$ implies that
      $T(\bx)=\bx+h(a)$ for all $\bx\in p$,
        the map $h$ extends to a morphism of monoids $h:\A^*\to\R^d$ 
        satisfying $h(u\cdot v)=h(u)+h(v)$;
      $\bv\in\R^{d+1}$ defines a half space 
             $H_\bv=\{\bx\in\R^d\mid v_0 + \sum_{i=1}^{d}v_ix_i\geq0\}$ 
             such that $D\cap H_\bv=W$.
      \Function{InducedTransformation}{$T$, $\bv$}
        \State $(\Pcal,h)\gets T$
        \State $(\Pcal',\omega)\gets \Call{InducedPartition}{T,\bv,\Pcal}$
        \State $T'\gets (\Pcal',h\circ\omega)$
                                 \Comment{the induced transformation}
        \State\Return $(T',\omega)$ 
      \EndFunction
      \Ensure 
      $T'$ is a PET equal to the induced transformation $\widehat{T}|_W$ given as
      a pair $(\Pcal',h\circ\omega)$ where 
      $\{q\mid (b,q)\in\Pcal'\}$ is a partition of $W$ into convex polytopes,
      $\B=\{b\mid (b,q)\in\Pcal'\}$ is some alphabet,
        $h\circ\omega:\Bcal^*\to\R^d$ is a morphism of monoids such that
      $(b,q)\in\Pcal'$ implies that
      $T'(\bx)=\bx+h\circ\omega(b)$ for all $\bx\in q$;
      $\omega:\Bcal^*\to\A^*$ is a morphism of monoid satisfying the
      following equation for one-dimensional subshifts:
      $\Xcal_{\Pcal,T}=\overline{\omega\left(\Xcal_{\widehat{\Pcal}|_W,
                                         \widehat{T}|_W}\right)}^\sigma$.
  \end{algorithmic}
\end{algorithm}

In this section, we provide two algorithms to compute the induced partition and
induced transformation of a polytope exchange transformation on a sub-domain.
More precisely, Algorithm~\ref{alg:induced-partition} computes the induced
partition $\widehat{\Pcal}|_W$ of a partition $\Pcal$ of $D$ by a polytope
exchange transformation $T:D\to D$ on a sub-domain $W\subset D$.
It also computes the substitution $\omega$ allowing to express configuration
coded by the partition $\Pcal$ as the image of configurations
coded by the induced partition $\widehat{\Pcal}|_W$.
Algorithm~\ref{alg:induced-transformation} computes the induced
transformation $\widehat{T}|_W$ of a polytope exchange transformation $T:D\to
D$ on a sub-domain $W\subset D$.  We first present the algorithm computing the
induced partition since the other one can be deduced from it.

Algorithms are written for domain $D\subset\R^d$, with $d\geq 1$, as they work
in arbitrary dimension.  All polytopes manipulated in the algorithms are
assumed to be open so that the intersection of two polytopes is nonempty if and
only if it is of positive volume and convex so that they can be represented as a set of linear
inequalities.  Partitions are represented as a
list of pairs $(a,p)$ where $a$ is the index associated to the convex polytope
$p$. The same index can be used for different polytopes, thus allowing
to deal with polytope partitions where some
atoms are not convex. Non-convex atoms must be split
into many convex polytopes with the same index.
For example, the atom with index 5 in Figure~\ref{fig:2d-walk} is not convex.

We assume that the polytope exchange transformation $T:D\to D$ is defined on
domain $D\subset\R^d$ which is a convex polytope and that the sub-domain $W$ on
which the induction is being done is such that $W=D\cap H_\bv$ where
$\bv\in\R^{d+1}$ and $H_\bv$ is a half-space given by the part of $\R^d$ on one
side of a hyperplane:
\[
H_\bv=\left\{\bx\in\R^d \;\middle|\; v_0 + \sum_{i=1}^{d}v_ix_i\geq0\right\}.
\]
The reason is that in the algorithms, polytopes are intersected sometimes with
$H_\bv$ and sometimes with the complement $\R^d\setminus H_\bv$, that is, both sides of the
hyperplane given by $\bv$ and we want the result to be convex in both cases.
Inducing on more general polytope $W$ must be done in many steps (once for each
inequalities defining $W$).

When considering $\Z^d$-action $R$ on a polytope $D\subset\R^d$ that is Cartesian on
a window $W\subset D$, it is possible to compute the induced $\Z^d$-action
$\widehat{R}|_W$ by considering each subaction $\widehat{R^{\be_i}}|_W$
individually for $i\in\{1,\dots,d\}$.
In the next sections,
we use the algorithms when
the return times to $W\subset D$ 
under $R^{\be_j}$ 
is $1$ for all $\bx\in W$ and for every direction $\be_j$ except for some $j=i$.
In that case, 
the set of return words
$\Lcal=\scReturnWord(W)$ is a set of one-dimensional words in the direction
$\be_i$ for some $i$ and the substitution $\omega$ is a $d$-dimensional morphism.

When the return time to the window $W$ is not bounded,
Algorithm~\ref{alg:induced-partition} does not halt. In this case, an
approximation of the induced partition can be obtained by replacing line
\ref{algo:line:while-not-empty} by
\textbf{while} $\sum_{(u,q)\in\Qcal}volume(q)<\varepsilon$ \textbf{do}
for some small $\varepsilon>0$ or by a fixed number of iterations. In the examples
considered in this article, the return time is always bounded.

Algorithm~\ref{alg:induced-partition} and
Algorithm~\ref{alg:induced-transformation} are implemented
in the module on PETs of the
optional package \texttt{slabbe} \cite{labbe_slabbe_0_6_2_2020}
for SageMath \cite{sagemathv9.2}.

\part{Substitutive structure of $\Xcal_{\Pcal_0,R_0}$}
\label{part:subst-struct-XP0R0}

In this part, we start with the partition 
introduced in \cite{labbe_markov_2021} shown in Figure~\ref{fig:2d-walk}
and construct a sequence of induced partition until the induction
process loops.
The proofs can be verified with a computer as we provide the SageMath code to
reproduce the computations of all induced $\Z^2$-actions, induced
partitions and 2-dimensional morphisms.

\section{Inducing the $\Z^2$-action $R_0$ and the partition $\Pcal_0$ 
to get $\Xcal_{\Pcal_1,R_1}$}

Let $\Gamma_0=\langle (\varphi,0), (1,\varphi+3) \rangle_\Z$ be a lattice
with $\varphi=\frac{1+\sqrt{5}}{2}$.
We consider the dynamical system $(\R^2/\Gamma_0,\Z^2,R_0)$ defined by the
$\Z^2$-action
\[
\begin{array}{rccl}
    R_0:&\Z^2\times\R^2/\Gamma_0 & \to & \R^2/\Gamma_0\\
    &(\bn,\bx) & \mapsto &\bx+\bn.
\end{array}
\]
We consider the topological partition $\Pcal_0$ of $\R^2/\Gamma_0$
illustrated in Figure~\ref{fig:JR-partition-0} on a fundamental domain
where each atom is associated to a letter in 
the alphabet $\A_0=\Zrange{10}$.
The partition $\Pcal_0$
defines a coding map $\sccode_0:\R^2/\Gamma_0\to\A_0$
and it was proved in \cite{labbe_markov_2021} that
the coding of orbits under the $\Z^2$-action $R_0$
are valid Wang configurations using the 11 Wang tiles from Jeandel-Rao tile set.

\begin{figure}[h]
\begin{center}
    \[
    \begin{array}{l}
        \Pcal_0=
    \end{array}
    \begin{array}{c} 
\begin{tikzpicture}[auto,scale=1.5,every node/.style={scale=0.8}]
\def\horizontalaxis{
\draw (-.05,0) -- (\p+.05,0);
\foreach \x/\y in
{0/0, 1/1,
    2-\p  /\frac{1}{\varphi^{2}},
    \p-1  /\frac{1}{\varphi},
    \p    /\varphi}
\draw (\x,0+.05) -- (\x,0-.05) node[below] {$\y$};}
\def\silentverticalaxis{
\draw[-latex] (0,-.05) -- (0,3+\p+.05);
\foreach \x/\y in
{0/0, 1/1, 2/2,
    \p+1  /\varphi+1,
    \p+2  /\varphi+2,
    \p+3  /\varphi+3}
\draw (.05,\x) -- (-.05,\x);}
\def\verticalaxis{
\draw[-latex] (0,-.05) -- (0,3+\p+.05);
\foreach \x/\y in
{0/0, 1/1, 2/2,
    \p+1  /\varphi+1,
    \p+2  /\varphi+2,
    \p+3  /\varphi+3}
\draw (.05,\x) -- (-.05,\x) node[left] {$\y$};}

\begin{scope}[xshift=0cm]

    \draw[thick] (0,2+\p) -- (2-\p,3+\p) -- (2-\p,2+\p) -- (1,3+\p)-- (1,2+\p) -- (0,2);
    \draw[thick] (0,1) -- (\p,1) -- (\p,2) -- (1,1) -- (1,2) 
    -- (\p,2+\p) -- (1,1+\p) -- (1,2+\p) -- cycle;
    \draw[thick] (0,0) rectangle (\p,3+\p);
    \draw[thick] (\p-1,2) -- (\p-1,1) -- (\p,2+\p) -- (0,1) -- (\p,1);
    \draw[thick] (0,2) -- (1,3+\p) -- (1,2+\p) -- (\p,3+\p);
    \draw[thick] (0,1+\p) -- (\p-1,2+\p);
    \draw[thick] (0,1) -- (0,0) -- (\p-1,1) --
                (\p-1,0) -- (1,1) --
                (1,0) -- (\p,1) -- (\p,2+\p) -- cycle;
    \draw[thick] (0,2+\p) -- (2-\p,3+\p) -- (2-\p,2+\p) --
                        (1,3+\p) -- (1,2+\p) -- (0,1);
    \draw[thick] (1,1+\p) -- (1,2+\p) -- (\p,3+\p) -- (\p,2+\p) -- cycle;
    \draw[thick]  (0,0) -- (0,1) -- (0,0) -- (\p-1,1) --
                (\p-1,0) -- (1,1) --
                (1,0) -- (\p,1);
    \draw[thick] (0,2) -- (0,1) -- (\p-1,2) --
                (\p-1,1) -- (1,2) --
                (1,1) -- (\p,2) -- (\p,3+\p) -- cycle;
    \draw[thick] (0,2) -- (1,3+\p) -- (1,2+\p);
    \draw[thick] (0,0) rectangle (\p,3+\p);

    \horizontalaxis
    \verticalaxis

\node at (0.14,4.36) {$6$}; 
\node at (0.64,4.36) {$6$}; 
\node at (1.22,4.36) {$6$}; 
\node at (1.29,3.62) {$7$}; 
\node at (0.24,3.43) {$5$}; 
\node at (0.77,3.62) {$4$}; 
\node at (0.14,2.65) {$2$}; 
\node at (0.81,2.65) {$10$}; 
\node at (0.82,2.00) {$8$}; 
\node at (0.18,2.00) {$7$}; 
\node at (1.26,2.00) {$3$}; 
\node at (0.37,1.25) {$9$}; 
\node at (0.84,1.25) {$9$}; 
\node at (1.42,1.25) {$9$}; 
\node at (0.23,0.75) {$1$}; 
\node at (0.77,0.75) {$1$}; 
\node at (1.26,0.75) {$1$}; 
\node at (0.37,0.25) {$0$}; 
\node at (0.84,0.25) {$0$}; 
\node at (1.42,0.25) {$0$}; 
\end{scope}
\end{tikzpicture}
    \end{array}
    \qquad
    \begin{array}{l}
        R_0^\bn(\bx)=\bx+\bn\quad \bmod\Gamma_0
    \end{array}
    \]
\end{center}
    \caption{The partition $\Pcal_0$ of $\R^2/\Gamma_0$ 
    into atoms associated to letters in $\A_0$
    illustrated on a rectangular fundamental domain.}
    \label{fig:JR-partition-0}
\end{figure}
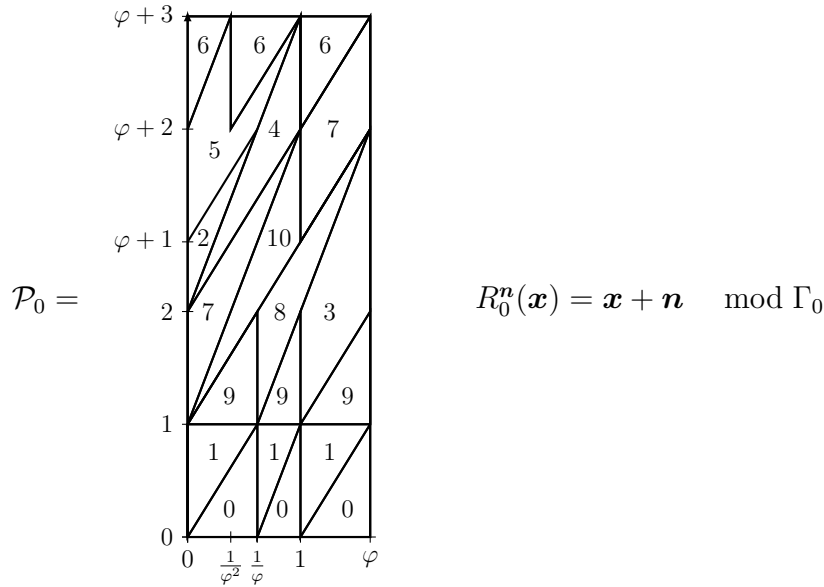

\begin{proposition}\label{prop:desubstitute-JR-sticks}
    Let $(\R^2/\Gamma_1,\Z^2,R_1)$ be a dynamical system
    with lattice $\Gamma_1=\varphi\Z\times\Z$
and $\Z^2$-action
\[
\begin{array}{rccl}
    R_1:&\Z^2\times\R^2/\Gamma_1 & \to & \R^2/\Gamma_1\\
    &(\bn,\bx) & \mapsto &\bx+n_1\be_1+n_2(\varphi^{-1},\varphi^{-2}).
\end{array}
\]
Let $\Pcal_1$ be the topological partition 
illustrated in Figure~\ref{fig:JR-partition-1}
where each atom is associated to a letter in 
$\A_1=\Zrange{27}$.
Then $\Xcal_{\Pcal_0,R_0}=\overline{\beta_0(\Xcal_{\Pcal_1,R_1})}^\sigma$
    where $\beta_0:\A_1^{*^2}\to \A_0^{*^2}$
    is the $2$-dimensional morphism defined by
    \begin{equation}\label{eq:beta-JR-w0-4}
        \beta_0:\left\{
            \setlength{\arraycolsep}{0pt}
            \footnotesize
            \begin{array}{lllllll}
0\mapsto \left(\begin{array}{r}
7 \\
3 \\
9 \\
0
\end{array}\right)
,&
1\mapsto \left(\begin{array}{r}
5 \\
7 \\
9 \\
0
\end{array}\right)
,&
2\mapsto \left(\begin{array}{r}
4 \\
10 \\
9 \\
0
\end{array}\right)
,&
3\mapsto \left(\begin{array}{r}
7 \\
3 \\
3 \\
1
\end{array}\right)
,&
4\mapsto \left(\begin{array}{r}
6 \\
7 \\
3 \\
1
\end{array}\right)
,&
5\mapsto \left(\begin{array}{r}
7 \\
8 \\
3 \\
1
\end{array}\right)
,&
6\mapsto \left(\begin{array}{r}
5 \\
2 \\
7 \\
1
\end{array}\right)
,\\
7\mapsto \left(\begin{array}{r}
5 \\
5 \\
7 \\
1
\end{array}\right)
,&
8\mapsto \left(\begin{array}{r}
6 \\
5 \\
7 \\
1
\end{array}\right)
,&
9\mapsto \left(\begin{array}{r}
5 \\
7 \\
8 \\
1
\end{array}\right)
,&
10\mapsto \left(\begin{array}{r}
4 \\
10 \\
8 \\
1
\end{array}\right)
,&
11\mapsto \left(\begin{array}{r}
5 \\
4 \\
10 \\
1
\end{array}\right)
,&
12\mapsto \left(\begin{array}{r}
6 \\
4 \\
10 \\
1
\end{array}\right)
,&
13\mapsto \left(\begin{array}{r}
7 \\
3 \\
3 \\
9 \\
0
\end{array}\right)
,\\
14\mapsto \left(\begin{array}{r}
6 \\
7 \\
3 \\
9 \\
0
\end{array}\right)
,&
15\mapsto \left(\begin{array}{r}
7 \\
8 \\
3 \\
9 \\
0
\end{array}\right)
,&
16\mapsto \left(\begin{array}{r}
5 \\
2 \\
7 \\
9 \\
0
\end{array}\right)
,&
17\mapsto \left(\begin{array}{r}
6 \\
2 \\
7 \\
9 \\
0
\end{array}\right)
,&
18\mapsto \left(\begin{array}{r}
5 \\
5 \\
7 \\
9 \\
0
\end{array}\right)
,&
19\mapsto \left(\begin{array}{r}
6 \\
5 \\
7 \\
9 \\
0
\end{array}\right)
,&
20\mapsto \left(\begin{array}{r}
5 \\
7 \\
8 \\
9 \\
0
\end{array}\right)
,\\
21\mapsto \left(\begin{array}{r}
4 \\
10 \\
8 \\
9 \\
0
\end{array}\right)
,&
22\mapsto \left(\begin{array}{r}
6 \\
4 \\
10 \\
9 \\
0
\end{array}\right)
,&
23\mapsto \left(\begin{array}{r}
6 \\
7 \\
3 \\
3 \\
1
\end{array}\right)
,&
24\mapsto \left(\begin{array}{r}
6 \\
7 \\
8 \\
3 \\
1
\end{array}\right)
,&
25\mapsto \left(\begin{array}{r}
6 \\
5 \\
2 \\
7 \\
1
\end{array}\right)
,&
26\mapsto \left(\begin{array}{r}
6 \\
4 \\
10 \\
8 \\
1
\end{array}\right)
,&
27\mapsto \left(\begin{array}{r}
6 \\
5 \\
4 \\
10 \\
1
\end{array}\right)
.
\end{array}\right.
    \end{equation}
\end{proposition}

\begin{proof}
    Let $W=(0,\varphi)\times(0,1)$.
The $\Z^2$-action $R_0$ is Cartesian on the window $W$.
Thus from Lemma~\ref{lem:Cartesian-induced-action},
the induced $\Z^2$-action $\widehat{R_0}|_W:\Z^2\times W\to W$ 
given by
\[
     (\widehat{R_0}|_W)^\bn(\bx)
    =(\widehat{R_0^{\be_1}}|_W)^{n_1}\circ
     (\widehat{R_0^{\be_2}}|_W)^{n_2}(\bx),
\]
for every $\bn=(n_1,n_2)\in\Z^2$, is well-defined.
    The first return time to $W$ under $R_0$ is $(1,4)$ or $(1,5)$.
    The orbits of points in $W$ until the first return to $W$ 
    under $R_0^{\be_2}$ is shown in Figure~\ref{fig:R0-vertical-induction}.

\begin{figure}[h]
\begin{center}
\begin{tikzpicture}[auto,scale=2.0,every node/.style={scale=0.7}]
\def\verticalaxis{
    \draw[-latex] (0,0) -- (0,3+\p+.2);
    \foreach \x/\y in
    {0/0, 1/1, 2/2, 3/3, 4/4,
        \p-1  /\varphi-1,
        \p    /\varphi,
        \p+1  /\varphi+1,
        \p+2  /\varphi+2,
        \p+3  /\varphi+3}
    \draw (.05,\x) -- (-.05,\x) node[left] {$\y$};}
\def\shortverticalaxis{
    \draw[-latex] (0,0) -- (0,1+.2);
    \foreach \x/\y in
    {0/0, 1/1, 
        2-\p  /\frac{1}{\varphi^2},
        }
    \draw (.05,\x) -- (-.05,\x) node[left] {$\y$};}
\def\horizontalaxis{
    \draw[-latex] (0,0) -- (\p+.1,0);
    \foreach \x/\y in
    {0/0, 1/1,
        \p-1  /\frac{1}{\varphi},
        \p    /\varphi}
    \draw (\x,.05) -- (\x,-.05) node[below] {$\y$};}
\begin{scope}[xshift=0.0cm]
    \verticalaxis
    \horizontalaxis
    \draw (0,0) rectangle (\p,3+\p); 
    \draw (1,0) -- (1, 3+\p);
    \draw (0,1) -- (\p, 1);
    \draw (0,2) -- (\p, 2);
    \draw (0,3) -- (\p, 3);
    \draw (0,4) -- (\p, 4);
    \draw (0,\p-1) -- (\p, \p-1);
    \draw (0,\p+0) -- (\p, \p+0);
    \draw (0,\p+1) -- (\p, \p+1);
    \draw (0,\p+2) -- (\p, \p+2);
    \node at (.5,   0.81) {$A$};
    \node at (1.32, 0.81) {$B$};
    \node at (.5,   0.31) {$C$};
    \node at (1.32, 0.31) {$D$};
    \node at (.5,   1.81) {$R_0^{(0,1)}(A)$};
    \node at (1.32, 1.81) {$R_0^{(0,1)}(B)$};
    \node at (.5,   1.31) {$R_0^{(0,1)}(C)$};
    \node at (1.32, 1.31) {$R_0^{(0,1)}(D)$};
    \node at (.5,   2.81) {$R_0^{(0,2)}(A)$};
    \node at (1.32, 2.81) {$R_0^{(0,2)}(B)$};
    \node at (.5,   2.31) {$R_0^{(0,2)}(C)$};
    \node at (1.32, 2.31) {$R_0^{(0,2)}(D)$};
    \node at (.5,   3.81) {$R_0^{(0,3)}(A)$};
    \node at (1.32, 3.81) {$R_0^{(0,3)}(B)$};
    \node at (.5,   3.31) {$R_0^{(0,3)}(C)$};
    \node at (1.32, 3.31) {$R_0^{(0,3)}(D)$};
    \node at (.5,   4.31) {$R_0^{(0,4)}(C)$};
    \node at (1.32, 4.31) {$R_0^{(0,4)}(D)$};
\end{scope}

\begin{scope}[xshift=2.5cm]
    \shortverticalaxis
    \horizontalaxis
    \draw (0,0) rectangle (\p,1); 
    \draw (\p-1,0) -- (\p-1, 1);
    \draw (0,2-\p) -- (\p, 2-\p);
    \node at (.31,  0.15) {$R_0^{(0,4)}(B)$};
    \node at (1.11, 0.15) {$R_0^{(0,4)}(A)$};
    \node at (.31,  0.69) {$R_0^{(0,5)}(D)$};
    \node at (1.11, 0.69) {$R_0^{(0,5)}(C)$};
\end{scope}
\end{tikzpicture}
\end{center}
\caption{The set $\{A,B,C,D\}$ forms a partition of the rectangle 
    $W=(0,\varphi)\times(0,1)$.
The return time to $W$ under the map 
$R_0^{\be_2}$ is 4 or 5. The orbit of $A$, $B$, $C$ and $D$ under $R_0^{\be_2}$
before it returns to $W$
yields a partitions of $\R^2/\Gamma_0$.
The first return map $\widehat{R_0^{\be_2}}|_W$ is equivalent to a toral translation by
the vector $(\frac{1}{\varphi},\frac{1}{\varphi^2})$ on
$\R^2/\Gamma_1$.}
\label{fig:R0-vertical-induction}
\end{figure}
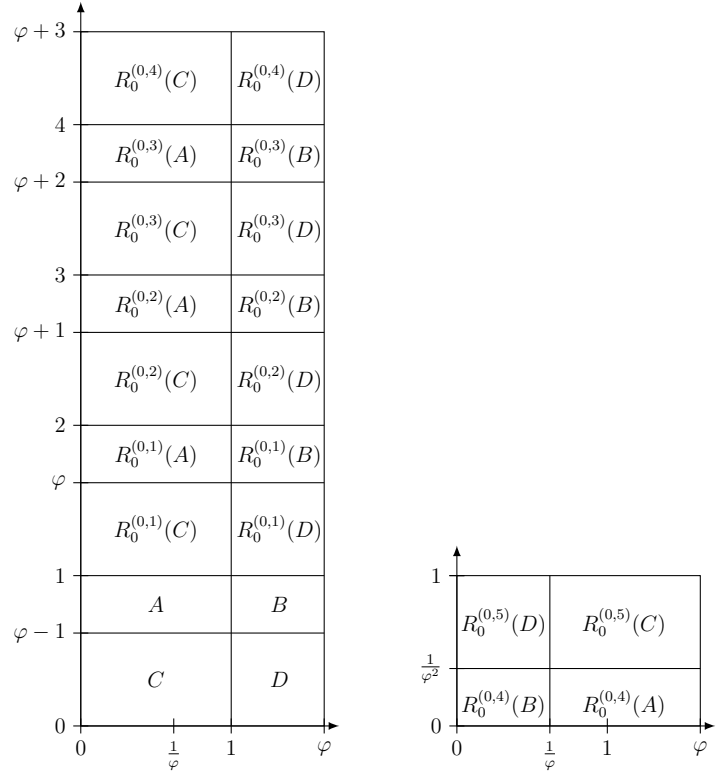

    The maximal subsets of $W$ on which the first return map 
    $\widehat{R_0^{\be_2}}|_W$ is continuous yield a partition
    $\{A,B,C,D\}$ of $W$ defined as
    \begin{align*}
        &A = (0,1)\times(\varphi^{-1},1),
        &&B = (1,\varphi)\times(\varphi^{-1},1),\\
        &C = (0,1)\times(0,\varphi^{-1}),
        &&D = (1,\varphi)\times(0,\varphi^{-1}).
    \end{align*}
    The return time of $A\cup B$ to $W$ under $R_0^{\be_2}$ is 4.
    The return time of $C\cup D$ to $W$ under $R_0^{\be_2}$ is 5.
The first return maps $\widehat{R_0^{\be_2}}|_W$
is defined as exchanges of rectangles $\{A,B,C,D\}$
which, from Lemma~\ref{lem:PET-iff-toral-rotation},
is equivalent to a toral translation 
    by $(\varphi^{-1},\varphi^{-2})$ on $\R^2/\Gamma_1$.

The first return map $\widehat{R_0^{\be_1}}|_W$ is defined as an exchange of
two rectangles, see Figure~\ref{fig:R0-horizontal-induction},
which, from Lemma~\ref{lem:PET-iff-toral-rotation},
is equivalent to a toral translation by $\be_1$ on $\R^2/\Gamma_1$.
Thus, we may identify $W$ with the torus $\R^2/\Gamma_1$ and define the induced
$\Z^2$-action as toral rotations on $\R^2/\Gamma_1$:
\[
    (\widehat{R_0}|_W)^\bn(\bx) 
    = \bx+n_1\be_1+n_2(\varphi^{-1},\varphi^{-2})
    \quad\bmod\Gamma_1.
\]
We observe that $R_1=\widehat{R_0}|_W$.

\begin{figure}[h]
\begin{center}
\begin{tikzpicture}[auto,scale=2.0,every node/.style={scale=0.7}]
\def\verticalaxis{
    \draw[-latex] (0,0) -- (0,3+\p+.2);
    \foreach \x/\y in
    {0/0, 1/1, 2/2, 3/3, 4/4,
        \p-1  /\varphi-1,
        \p    /\varphi,
        \p+1  /\varphi+1,
        \p+2  /\varphi+2,
        \p+3  /\varphi+3}
    \draw (.05,\x) -- (-.05,\x) node[left] {$\y$};}
\def\shortverticalaxis{
    \draw[-latex] (0,0) -- (0,1+.2);
    \foreach \x/\y in
    {0/0, 1/1, 
        2-\p  /\frac{1}{\varphi^2},
        }
    \draw (.05,\x) -- (-.05,\x) node[left] {$\y$};}
\def\horizontalaxis{
    \draw[-latex] (0,0) -- (\p+.1,0);
    \foreach \x/\y in
    {0/0, 1/1,
        \p-1  /\frac{1}{\varphi},
        \p    /\varphi}
    \draw (\x,.05) -- (\x,-.05) node[below] {$\y$};}
\begin{scope}[xshift=0.0cm,yshift=-1.7cm]
    \shortverticalaxis
    \horizontalaxis
    \draw (0,0) rectangle (\p,1); 
    \draw (\p-1,0) -- (\p-1, 1);
    \node at (.31,  0.5) {$L$};
    \node at (1.11, 0.5) {$R$};
\end{scope}
\begin{scope}[xshift=2.5cm,yshift=-1.7cm]
    \shortverticalaxis
    \horizontalaxis
    \draw (0,0) rectangle (\p,1); 
    \draw (1,0) -- (1, 1);
    \node at (0.5,  0.5) {$R_0^{\be_1}(R)$};
    \node at (1.31, 0.5) {$R_0^{\be_1}(L)$};
\end{scope}
\end{tikzpicture}
\end{center}
\caption{
The return time to $W$ under map $R_0^{\be_1}$ is always 1.
The first return map $\widehat{R_0^{\be_1}}|_W$ is equivalent to a toral translation by the vector $(1,0)$ on $\R^2/\Gamma_1$.
}
\label{fig:R0-horizontal-induction}
\end{figure}
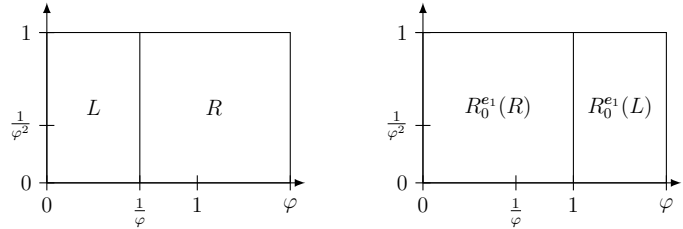

Thus we have that the return word function 
$\scReturnWord:W  \to  \Zrange{10}^{*^2}$ is
\[
    \bx\mapsto
\left(\begin{array}{c}
        \sccode_0(R_0^{(0,3)}\bx) \\ 
        \sccode_0(R_0^{(0,2)}\bx) \\ 
        \sccode_0(R_0^{(0,1)}\bx) \\ 
        \sccode_0(R_0^{(0,0)}\bx)
    \end{array}\right)
    \text{if } \bx\in A\cup B
    \quad\text{ or }
    \quad
    \bx\mapsto
\left(\begin{array}{c}
        \sccode_0(R_0^{(0,4)}\bx) \\ 
        \sccode_0(R_0^{(0,3)}\bx) \\ 
        \sccode_0(R_0^{(0,2)}\bx) \\ 
        \sccode_0(R_0^{(0,1)}\bx) \\ 
        \sccode_0(R_0^{(0,0)}\bx)
    \end{array}\right)
    \text{if } \bx\in C\cup D.
\]
We get the set of return words
$\Lcal=\{w_b\}_{b\in\Zrange{27}}=\scReturnWord(W)$
as listed in \eqref{eq:beta-JR-w0-4}.
The partition induced by $R_0$ on $W$ is
\[
    \Pcal_1:=\widehat{\Pcal_0}|_W = \{\scReturnWord^{-1}(w_b)\}_{b\in\Zrange{27}}
\]
which is a topological partition of $W$ 
made of 28 atoms (the atoms $\scReturnWord^{-1}(w_{19})$ and
$\scReturnWord^{-1}(w_{22})$ are both the union of two triangles), see
Figure~\ref{fig:JR-partition-1}.

\begin{figure}[h]
\begin{center}
    \[
    \begin{array}{l}
        \Pcal_1=
    \end{array}
    \begin{array}{cc} 
    \includegraphics{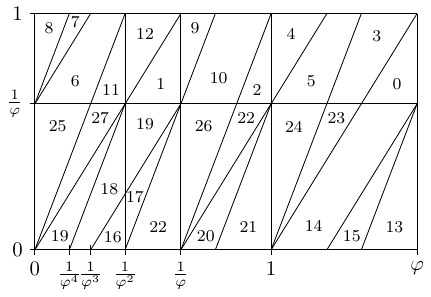}\\
    \end{array}
    \quad
    \begin{array}{ll}
        R_1^\bn(\bx)=&\bx+n_1\be_1\\&+n_2(\varphi^{-1},\varphi^{-2})
        \quad\bmod\Gamma_1
    \end{array}
    \]
\end{center}
    \caption{The partition $\Pcal_1:=\widehat{\Pcal_0}|_W$ of 
    $\R^2/\Gamma_1$ into 30 convex atoms
    each associated to one of the 28 letters in $\A_1$ 
    (indices 19 and 22 are both used twice).}
    \label{fig:JR-partition-1}
\end{figure}

The induced coding is
\[
\begin{array}{rccl}
    \sccode_1:=\sccode_0|_W:&W & \to & \Zrange{27}\\
    &\by & \mapsto & b \quad \text{ if and only if } 
                \by\in \scReturnWord^{-1}(w_b).
\end{array}
\]
A natural substitution comes out of this construction:
\[
\begin{array}{rccc}
    \beta_0:&\Zrange{27} & \to & \Zrange{10}^*\\
    &b & \mapsto & w_b
\end{array}
\]
which corresponds to the one defined in Equation~\eqref{eq:beta-JR-w0-4}.
From Proposition~\ref{prop:orbit-preimage}, we have
$\Xcal_{\Pcal_0,R_0}=\overline{\beta_0(\Xcal_{\Pcal_1,R_1})}^\sigma$
which ends the proof.
\end{proof}

\begin{remark}[SageMath Code]
The following allows to reproduce the proof of
    Proposition~\ref{prop:desubstitute-JR-sticks} using 
SageMath \cite{sagemathv9.2} with 
the optional package \texttt{slabbe} \cite{labbe_slabbe_0_6_2_2020}.
First we construct the golden mean as a element of a quadratic number field
because it is more efficient for arithmetic operations and comparisons:
{\footnotesize
\begin{verbatim}
sage: z = polygen(QQ, "z")
sage: K = NumberField(z**2-z-1, "phi", embedding=RR(1.6))
sage: phi = K.gen()
\end{verbatim}}
\noindent
We import the polygon partition $\Pcal_0$ of $\R^2/\Gamma_0$ which is predefined
    in \texttt{slabbe}:
{\footnotesize
\begin{verbatim}
sage: from slabbe.arXiv_1903_06137 import jeandel_rao_wang_shift_partition
sage: P0 = jeandel_rao_wang_shift_partition()
\end{verbatim}}
\noindent
We import polyhedron exchange transformations from the package:
{\footnotesize
\begin{verbatim}
sage: from slabbe import PolyhedronExchangeTransformation as PET
\end{verbatim}}
\noindent
    We define the lattice $\Gamma_0$ and the maps $R_0^{\be_1}$, $R_0^{\be_2}$
    which can be seen as a polygon exchange transformations on a rectangular
    fundamental domain of $\R^2/\Gamma_0$:
{\footnotesize
\begin{verbatim}
sage: Gamma0 = matrix.column([(phi,0), (1,phi+3)])
sage: fundamental_domain = polytopes.parallelotope([(phi,0), (0,phi+3)])
sage: R0e1 = PET.toral_translation(Gamma0, vector((1,0)), fundamental_domain)
sage: R0e2 = PET.toral_translation(Gamma0, vector((0,1)), fundamental_domain)
\end{verbatim}}
\noindent
We compute the induced partition $\Pcal_1$ of $\R^2/\Gamma_1$,
the substitution $\beta_0$ and
the induced $\Z^2$-action $R_1$:
{\footnotesize
\begin{verbatim}
sage: y_le_1 = [1, 0, -1]   # syntax for the inequality y <= 1
sage: P1,beta0 = R0e2.induced_partition(y_le_1, P0, substitution_type="column")
sage: R1e1,_ = R0e1.induced_transformation(y_le_1)
sage: R1e2,_ = R0e2.induced_transformation(y_le_1)
\end{verbatim}}
\end{remark}

\begin{remark}[SageMath Code]\label{rem:sagemath-plot}
To observe and verify the computed induced partitions, substitutions and
induced $\Z^2$-action, one can do:
{\footnotesize
\begin{verbatim}
sage: P1.plot()                     # or P1.tikz().pdf()
sage: show(beta0)
sage: R1e1
Polyhedron Exchange Transformation of Polyhedron partition of 2 atoms with 2 letters
with translations {0: (1, 0), 1: (-phi + 1, 0)}
sage: R1e1.plot()
sage: R1e2
Polyhedron Exchange Transformation of Polyhedron partition of 4 atoms with 4 letters
with translations {0: (phi - 1, -phi + 1), 1: (-1, -phi + 1), 
                   2: (phi - 1, -phi + 2), 3: (-1, -phi + 2)}
sage: R1e2.plot()
\end{verbatim}}
\end{remark}

\section{Changing the base of the $\Z^2$-action to get $\Xcal_{\Pcal_2,R_2}$}

Recall that $\Gamma_1=\varphi\Z\times\Z$ and that the coding on the torus
$\R^2/\Gamma_1$ is given by the $\Z^2$-action $R_1$ of
    $(\bn,\bx) \mapsto \bx+n_1\be_1+n_2(\varphi^{-1},\varphi^{-2})$.
We remark that $R_1^{\be_1}$ is a horizontal translation 
while $R_1^{\be_2}$ is not.
But we observe that $R_1^{\be_1+\be_2}=R_1^{\be_1}\circ R_1^{\be_2}$ 
is a vertical translation 
$\bx\mapsto\bx+(0,\varphi^{-2})$.
So it is natural to change the base of the $\Z^2$-action $R_1$
by using $R_1^{\be_1+\be_2}$ instead of $R_1^{\be_2}$.
Therefore we let
    \[
        R_2^{(n_1,n_2)}(\bx) 
        = R_1^{(n_1+n_2,n_2)}(\bx)
        = \bx+n_1\be_1+\varphi^{-2}n_2\be_2
    \]
for every $\bn=(n_1,n_2)\in\Z^2$
which defines an action $\Z^2\times\R^2/\Gamma_2\to\R^2/\Gamma_2$
where $\Gamma_2=\Gamma_1$.
We keep the partition as is, that is $\A_2=\A_1$,
$\Pcal_2=\Pcal_1$
and $\sccode_2=\sccode_1$.
We get the setup illustrated in Figure~\ref{fig:JR-partition-2}.
We show in this section that this base change corresponds to a shear conjugacy
    $\beta_1:\Xcal_{\Pcal_2,R_2}\to \Xcal_{\Pcal_1,R_1}$.

\begin{remark}[SageMath Code]
    The following defines $\Pcal_2$, $R_2^{\be_1}$ and $R_2^{\be_2}$ in SageMath:
{\footnotesize
\begin{verbatim}
sage: P2 = P1
sage: R2e1 = R1e1
sage: R2e2 = (R1e1 * R1e2).merge_atoms_with_same_translation()
\end{verbatim}}
\end{remark}

\begin{figure}[h]
\begin{center}
    \[
    \begin{array}{l}
        \Pcal_2=
    \end{array}
    \begin{array}{cc} 
    \includegraphics{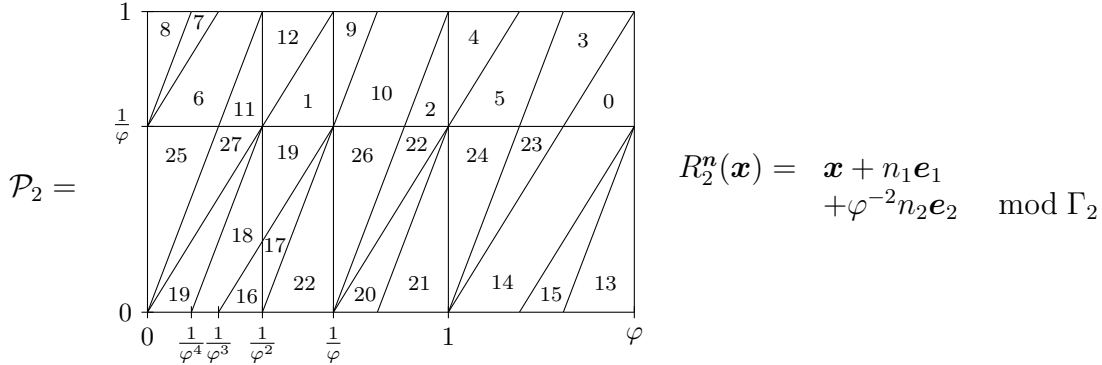}\\
    \end{array}
    \begin{array}{ll}
        R_2^\bn(\bx)=&\bx+n_1\be_1\\&+\varphi^{-2}n_2\be_2\quad\bmod\Gamma_2
    \end{array}
    \]
\end{center}
    \caption{The partition $\Pcal_2$ of 
    $\R^2/\Gamma_2$ into atoms
    associated to letters in $\A_2$.}
    \label{fig:JR-partition-2}
\end{figure}

In general we have the following result.
\begin{lemma}\label{lem:GLd-conjugacy}
Let $\A$ be some finite alphabet.
For every $M\in\GL_d(\Z)$, the map
\[
\begin{array}{rccl}
    \theta_M:&\A^{\Z^d} & \to & \A^{\Z^d} \\
    & w & \mapsto & (\bn\mapsto w({M^{-1}\bn}))
\end{array}
\]
is a $\GL_d(\Z)$-conjugacy of the full shift $(\A^{\Z^d},\Z^d,\sigma)$
satisfying
$\sigma^{M\bk}\circ\theta_M=\theta_M\circ\sigma^\bk$
for all $\bk\in\Z^d$.
\end{lemma}

\begin{proof}
    The map $\theta_M$ is continuous and admits an inverse $\theta_{M^{-1}}$
    which is also continuous.
    If $w\in\A^{\Z^d}$ and $\bk,\bn\in\Z^d$, then
    \begin{align*}
        \left[
            \sigma^{M\bk}\circ\theta_M(w)
            \right](\bn)
        &= \left[ \theta_M(w) \right](\bn+M\bk)
        = w(M^{-1}(\bn+M\bk))
        = w(M^{-1}\bn+\bk)\\
        &= \left[ \sigma^{\bk}w \right](M^{-1}\bn)
        = \left[ \theta_M(\sigma^{\bk}w) \right](\bn)
        = \left[ \theta_M\circ \sigma^{\bk}(w) \right](\bn).
    \end{align*}
    Hence $\theta_M$ is a $\GL_d(\Z)$-conjugacy of $(\A^{\Z^d},\Z^d,\sigma)$.
\end{proof}

Let 
$\beta_1=\theta_M:\A_2^{\Z^2}\to\A_1^{\Z^2}$ be the map
defined by the shear matrix
$M = \left(\begin{smallmatrix}
        1 & 1 \\
        0 & 1
\end{smallmatrix}\right)$
where $\A_1=\A_2=\Zrange{27}$.

\begin{lemma}\label{lem:beta1}
    The map $\beta_1$ is a shear conjugacy $\beta_1:\Xcal_{\Pcal_2,R_2}\to \Xcal_{\Pcal_1,R_1}$
    satisfying
    $\sigma^{M\bk}\circ\beta_1=\beta_1\circ\sigma^\bk$
    for every $\bk\in\Z^d$
    with
$M=\left(\begin{smallmatrix}
        1 & 1 \\
        0 & 1
\end{smallmatrix}\right)$.
\end{lemma}

\begin{proof}
    From Lemma~\ref{lem:GLd-conjugacy},
    $\beta_1$ is a shear-conjugacy
    $(\A_2^{\Z^2},\Z^2,\sigma)\to(\A_1^{\Z^2},\Z^2,\sigma)$ satisfying
    $\sigma^{M\bk}\circ\beta_1=\beta_1\circ\sigma^\bk$
    for all $\bk\in\Z^d$.
    The restriction of the map $\beta_1$ on the domain
    $\Xcal_{\Pcal_2,R_2}$ is a shear conjugacy
    onto its image. It remains to show that the image satisfies
    $\beta_1(\Xcal_{\Pcal_2,R_2})=\Xcal_{\Pcal_1,R_1}$.

    For every $\bx\in\R^2/\Gamma_1=\R^2/\Gamma_2$ and $m,n\in\Z$, we have
    \begin{align*}
        \beta_1\left(\scConfig^{\Pcal_2,R_{2}}_{\bx}(m,n)\right)
        &= \scConfig^{\Pcal_2,R_{2}}_{\bx}(m-n,n)
        = \sccode_2(R_2^{(m-n,n)}(\bx))\\
        &= \sccode_1(R_1^{(m,n)}(\bx))
        = \scConfig^{\Pcal_1,R_1}_{\bx}(m,n)
    \end{align*}
    Therefore
    \begin{align*}
    \beta_1(\Xcal_{\Pcal_2,R_2})
        &= \beta_1\left(\overline{\left\{\scConfig^{\Pcal_2,R_2}_\bx\mid
                       \bx\in\R^2/\Gamma_2\setminus \Delta_{\Pcal_2,R_2}\right\}}\right)\\
        &= \overline{\left\{\beta_1\left(\scConfig^{\Pcal_2,R_2}_\bx\right)\mid
                       \bx\in\R^2/\Gamma_2\setminus \Delta_{\Pcal_2,R_2}\right\}}\\
        &= \overline{\left\{\scConfig^{\Pcal_1,R_1}_\bx\mid
                       \bx\in\R^2/\Gamma_1\setminus \Delta_{\Pcal_1,R_1}\right\}}
        =\Xcal_{\Pcal_1,R_1}.
    \end{align*}
    Thus $\beta_1$ is a shear conjugacy $\Xcal_{\Pcal_2,R_2}\to \Xcal_{\Pcal_1,R_1}$.
\end{proof}

\section{Inducing the symbolic dynamical system $\Xcal_{\Pcal_2,R_2}$ to
get $\Xcal_{\Pcal_8,R_8}$ and $\Xcal_{\Pcal_\U,R_\U}$}

In this section, we induce the topological partition
$\Pcal_2$ until the process loops.
We need six induction steps before obtaining a topological partition $\Pcal_8$
which is self-induced, see Figure~\ref{fig:XP2R2-subst-structure}.

\begin{figure}[h]
\begin{center}
\begin{tikzpicture}[auto,xscale=2.00,every node/.style={scale=0.95}]
    \node (chi2) at (0,0) {$\Xcal_{\Pcal_2,R_2}$};
    \node (chi3) at (1,0) {$\Xcal_{\Pcal_3,R_3}$};
    \node (chi4) at (2,0) {$\Xcal_{\Pcal_4,R_4}$};
    \node (chi5) at (3,0) {$\Xcal_{\Pcal_5,R_5}$};
    \node (chi6) at (4,0) {$\Xcal_{\Pcal_6,R_6}$};
    \node (chi7) at (5,0) {$\Xcal_{\Pcal_7,R_7}$};
    \node (chi8) at (6,0) {$\Xcal_{\Pcal_8,R_8}$};
    \node (chi9) at (7,1) {$\Xcal_{\Pcal_9,R_9}$};
    \node (chi10) at (7,-1) {$\Xcal_{\Pcal_{10},R_{10}}$};
    \draw[to-,very thick] (chi2) to node {$\beta_2$} (chi3);
    \draw[to-,very thick] (chi3) to node {$\beta_3$} (chi4);
    \draw[to-,very thick] (chi4) to node {$\beta_4$} (chi5);
    \draw[to-,very thick] (chi5) to node {$\beta_5$} (chi6);
    \draw[to-,very thick] (chi6) to node {$\beta_6$} (chi7);
    \draw[to-,very thick] (chi7) to node {$\beta_7$} (chi8);
    \draw[to-,very thick] (chi8) to node {$\beta_8$} (chi9);
    \draw[to-,very thick] (chi9) to node {$\beta_9$} (chi10);
    \draw[to-,very thick] (chi10) to node {$\tau$} (chi8);
\end{tikzpicture}
\end{center}
    \caption{The substitutive structure of $\Xcal_{\Pcal_2,R_2}$.}
    \label{fig:XP2R2-subst-structure}
\end{figure}

\begin{proposition}\label{prop:fromcode2-to-code8}
    For every $i$ with $2\leq i\leq 9$,
	there exist
    a lattice $\Gamma_{i+1}$,
    an alphabet $\A_{i+1}$,
    a $\Z^2$-action $R_{i+1}:\Z^2\times \R^2/\Gamma_{i+1}\to\R^2/\Gamma_{i+1}$,
    a topological partition $\Pcal_{i+1}$ of $\R^2/\Gamma_{i+1}$
    and substitutions
    $\beta_i:\A_{i+1}\to\A_i^{*^2}$
    such that
    \[
        \Xcal_{\Pcal_i,R_i}=\overline{\beta_i(\Xcal_{\Pcal_{i+1},R_{i+1}})}^\sigma.
    \]
    Moreover, there exist a bijection $\tau:\A_{8}\to \A_{10}$
    such that $\Xcal_{\Pcal_{10},R_{10}}=\tau(\Xcal_{\Pcal_8,R_8})$.

    The partition $\Pcal_8$ is self-induced and
    $\Xcal_{\Pcal_8,R_8}$ is self-similar
    satisfying
    $\Xcal_{\Pcal_{8},R_{8}}=\overline{\beta_8\beta_9\tau(\Xcal_{\Pcal_8,R_8})}^\sigma$
    where $\beta_8\beta_9\tau$ is an expansive and primitive self-similarity.
\end{proposition}

The SageMath code allowing to reproduce the proof of
Proposition~\ref{prop:fromcode2-to-code8}, that is, to construct the partitions
$\Pcal_2$, $\Pcal_3$, $\Pcal_4$, $\Pcal_5$, $\Pcal_6$, $\Pcal_7$, $\Pcal_8$,
$\Pcal_9$, $\Pcal_{10}$ and show that $\Pcal_8$ and $\Pcal_{10}$ are equivalent
is embedded in the proof below.
To visualize the induced partitions, substitutions and induced PETs that are
computed, do as in Remark~\ref{rem:sagemath-plot}.

\begin{proof}
We proceed as in the proof of
Proposition~\ref{prop:desubstitute-JR-sticks} by making use of 
Proposition~\ref{prop:orbit-preimage} many times.
We use again and again the same induction process by producing at each step a
new partition until we reach a partition which is self-induced.

We start with the lattice $\Gamma_2=\langle (\varphi,0), (0,1) \rangle_\Z$, the
partition $\Pcal_2$, the coding map $\sccode_2:\R^2/\Gamma_2\to\A_2$, the
alphabet $\A_2=\Zrange{27}$ and $\Z^2$-action $R_2$ defined on
$\R^2/\Gamma_2$ as in Figure~\ref{fig:JR-partition-2}.

    We consider the window $W_2=(0,1)\times(0,1)$ as a subset of
    $\R^2/\Gamma_2$.
    The action $R_2$ is Cartesian on $W_2$.
    Thus from Lemma~\ref{lem:Cartesian-induced-action},
    $R_3:=\widehat{R_2}|_{W_2}:\Z^2\times W_2\to W_2$ is a well-defined
    $\Z^2$-action.
    From Lemma~\ref{lem:PET-iff-toral-rotation},
    the $\Z^2$-action $R_3$ can be seen as toral translation on
    $\R^2/\Gamma_3$ with $\Gamma_3=\Z^2$.
    Let $\Pcal_3=\widehat{\Pcal_2}|_{W_2}$ be the induced partition.
    From Proposition~\ref{prop:orbit-preimage}, then
    $\Xcal_{\Pcal_2,R_2}=
    \overline{\beta_2(\Xcal_{\Pcal_3,R_3})}^\sigma$.
    The partition $\Pcal_3$, the action $R_3$ and substitution
    $\beta_2$ are given below with alphabet $\A_3=\Zrange{19}$:
{\footnotesize
\begin{verbatim}
sage: x_le_1 = [1, -1, 0]   # syntax for x <= 1
sage: P3,beta2 = R2e1.induced_partition(x_le_1, P2, substitution_type="row")
sage: R3e1,_ = R2e1.induced_transformation(x_le_1)
sage: R3e2,_ = R2e2.induced_transformation(x_le_1)
\end{verbatim}}
    \[
    \begin{array}{c}
        \Pcal_3=
    \end{array}
    \begin{array}{ccc} 
    \includegraphics{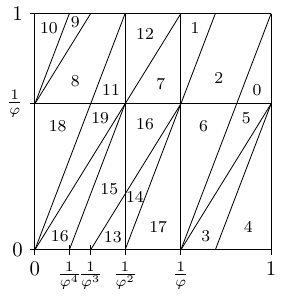}\\
    \end{array}
        \quad
        \begin{array}{c}
            R_3^\bn(\bx)=\\\bx+\varphi^{-2}\bn
        \end{array}
        \quad
    \begin{array}{c}\beta_{2}:\A_3\to \A_2^{*^2}\\[2mm]
        \left\{
            \setlength{\arraycolsep}{0pt}
            \footnotesize
    \begin{array}{ll}
0\mapsto \left(2\right)
,&
1\mapsto \left(9\right)
,\\
2\mapsto \left(10\right)
,&
3\mapsto \left(20\right)
,\\
4\mapsto \left(21\right)
,&
5\mapsto \left(22\right)
,\\
6\mapsto \left(26\right)
,&
7\mapsto \left(1,\,0\right)
,\\
8\mapsto \left(6,\,5\right)
,&
9\mapsto \left(7,\,4\right)
,\\
10\mapsto \left(8,\,4\right)
,&
11\mapsto \left(11,\,3\right)
,\\
12\mapsto \left(12,\,3\right)
,&
13\mapsto \left(16,\,15\right)
,\\
14\mapsto \left(17,\,15\right)
,&
15\mapsto \left(18,\,14\right)
,\\
16\mapsto \left(19,\,14\right)
,&
17\mapsto \left(22,\,13\right)
,\\
18\mapsto \left(25,\,24\right)
,&
19\mapsto \left(27,\,23\right)
.
\end{array}
    \right.
    \end{array}
    \]
    We consider the window $W_3=(0,\varphi^{-1})\times(0,1)$ as a subset of
    $\R^2/\Gamma_3$.
    The action $R_3$ is Cartesian on $W_3$.
    Thus from Lemma~\ref{lem:Cartesian-induced-action},
    $R_4:=\widehat{R_3}|_{W_3}:\Z^2\times W_3\to W_3$ is a well-defined
    $\Z^2$-action.
    From Lemma~\ref{lem:PET-iff-toral-rotation},
    the $\Z^2$-action $R_4$ can be seen as toral translation on
    $\R^2/\Gamma_4$ with $\Gamma_4=(\varphi^{-1}\Z)\times\Z$.
    Let $\Pcal_4=\widehat{\Pcal_3}|_{W_3}$ be the induced partition.
    From Proposition~\ref{prop:orbit-preimage}, then
    $\Xcal_{\Pcal_3,R_3}=
    \overline{\beta_3(\Xcal_{\Pcal_4,R_4})}^\sigma$.
    The partition $\Pcal_4$, the action $R_4$ and substitution
    $\beta_3$ are given below with alphabet $\A_4=\Zrange{19}$:
{\footnotesize
\begin{verbatim}
sage: x_le_phi_inv = [phi^-1, -1, 0]    # syntax for x <= phi^-1
sage: P4,beta3 = R3e1.induced_partition(x_le_phi_inv, P3, substitution_type="row")
sage: R4e1,_ = R3e1.induced_transformation(x_le_phi_inv)
sage: R4e2,_ = R3e2.induced_transformation(x_le_phi_inv)
\end{verbatim}}
\[
    \begin{array}{c}
        \Pcal_4=
    \end{array}
\begin{array}{c} 
\includegraphics{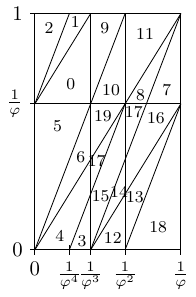}\\
\end{array}
        \quad
    \begin{array}{c}
        R_4^\bn(\bx)=\bx\\[2mm]
        +(-\varphi^{-3},0)n_1\\[2mm]
        +(0,\varphi^{-2})n_2
    \end{array}
        \quad
\begin{array}{c}
    \beta_{3}:\A_4\to \A_3^{*^2}\\[2mm]
    \left\{
            \setlength{\arraycolsep}{0pt}
            \footnotesize
        \begin{array}{ll}
0\mapsto \left(8\right)
,&
1\mapsto \left(9\right)
,\\
2\mapsto \left(10\right)
,&
3\mapsto \left(15\right)
,\\
4\mapsto \left(16\right)
,&
5\mapsto \left(18\right)
,\\
6\mapsto \left(19\right)
,&
7\mapsto \left(7,\,0\right)
,\\
8\mapsto \left(7,\,2\right)
,&
9\mapsto \left(8,\,1\right)
,\\
10\mapsto \left(11,\,2\right)
,&
11\mapsto \left(12,\,2\right)
,\\
12\mapsto \left(13,\,3\right)
,&
13\mapsto \left(14,\,3\right)
,\\
14\mapsto \left(15,\,5\right)
,&
15\mapsto \left(15,\,6\right)
,\\
16\mapsto \left(16,\,5\right)
,&
17\mapsto \left(16,\,6\right)
,\\
18\mapsto \left(17,\,4\right)
,&
19\mapsto \left(19,\,6\right)
.
\end{array}\right.
\end{array}
\]
    Now we consider the window $W_4=(0,\varphi^{-1})\times(0,\varphi^{-1})$ as
    a subset of $\R^2/\Gamma_4$.
    The action $R_4$ is Cartesian on $W_4$.
    Thus from Lemma~\ref{lem:Cartesian-induced-action},
    $R_5:=\widehat{R_4}|_{W_4}:\Z^2\times W_4\to W_4$ is a well-defined
    $\Z^2$-action.
    From Lemma~\ref{lem:PET-iff-toral-rotation},
    the $\Z^2$-action $R_5$ can be seen as toral translation on
    $\R^2/\Gamma_5$ with $\Gamma_5=(\varphi^{-1}\Z)\times(\varphi^{-1}\Z)$.
    Let $\Pcal_5=\widehat{\Pcal_4}|_{W_4}$ be the induced partition.
    From Proposition~\ref{prop:orbit-preimage}, then
    $\Xcal_{\Pcal_4,R_4}=
    \overline{\beta_4(\Xcal_{\Pcal_5,R_5})}^\sigma$.
    The partition $\Pcal_5$, the action $R_5$ and substitution
    $\beta_4$ are given below with alphabet $\A_5=\Zrange{21}$:
{\footnotesize
\begin{verbatim}
sage: y_le_phi_inv = [phi^-1, 0, -1]    # syntax for y <= phi^-1
sage: P5,beta4 = R4e2.induced_partition(y_le_phi_inv, P4, substitution_type="column")
sage: R5e1,_ = R4e1.induced_transformation(y_le_phi_inv)
sage: R5e2,_ = R4e2.induced_transformation(y_le_phi_inv)
\end{verbatim}}
    \[
    \begin{array}{c}
        \Pcal_5=
    \end{array}
    \begin{array}{c} 
    \includegraphics{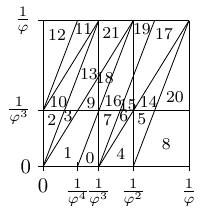}\\
    \end{array}
        \quad
    \begin{array}{l}
        R_5^\bn(\bx)=\\\bx-\varphi^{-3}\bn
    \end{array}
        \quad
    \begin{array}{c}
        \beta_{4}:\A_5\to \A_4^{*^2}\\[2mm]
        \left\{
            \setlength{\arraycolsep}{0pt}
            \footnotesize
            \begin{array}{lll}
0\mapsto \left(3\right)
,&
1\mapsto \left(4\right)
,&
2\mapsto \left(5\right)
,\\
3\mapsto \left(6\right)
,&
4\mapsto \left(12\right)
,&
5\mapsto \left(13\right)
,\\
6\mapsto \left(14\right)
,&
7\mapsto \left(15\right)
,&
8\mapsto \left(18\right)
,\\
9\mapsto \left(\begin{array}{r}
0 \\
4
\end{array}\right)
,&
10\mapsto \left(\begin{array}{r}
0 \\
5
\end{array}\right)
,&
11\mapsto \left(\begin{array}{r}
1 \\
5
\end{array}\right)
,\\
12\mapsto \left(\begin{array}{r}
2 \\
5
\end{array}\right)
,&
13\mapsto \left(\begin{array}{r}
0 \\
6
\end{array}\right)
,&
14\mapsto \left(\begin{array}{r}
8 \\
13
\end{array}\right)
,\\
15\mapsto \left(\begin{array}{r}
10 \\
14
\end{array}\right)
,&
16\mapsto \left(\begin{array}{r}
10 \\
15
\end{array}\right)
,&
17\mapsto \left(\begin{array}{r}
11 \\
16
\end{array}\right)
,\\
18\mapsto \left(\begin{array}{r}
9 \\
17
\end{array}\right)
,&
19\mapsto \left(\begin{array}{r}
11 \\
17
\end{array}\right)
,&
20\mapsto \left(\begin{array}{r}
7 \\
18
\end{array}\right)
,\\
21\mapsto \left(\begin{array}{r}
9 \\
19
\end{array}\right)
.
\end{array}\right.
    \end{array}
    \]
    Now it is appropriate to rescale the partition $\Pcal_5$ by
    the factor $-\varphi$.  Doing so, the new obtained action 
    $R'_5$ is the same as two steps before, that is, $R_3$ on
    $\R^2/\Z^2$. 
    More formally, let $h:(\R/\varphi^{-1}\Z)^2\to(\R/\Z)^2$
be the homeomorphism defined by $h(\bx)=-\varphi\bx$. 
We define 
$\Pcal_5'=h(\Pcal_5)$,
$\sccode_5'=\sccode_5\circ h^{-1}$,
 $(R'_5)^\bn=h\circ (R_5)^\bn\circ h^{-1}$
as shown below:
{\footnotesize
\begin{verbatim}
sage: P5_scaled = (-phi*P5).translate((1,1))
sage: R5e1_scaled = (-phi*R5e1).translate_domain((1,1))
sage: R5e2_scaled = (-phi*R5e2).translate_domain((1,1))
\end{verbatim}}
    \[
    \begin{array}{c}
        \Pcal_5'=
    \end{array}
    \begin{array}{c} 
    \includegraphics{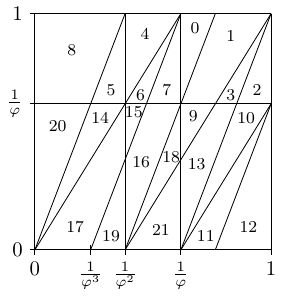}\\
    \end{array}
    \begin{array}{l}
        (R'_5)^\bn(\bx)=\bx+\varphi^{-2}\bn
    \end{array}
    \] 
By construction, the following diagrams commute:
\[ 
\begin{tikzcd}
    (\R/\varphi^{-1}\Z)^2 \arrow{r}{h} \arrow[swap]{d}{R_5^\bn}
    & (\R/\Z)^2 \arrow{d}{(R_5')^\bn} \\%
    (\R/\varphi^{-1}\Z)^2 \arrow{r}{h}
    & (\R/\Z)^2
\end{tikzcd}
\qquad
    \text{ and }
\qquad
\begin{tikzcd}
    (\R/\varphi^{-1}\Z)^2 \arrow{r}{h} \arrow[swap]{rd}{\sccode_5}
    & (\R/\Z)^2 \arrow{d}{\sccode_5'} \\%
    & \A_5 
\end{tikzcd}.
\]
Using the above commutative properties, for every $\by\in\R^2/\Gamma_5$ and
$m,n\in\Z$, we have
\begin{align*}
    \scConfig^{\Pcal_5,R_5}_{\by} (\bn)
    &= \sccode_5(R_5^\bn (\by))
    = \sccode_5' \circ h(R_5^\bn (\by))\\
    &= \sccode_5' \circ (R_5')^\bn (h(\by)))
    = \scConfig^{\Pcal_5',R_5'}_{h(\by)} (\bn).
\end{align*}
    Thus $\Xcal_{\Pcal_5,R_5} =\Xcal_{\Pcal_5',R_5'}$.
    Now we consider the window $W_5=(0,\varphi^{-1})\times(0,1)$ as a subset of
    $\R^2/\Z^2$.
    The action $R_5'$ is Cartesian on $W_5$.
    Thus from Lemma~\ref{lem:Cartesian-induced-action},
    $R_6:=\widehat{R_5'}|_{W_5}:\Z^2\times W_5\to W_5$ is a well-defined
    $\Z^2$-action.
    From Lemma~\ref{lem:PET-iff-toral-rotation},
    the $\Z^2$-action $R_6$ can be seen as toral translation on
    $\R^2/\Gamma_6$ with $\Gamma_6=(\varphi^{-1}\Z)\times\Z$.
    Let $\Pcal_6=\widehat{\Pcal_5}|_{W_5}$ be the induced partition.
    From Proposition~\ref{prop:orbit-preimage}, then
    $\Xcal_{\Pcal_5,R_5}=
    \overline{\beta_5(\Xcal_{\Pcal_6,R_6})}^\sigma$.
    The partition $\Pcal_6$, the action $R_6$ and substitution
    $\beta_5$ are given below with alphabet $\A_6=\Zrange{17}$:
{\footnotesize
\begin{verbatim}
sage: P6,beta5 = R5e1_scaled.induced_partition(x_le_phi_inv, P5_scaled, substitution_type="row")
sage: R6e1,_ = R5e1_scaled.induced_transformation(x_le_phi_inv)
sage: R6e2,_ = R5e2_scaled.induced_transformation(x_le_phi_inv)
\end{verbatim}}
    \[
    \begin{array}{c}
        \Pcal_6=
    \end{array}
        \begin{array}{c} 
    \includegraphics{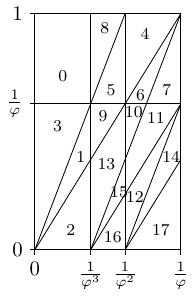}\\
    \end{array}
        \quad
    \begin{array}{c}
        R_6^\bn(\bx)=\bx\\[2mm]
        +(-\varphi^{-3},0)n_1\\[2mm]
        +(0,\varphi^{-2})n_2
    \end{array}
        \quad
    \begin{array}{c}
        \beta_{5}:\A_{6}\to \A_5^{*^2}\\[2mm]
        \left\{
            \setlength{\arraycolsep}{0pt}
            \footnotesize
            \begin{array}{ll}
0\mapsto \left(8\right)
,&
1\mapsto \left(14\right)
,\\
2\mapsto \left(17\right)
,&
3\mapsto \left(20\right)
,\\
4\mapsto \left(4,\,1\right)
,&
5\mapsto \left(5,\,1\right)
,\\
6\mapsto \left(6,\,3\right)
,&
7\mapsto \left(7,\,2\right)
,\\
8\mapsto \left(8,\,0\right)
,&
9\mapsto \left(14,\,9\right)
,\\
10\mapsto \left(15,\,13\right)
,&
11\mapsto \left(16,\,10\right)
,\\
12\mapsto \left(16,\,11\right)
,&
13\mapsto \left(17,\,13\right)
,\\
14\mapsto \left(18,\,12\right)
,&
15\mapsto \left(19,\,10\right)
,\\
16\mapsto \left(19,\,11\right)
,&
17\mapsto \left(21,\,12\right)
.
\end{array}\right.
    \end{array}
    \]
    Now we consider the window $W_6=(0,\varphi^{-1})\times(0,\varphi^{-1})$ as
    a subset of $\R^2/\Gamma_6$.
    The action $R_6$ is Cartesian on $W_6$.
    Thus from Lemma~\ref{lem:Cartesian-induced-action},
    $R_7:=\widehat{R_6}|_{W_6}:\Z^2\times W_6\to W_6$ is a well-defined
    $\Z^2$-action.
    From Lemma~\ref{lem:PET-iff-toral-rotation},
    the $\Z^2$-action $R_7$ can be seen as toral translation on
    $\R^2/\Gamma_7$ with $\Gamma_7=(\varphi^{-1}\Z)\times(\varphi^{-1}\Z)$.
    Let $\Pcal_7=\widehat{\Pcal_6}|_{W_6}$ be the induced partition.
    From Proposition~\ref{prop:orbit-preimage}, then
    $\Xcal_{\Pcal_6,R_6}=
    \overline{\beta_6(\Xcal_{\Pcal_7,R_7})}^\sigma$.
    The partition $\Pcal_7$, the action $R_7$ and substitution
    $\beta_6$ are given below with alphabet $\A_7=\Zrange{20}$:
{\footnotesize
\begin{verbatim}
sage: P7,beta6 = R6e2.induced_partition(y_le_phi_inv, P6, substitution_type="column")
sage: R7e1,_ = R6e1.induced_transformation(y_le_phi_inv)
sage: R7e2,_ = R6e2.induced_transformation(y_le_phi_inv)
\end{verbatim}}
    \[
    \begin{array}{c}
        \Pcal_7=
    \end{array}
        \begin{array}{c} 
    \includegraphics{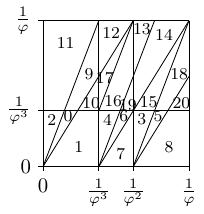}\\
    \end{array}
        \quad
    \begin{array}{l}
        R_7^\bn(\bx)=\\\bx-\varphi^{-3}\bn
    \end{array}
        \quad
    \begin{array}{c}
        \beta_{6}:\A_{7}\to \A_{6}^{*^2}\\[2mm]
        \left\{
            \setlength{\arraycolsep}{0pt}
            \footnotesize
            \begin{array}{lll}
0\mapsto \left(1\right)
,&
1\mapsto \left(2\right)
,&
2\mapsto \left(3\right)
,\\
3\mapsto \left(12\right)
,&
4\mapsto \left(13\right)
,&
5\mapsto \left(14\right)
,\\
6\mapsto \left(15\right)
,&
7\mapsto \left(16\right)
,&
8\mapsto \left(17\right)
,\\
9\mapsto \left(\begin{array}{r}
0 \\
1
\end{array}\right)
,&
10\mapsto \left(\begin{array}{r}
0 \\
2
\end{array}\right)
,&
11\mapsto \left(\begin{array}{r}
0 \\
3
\end{array}\right)
,\\
12\mapsto \left(\begin{array}{r}
8 \\
9
\end{array}\right)
,&
13\mapsto \left(\begin{array}{r}
4 \\
10
\end{array}\right)
,&
14\mapsto \left(\begin{array}{r}
4 \\
11
\end{array}\right)
,\\
15\mapsto \left(\begin{array}{r}
6 \\
12
\end{array}\right)
,&
16\mapsto \left(\begin{array}{r}
5 \\
13
\end{array}\right)
,&
17\mapsto \left(\begin{array}{r}
8 \\
13
\end{array}\right)
,\\
18\mapsto \left(\begin{array}{r}
7 \\
14
\end{array}\right)
,&
19\mapsto \left(\begin{array}{r}
5 \\
15
\end{array}\right)
,&
20\mapsto \left(\begin{array}{r}
7 \\
17
\end{array}\right)
.
\end{array}\right.
    \end{array}
    \]
    Again it is appropriate to rescale the partition $\Pcal_7$ by
    the factor $-\varphi$.  
    We use the homeomorphism 
    $h:(\R/\varphi^{-1}\Z)^2\to(\R/\Z)^2$ defined previously by
    $h(\bx)=-\varphi\bx$. 
    We define 
    $\Pcal_7'=h(\Pcal_7)$,
    $\sccode_7'=\sccode_7\circ h^{-1}$,
    $(R'_7)^\bn=h\circ (R_7)^\bn\circ h^{-1}$
    as shown below:
{\footnotesize
\begin{verbatim}
sage: P7_scaled = (-phi*P7).translate((1,1))
sage: R7e1_scaled = (-phi*R7e1).translate_domain((1,1))
sage: R7e2_scaled = (-phi*R7e2).translate_domain((1,1))
\end{verbatim}}
    \[
    \begin{array}{c}
        \Pcal_7'=
    \end{array}
        \begin{array}{c} 
        \includegraphics{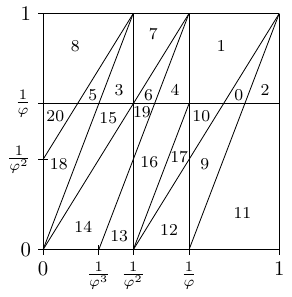}\\
        \end{array}
        \begin{array}{l}
            (R'_7)^\bn(\bx)=\bx+\varphi^{-2}\bn
        \end{array}.
    \] 
    Using the above commutative properties, for every $\by\in\R^2/\Gamma_7$ and
    $m,n\in\Z$, we have
    \begin{align*}
        \scConfig^{\Pcal_7,R_7}_{\by} (\bn)
        &= \sccode_7(R_7^\bn (\by))
        = \sccode_7' \circ h(R_7^\bn (\by))\\
        &= \sccode_7' \circ (R_7')^\bn (h(\by)))
        = \scConfig^{\Pcal_7',R_7'}_{h(\by)} (\bn).
    \end{align*}
    Thus $\Xcal_{\Pcal_7,R_7} =\Xcal_{\Pcal_7',R_7'}$.
    Now we consider the window $W_7=(0,\varphi^{-1})\times(0,1)$ as a subset of
    $\R^2/\Z^2$.
    The action $R_7'$ is Cartesian on $W_7$.
    Thus from Lemma~\ref{lem:Cartesian-induced-action},
    $R_8:=\widehat{R_7'}|_{W_7}:\Z^2\times W_7\to W_7$ is a well-defined
    $\Z^2$-action.
    From Lemma~\ref{lem:PET-iff-toral-rotation},
    the $\Z^2$-action $R_8$ can be seen as toral translation on
    $\R^2/\Gamma_8$ with $\Gamma_8=(\varphi^{-1}\Z)\times\Z$.
    Let $\Pcal_8=\widehat{\Pcal_7}|_{W_7}$ be the induced partition.
    From Proposition~\ref{prop:orbit-preimage}, then
    $\Xcal_{\Pcal_7,R_7}=
    \overline{\beta_7(\Xcal_{\Pcal_8,R_8})}^\sigma$.
    The partition $\Pcal_8$, the action $R_8$ and substitution
    $\beta_7$ are given below with alphabet $\A_8=\Zrange{18}$:
{\footnotesize
\begin{verbatim}
sage: P8,beta7 = R7e1_scaled.induced_partition(x_le_phi_inv, P7_scaled, substitution_type="row")
sage: R8e1,_ = R7e1_scaled.induced_transformation(x_le_phi_inv)
sage: R8e2,_ = R7e2_scaled.induced_transformation(x_le_phi_inv)
\end{verbatim}}
\[
    \begin{array}{c}
        \Pcal_8=
    \end{array}
\begin{array}{c} 
\includegraphics{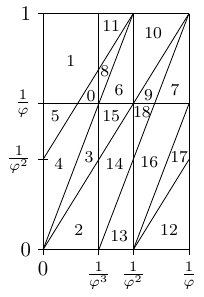}\\
\end{array}
        \quad
    \begin{array}{c}
        R_8^\bn(\bx)=\bx\\[2mm]
        +(-\varphi^{-3},0)n_1\\[2mm]
        +(0,\varphi^{-2})n_2
    \end{array}
        \quad
\begin{array}{c}
    \beta_{7}:\A_{8}\to \A_{7}^{*^2}\\[2mm]
    \left\{
            \setlength{\arraycolsep}{0pt}
            \footnotesize
        \begin{array}{ll}
0\mapsto \left(5\right)
,&
1\mapsto \left(8\right)
,\\
2\mapsto \left(14\right)
,&
3\mapsto \left(15\right)
,\\
4\mapsto \left(18\right)
,&
5\mapsto \left(20\right)
,\\
6\mapsto \left(3,\,1\right)
,&
7\mapsto \left(4,\,2\right)
,\\
8\mapsto \left(5,\,1\right)
,&
9\mapsto \left(6,\,0\right)
,\\
10\mapsto \left(7,\,1\right)
,&
11\mapsto \left(8,\,1\right)
,\\
12\mapsto \left(12,\,11\right)
,&
13\mapsto \left(13,\,11\right)
,\\
14\mapsto \left(14,\,9\right)
,&
15\mapsto \left(15,\,10\right)
,\\
16\mapsto \left(16,\,11\right)
,&
17\mapsto \left(17,\,11\right)
,\\
18\mapsto \left(19,\,9\right)
.
\end{array}\right.
\end{array}
\]
We need two more steps 
before the induction procedure loops. 
We obtain the partitions and substitutions below.
{\footnotesize
\begin{verbatim}
sage: P9,beta8 = R8e2.induced_partition(y_le_phi_inv, P8, substitution_type="column")
sage: R9e1,_ = R8e1.induced_transformation(y_le_phi_inv)
sage: R9e2,_ = R8e2.induced_transformation(y_le_phi_inv)
\end{verbatim}}
\[
    \begin{array}{c}
        \Pcal_9=
    \end{array}
\begin{array}{c} 
\includegraphics{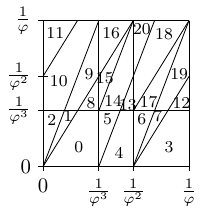}\\
\end{array}
        \quad
\begin{array}{l}
        R_9^\bn(\bx)=\\\bx-\varphi^{-3}\bn
\end{array}
        \quad
\begin{array}{c}
    \beta_{8}:\A_{9}\to \A_{8}^{*^2}\\[2mm]
    \left\{
            \setlength{\arraycolsep}{0pt}
            \footnotesize
        \begin{array}{lll}
0\mapsto \left(2\right)
,&
1\mapsto \left(3\right)
,&
2\mapsto \left(4\right)
,\\
3\mapsto \left(12\right)
,&
4\mapsto \left(13\right)
,&
5\mapsto \left(14\right)
,\\
6\mapsto \left(16\right)
,&
7\mapsto \left(17\right)
,&
8\mapsto \left(\begin{array}{r}
0 \\
2
\end{array}\right)
,\\
9\mapsto \left(\begin{array}{r}
1 \\
3
\end{array}\right)
,&
10\mapsto \left(\begin{array}{r}
1 \\
4
\end{array}\right)
,&
11\mapsto \left(\begin{array}{r}
1 \\
5
\end{array}\right)
,\\
12\mapsto \left(\begin{array}{r}
7 \\
12
\end{array}\right)
,&
13\mapsto \left(\begin{array}{r}
6 \\
13
\end{array}\right)
,&
14\mapsto \left(\begin{array}{r}
6 \\
14
\end{array}\right)
,\\
15\mapsto \left(\begin{array}{r}
8 \\
14
\end{array}\right)
,&
16\mapsto \left(\begin{array}{r}
11 \\
15
\end{array}\right)
,&
17\mapsto \left(\begin{array}{r}
9 \\
16
\end{array}\right)
,\\
18\mapsto \left(\begin{array}{r}
10 \\
16
\end{array}\right)
,&
19\mapsto \left(\begin{array}{r}
7 \\
17
\end{array}\right)
,&
20\mapsto \left(\begin{array}{r}
10 \\
18
\end{array}\right)
.
\end{array}\right.
\end{array}
\]
{\footnotesize
\begin{verbatim}
sage: P9_scaled = (-phi*P9).translate((1,1))
sage: R9e1_scaled = (-phi*R9e1).translate_domain((1,1))
sage: R9e2_scaled = (-phi*R9e2).translate_domain((1,1))
\end{verbatim}}
    \[
    \begin{array}{c}
        \Pcal_9'=
    \end{array}
        \begin{array}{c} 
    \includegraphics{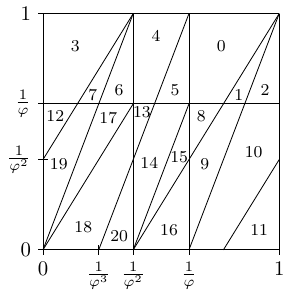}\\
    \end{array} 
        \begin{array}{l}
            (R'_9)^\bn(\bx)=\bx+\varphi^{-2}\bn
        \end{array}
    \]
{\footnotesize
\begin{verbatim}
sage: P10,beta9 = R9e1_scaled.induced_partition(x_le_phi_inv, P9_scaled, substitution_type="row")
sage: R10e1,_ = R9e1_scaled.induced_transformation(x_le_phi_inv)
sage: R10e2,_ = R9e2_scaled.induced_transformation(x_le_phi_inv)
\end{verbatim}}
\[
    \begin{array}{c}
        \Pcal_{10}=
    \end{array}
\begin{array}{c} 
\includegraphics{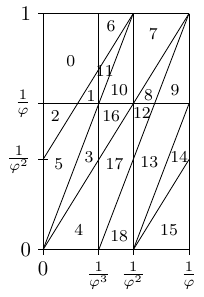}\\
\end{array}
        \quad
    \begin{array}{c}
        R_{10}^\bn(\bx)=\bx\\[2mm]
        +(-\varphi^{-3},0)n_1\\[2mm]
        +(0,\varphi^{-2})n_2
    \end{array}
        \quad
\begin{array}{c}
    \beta_{9}:\A_{10}\to \A_{9}^{*^2}\\[2mm]
\left\{
            \setlength{\arraycolsep}{0pt}
            \footnotesize
    \begin{array}{ll}
0\mapsto \left(3\right)
,&
1\mapsto \left(7\right)
,\\
2\mapsto \left(12\right)
,&
3\mapsto \left(17\right)
,\\
4\mapsto \left(18\right)
,&
5\mapsto \left(19\right)
,\\
6\mapsto \left(3,\,0\right)
,&
7\mapsto \left(4,\,0\right)
,\\
8\mapsto \left(4,\,1\right)
,&
9\mapsto \left(5,\,2\right)
,\\
10\mapsto \left(6,\,0\right)
,&
11\mapsto \left(7,\,0\right)
,\\
12\mapsto \left(13,\,9\right)
,&
13\mapsto \left(14,\,10\right)
,\\
14\mapsto \left(15,\,10\right)
,&
15\mapsto \left(16,\,11\right)
,\\
16\mapsto \left(17,\,8\right)
,&
17\mapsto \left(18,\,9\right)
,\\
18\mapsto \left(20,\,10\right)
.
\end{array}\right.
\end{array}
\]
We may observe that 
$R_{10}=R_8$ and
the partition $\Pcal_{10}$ is the same as $\Pcal_8$ 
up to a permutation $\tau$ of the indices of the atoms:
{\footnotesize
\begin{verbatim}
sage: P8.is_equal_up_to_relabeling(P10)
True
sage: from slabbe import Substitution2d
sage: tau = Substitution2d.from_permutation(P8.keys_permutation(P10))
\end{verbatim}}
\[
\begin{array}{c}
    \tau:\A_{8}\to \A_{10}\\[2mm]
\left\{
            \setlength{\arraycolsep}{0pt}
            \footnotesize
    \begin{array}{lllllll}
0\mapsto \left(1\right)
,&
1\mapsto \left(0\right)
,&
2\mapsto \left(4\right)
,&
3\mapsto \left(3\right)
,&
4\mapsto \left(5\right)
,&
5\mapsto \left(2\right)
,&
6\mapsto \left(10\right)
,\\
7\mapsto \left(9\right)
,&
8\mapsto \left(11\right)
,&
9\mapsto \left(8\right)
,&
10\mapsto \left(7\right)
,&
11\mapsto \left(6\right)
,&
12\mapsto \left(15\right)
,&
13\mapsto \left(18\right)
,\\
14\mapsto \left(17\right)
,&
15\mapsto \left(16\right)
,&
16\mapsto \left(13\right)
,&
17\mapsto \left(14\right)
,&
18\mapsto \left(12\right)
.
\end{array}\right.
\end{array}
\]
Thus the induction procedure loops as we have
\[
    \Xcal_{\Pcal_8,R_8}
    = \overline{\beta_8(\Xcal_{\Pcal_9,R_9})}^\sigma
    = \overline{\beta_8\beta_9(\Xcal_{\Pcal_{10},R_{10}})}^\sigma
    = \overline{\beta_8\beta_9\tau(\Xcal_{\Pcal_{8},R_{8}})}^\sigma
\]
The self-similarity $\beta_8\beta_9\tau$ is
{\footnotesize
\begin{verbatim}
sage: show(beta8*beta9*tau)
\end{verbatim}}
\[
\begin{array}{c}
    \beta_8\beta_9\tau:\A_{8}\to \A_{8}^{*^2}\\[2mm]
\left\{
            \setlength{\arraycolsep}{2pt}
            \footnotesize
    \begin{array}{lllll}
0\mapsto \left(17\right)
,&
1\mapsto \left(12\right)
,&
2\mapsto \left(\begin{array}{r}
10 \\
16
\end{array}\right)
,&
3\mapsto \left(\begin{array}{r}
9 \\
16
\end{array}\right)
,&
4\mapsto \left(\begin{array}{r}
7 \\
17
\end{array}\right)
,\\
5\mapsto \left(\begin{array}{r}
7 \\
12
\end{array}\right)
,&
6\mapsto \left(16,\,2\right)
,&
7\mapsto \left(14,\,4\right)
,&
8\mapsto \left(17,\,2\right)
,&
9\mapsto \left(13,\,3\right)
,\\
10\mapsto \left(13,\,2\right)
,&
11\mapsto \left(12,\,2\right)
,&
12\mapsto \left(\begin{array}{rr}
11 & 1 \\
15 & 5
\end{array}\right)
,&
13\mapsto \left(\begin{array}{rr}
10 & 1 \\
18 & 4
\end{array}\right)
,&
14\mapsto \left(\begin{array}{rr}
10 & 1 \\
16 & 3
\end{array}\right)
,\\
15\mapsto \left(\begin{array}{rr}
9 & 0 \\
16 & 2
\end{array}\right)
,&
16\mapsto \left(\begin{array}{rr}
6 & 1 \\
14 & 4
\end{array}\right)
,&
17\mapsto \left(\begin{array}{rr}
8 & 1 \\
14 & 4
\end{array}\right)
,&
18\mapsto \left(\begin{array}{rr}
6 & 1 \\
13 & 3
\end{array}\right)
.
\end{array}\right.
\end{array}
\]
which is primitive and expansive.
\end{proof}

We now want to make the link between the partition $\Pcal_8$
and dynamical system $\Xcal_{\Pcal_8,R_8}$
computed above from the induction process and the partition $\Pcal_\U$
and dynamical system $\Xcal_{\Pcal_\U,R_\U}$
that was introduced in \cite{labbe_markov_2021}.
Recall that on the torus $\torusI^2=(\R/\Z)^2$,
the dynamical system $(\torusI^2,\Z^2,R_\U)$ was defined by the
$\Z^2$-action
\[
\begin{array}{rccl}
    R_\U:&\Z^2\times\torusI^2 & \to & \torusI^2\\
    &(\bn,\bx) & \mapsto &\bx+\varphi^{-2}\bn
\end{array}
\]
where $\varphi=\frac{1+\sqrt{5}}{2}$.
In \cite{labbe_markov_2021}, we proved
that
$\Pcal_\U=\{P_a\}_{a\in\Zrange{18}}$ 
is a Markov partition 
for the dynamical system $(\torusI^2,\Z^2,R_\U)$.
It is illustrated in Figure~\ref{fig:partition-U}.

\begin{figure}[h]
    \begin{center}
    \includegraphics{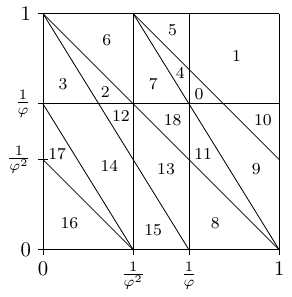}
    \end{center}
    \caption{
    The topological partition $\Pcal_\U$ of $\torusI^2$
    using alphabet $\Zrange{18}$ for the atoms.}
    \label{fig:partition-U}
\end{figure}

\begin{lemma}\label{lem:Y8-same-as-YU}
    $\Xcal_{\Pcal_8,R_8}$ and $\Xcal_{\Pcal_\U,R_\U}$ are topologically
    conjugate.
    More precisely,
    $\Xcal_{\Pcal_8,R_8}=\zeta(\Xcal_{\Pcal_\U,R_\U})$
    where $\zeta:\Zrange{18}^{*^2}\to\Zrange{18}^{*^2}$ is the $2$-dimensional
    morphism defined by the following letter-to-letter bijection:
\begin{equation}\label{eq:rho-permutation}
    \zeta:
\left\{
 \begin{array}{lllll}
0\mapsto 0,&
1\mapsto 1,&
2\mapsto 9,&
3\mapsto 7,&
4\mapsto 8,\\
5\mapsto 11,&
6\mapsto 10,&
7\mapsto 6,&
8\mapsto 2,&
9\mapsto 4,\\
10\mapsto 5,&
11\mapsto 3,&
12\mapsto 18,&
13\mapsto 14,&
14\mapsto 16,\\
15\mapsto 13,&
16\mapsto 12,&
17\mapsto 17,&
18\mapsto 15.
\end{array}\right.
\end{equation}
\end{lemma}

\begin{proof}
The topological partition $\Pcal_\U$
defines a coding map $\sccode_\U:\torusI^2\to\Zrange{18}$.
We have $R_\U^\bn(\bx)=\bx+\varphi^{-2}\bn$.
    Let $g:\R^2/\Gamma_8\to\torusI^2$ be the homeomorphism defined by
    $(x,y)\mapsto (-\varphi x, y)$.
    We observe that $g(\Pcal_8)$ is equal to $\Pcal_\U$ 
    after permuting the coding alphabet of $\Pcal_\U$ by $\zeta$.
    In other words, $\sccode_8=\zeta\circ\sccode_\U\circ g$.
    Also $g\circ R_8^\bn = R_\U^\bn\circ g$.
    Thus for every $\bx\in\R^2/\Gamma_8$ and $\bn\in\Z^2$, we have
\begin{align*}
	\scConfig^{\Pcal_8,R_8}_{\bx}(\bn)
    &= \sccode_8\circ R_8^\bn (\bx)
    = \zeta\circ\sccode_\U \circ g\circ R_8^\bn(\bx)\\
    &= \zeta\circ\sccode_\U \circ R_\U^\bn\circ g(\bx)
    = \zeta\left(\scConfig^{\Pcal_\U,R_\U}_{g(\bx)} (\bn)\right).
\end{align*}
    Thus we conclude
    $\Xcal_{\Pcal_8,R_8}=\zeta\left(\Xcal_{\Pcal_\U,R_\U}\right)$.
\end{proof}

\section{Proof of main results}

We may now complete the proof of the substitutive structure of
$\Xcal_{\Pcal_0,R_0}$.

\begin{proof}[Proof of Theorem~\ref{thm:XP0R0-subst-structure}]
    (i) The existence of $\beta_0$ was proved in Proposition~\ref{prop:desubstitute-JR-sticks}.
    (ii) The existence of the shear conjugacy $\beta_1$ was proved in Lemma~\ref{lem:beta1}.
    (iii) The existence of $\beta_2$, $\beta_3$, $\beta_4$, $\beta_5$, $\beta_6$ and $\beta_7$
    was proved in Proposition~\ref{prop:fromcode2-to-code8}.
    (iv) The existence of $\beta_8$, $\beta_9$ and $\tau$
    was proved in Proposition~\ref{prop:fromcode2-to-code8}.
    (v) The topological conjugacy
    $\zeta:\Xcal_{\Pcal_\U,R_\U}\to\Xcal_{\Pcal_8,R_8}$ was proved in
    Lemma~\ref{lem:Y8-same-as-YU}.
\end{proof}

For comparison with
the statement of Theorem~\ref{thm:XP0R0-subst-structure},
we now recall the main result proved in 
    \cite{MR4226493} 
    about the substitutive structure of
    the minimal subshift $X_0\subset\Omega_0$ of the Jeandel-Rao Wang shift.

\begin{theorem}\label{thm:article2}
    {\rm\cite{MR4226493}}
    Let $\Omega_0$ be the Jeandel-Rao Wang shift.
There exist sets of Wang tiles $\{\T_{i}\}_{1\leq i\leq 12}$
together with their associated Wang shifts $\{\Omega_{i}\}_{1\leq i\leq 12}$
that provide the substitutive structure of Jeandel-Rao Wang shift.
More precisely,
\begin{enumerate}[\rm (i)]
\item There exists a sequence of recognizable $2$-dimensional morphisms:
    \begin{equation*}
        \Omega_0 \xleftarrow{\omega_0}
        \Omega_1 \xleftarrow{\omega_1}
        \Omega_2 \xleftarrow{\omega_2}
        \Omega_3 \xleftarrow{\omega_3}
        \Omega_4
    \end{equation*}
    that are onto up to a shift, i.e.,
    $\overline{\omega_i(\Omega_{i+1})}^\sigma=\Omega_{i}$
    for each $i\in\{0,1,2,3\}$.
\item There exists an embedding $\jmath:\Omega_5\to\Omega_4$ which is a
    topological conjugacy onto its image.
\item There exists a shear conjugacy
    $\eta:\Omega_6\to\Omega_5$
    which shears configurations by the action of the matrix
    $\left(\begin{smallmatrix}
            1 & 1 \\
            0 & 1
    \end{smallmatrix}\right)$.
\item
    There exists a sequence of recognizable $2$-dimensional morphisms:
    \begin{equation*}
        \Omega_6 \xleftarrow{\omega_6}
        \Omega_7 \xleftarrow{\omega_7}
        \Omega_8 \xleftarrow{\omega_8}
        \Omega_9 \xleftarrow{\omega_9}
        \Omega_{10} \xleftarrow{\omega_{10}}
        \Omega_{11} \xleftarrow{\omega_{11}}
        \Omega_{12} 
    \end{equation*}
    that are onto up to a shift, i.e.,
    $\overline{\omega_i(\Omega_{i+1})}^\sigma=\Omega_{i}$
    for each $i\in\{6,7,8,9,10,11\}$.
\item The Wang shift $\Omega_{12}$ is equivalent to $\Omega_\U$
    for some topological conjugacy $\rho:\Omega_\U\to\Omega_{12}$:
    \begin{equation*}
        \Omega_{12} \xleftarrow{\rho}
        \Omega_{\U} 
    \end{equation*}
        thus $\Omega_{12}$ is self-similar, aperiodic and minimal.
\item $X_0=\omega_0\omega_1\omega_2\omega_3\jmath(\Omega_5)\subsetneq\Omega_0$ is
    an aperiodic minimal subshift of Jeandel-Rao Wang shift $\Omega_0$.
\end{enumerate}
\end{theorem}

The Wang shift $\Omega_\U$ given by a set of 19 Wang tiles describing the
internal self-similar structure hidden in the Jeandel-Rao Wang shift was
studied separately in \cite{MR3978536}.

\begin{theorem}
    \label{thm:article1}
    {\rm\cite{MR3978536}}
    The Wang shift $\Omega_\U$ is aperiodic, minimal
    and self-similar 
    satisfying $\Omega_\U=\overline{\omega_\U(\Omega_\U)}^\sigma=\X_{\omega_\U}$
    for some primitive and expansive $2$-dimensional morphism $\omega_\U$.
\end{theorem}

The substitutions $\omega_i$, $\jmath$ and $\eta$ were computed
from algorithms on Wang tiles based on marker tiles.
Surprisingly, they are the same as the substitutions constructed here in
previous sections using 2-dimensional Rauzy induction on toral partitions!
Thus, we can now prove that the symbolic dynamical system $\Xcal_{\Pcal_0,R_0}$
and the subshift $X_0\subset\Omega_0$ have the same substitutive structure.

\begin{proof}[Proof of Theorem~\ref{thm:substitutionequality}]
    The reader may check that the equality
    $\beta_0=\omega_0\,\omega_1\,\omega_2\,\omega_3$
    holds by comparing $\beta_0$ computed in
    Proposition~\ref{prop:desubstitute-JR-sticks}
    with the product
    $\omega_0\,\omega_1\,\omega_2\,\omega_3$
    shown in \cite[end of \S 5]{MR4226493}.

    The maps $\eta:\Omega_6\to\Omega_5$
    and $\beta_1:\Xcal_{\Pcal_2,R_2}\to\Xcal_{\Pcal_1,R_1}$
    are shear conjugacies both shearing configurations by the action of the
    matrix $M=\left(\begin{smallmatrix}
        1 & 1 \\
        0 & 1
    \end{smallmatrix}\right)$.
    The reader may check that 
    the effect of the map $\beta_1\beta_2$
    on the alphabet $\A_3=\Zrange{19}$ is the same
    as $\jmath\,\eta\,\omega_6$.

    The reader may check that the equalities
        $\beta_3=\omega_7$,
        $\beta_4=\omega_8$,
        $\beta_5=\omega_9$,
        $\beta_6=\omega_{10}$ and
        $\beta_7=\omega_{11}$
    hold
    by comparing the $\beta_i$ computed in
    Proposition~\ref{prop:fromcode2-to-code8}
    with the $\omega_{i+4}$ computed in
    \cite[Prop.~8.1]{MR4226493}.

    The reader may verify the equality
        $\zeta=\rho$
        by comparing $\zeta$ computed in
        Lemma~\ref{lem:Y8-same-as-YU}
        with $\rho$ computed in
    \cite[Prop.~9.1]{MR4226493}.
    Both maps are $2$-dimensional morphism defined by a permutation of the alphabet.

    To prove that
        $\beta_8\,\beta_9\,\tau=\rho\,\omega_\U\,\rho^{-1}$
    the reader may compare the product $\beta_8\beta_9\tau$
    shown at the end of the proof of Proposition~\ref{prop:fromcode2-to-code8}
    but unfortunately the product $\rho\,\omega_\U\,\rho^{-1}$
    is not computed explicitly
    in \cite{MR4226493}.
    An alternative is to compute
        $\rho^{-1}\beta_8\,\beta_9\,\tau\,\rho
        =\zeta^{-1}\beta_8\,\beta_9\,\tau\,\zeta$
        and prove that it is equal to $\omega_\U$. We obtain
{\footnotesize
\begin{verbatim}
sage: zeta = Substitution2d.from_permutation({0:0, 1:1, 2:9, 3:7, 4:8, 5:11, 6:10, 
....: 7:6, 8:2, 9:4, 10:5, 11:3, 12:18, 13:14, 14:16, 15:13, 16:12, 17:17, 18:15})
sage: show(zeta.inverse()*beta8*beta9*tau*zeta)
\end{verbatim}}
    \[
        \begin{array}{c}
        \zeta^{-1}\beta_8\,\beta_9\,\tau\,\zeta:\Zrange{18}\to\Zrange{18}^{*^2}\\[2mm]
        \left\{
            \setlength{\arraycolsep}{2pt}
            \footnotesize
            \begin{array}{lllll}
0\mapsto \left(17\right)
,&
1\mapsto \left(16\right)
,&
2\mapsto \left(15,\,11\right)
,&
3\mapsto \left(13,\,9\right)
,&
4\mapsto \left(17,\,8\right)
,\\
5\mapsto \left(16,\,8\right)
,&
6\mapsto \left(15,\,8\right)
,&
7\mapsto \left(14,\,8\right)
,&
8\mapsto \left(\begin{array}{r}
6 \\
14
\end{array}\right)
,&
9\mapsto \left(\begin{array}{r}
3 \\
17
\end{array}\right)
,\\
10\mapsto \left(\begin{array}{r}
3 \\
16
\end{array}\right)
,&
11\mapsto \left(\begin{array}{r}
2 \\
14
\end{array}\right)
,&
12\mapsto \left(\begin{array}{rr}
7 & 1 \\
15 & 11
\end{array}\right)
,&
13\mapsto \left(\begin{array}{rr}
6 & 1 \\
14 & 11
\end{array}\right)
,&
14\mapsto \left(\begin{array}{rr}
7 & 1 \\
13 & 9
\end{array}\right)
,\\
15\mapsto \left(\begin{array}{rr}
6 & 1 \\
12 & 9
\end{array}\right)
,&
16\mapsto \left(\begin{array}{rr}
5 & 1 \\
18 & 10
\end{array}\right)
,&
17\mapsto \left(\begin{array}{rr}
4 & 1 \\
13 & 9
\end{array}\right)
,&
18\mapsto \left(\begin{array}{rr}
2 & 0 \\
14 & 8
\end{array}\right)
.
\end{array}\right.
    \end{array}
\]
    which is equal to the 2-dimensional morphism $\omega_\U$
    computed in \cite{MR3978536}.
\end{proof}

We may now relate the self-similarity
$\zeta^{-1}\beta_8\,\beta_9\,\tau\,\zeta$ of $\Xcal_{\Pcal_\U,R_\U}$ computed
here from the induction of $\Z^2$-action and toral partitions with the
$2$-dimension morphism $\omega_\U$ computed in \cite{MR3978536} from Wang tiles
and the notion of markers tiles.
This allows to conclude
the equality $\Xcal_{\Pcal_\U,R_\U} = \Omega_\U$
from their self-similarity
since both systems have, for independent reasons, the same substitutive structure.
Note that this equality was proved in \cite{labbe_markov_2021} 
using another approach based on the refinement of four partitions associated
to the set $\U$ of 19 Wang tiles.
As a consequence, it shows the equality between the minimal subshift
$X_0\subset\Omega_0$ of Jeandel-Rao Wang shift and the symbolic dynamical
system $\Xcal_{\Pcal_0,R_0}$.

\begin{proof}[Proof of Corollary~\ref{cor:equality-X0=XP0R0}]
From Theorem~\ref{thm:XP0R0-subst-structure},
we have that
    the subshift $\Xcal_{\Pcal_8,R_8}$ is self-similar satisfying
    $\Xcal_{\Pcal_8,R_8}=\overline{\beta_{8}\beta_9\tau(\Xcal_{\Pcal_{8},R_{8}})}^\sigma$
    where the product $\beta_8\beta_9\tau$ is an expansive and primitive
    self-similarity and $\Xcal_{\Pcal_{8},R_{8}}=\zeta(\Xcal_{\Pcal_\U,R_\U})$
    where $\zeta$ is a topological conjugacy.
    Therefore,
    \[
        \Xcal_{\Pcal_\U,R_\U}
        =\overline{\zeta^{-1}\beta_8\,\beta_9\,\tau\,\zeta(\Xcal_{\Pcal_\U,R_\U})}^\sigma
        =\overline{\omega_\U(\Xcal_{\Pcal_\U,R_\U})}^\sigma.
    \]
As a consequence the symbolic dynamical system
$\Xcal_{\Pcal_\U,R_\U}$ contains the substitutive subshift $\X_{\omega_\U}$.
From Lemma~\ref{lem:minimal-aperiodic}
$\Xcal_{\Pcal_\U,R_\U}$ is minimal from which we conclude that
$\Xcal_{\Pcal_\U,R_\U}=\X_{\omega_\U}$.
    From Theorem~\ref{thm:article1}, we have
    $\Omega_\U=\overline{\omega_\U(\Omega_\U)}^\sigma=\X_{\omega_\U}$
    so that $\Xcal_{\Pcal_\U,R_\U}=\X_{\omega_\U}=\Omega_\U$.

    From Theorem~\ref{thm:article2},
    we have
    $\rho(\Omega_\U) = \Omega_{12}$.
    From Theorem~\ref{thm:substitutionequality}, the equality $\rho=\zeta$ holds.
    Thus 
    \[
        \Xcal_{\Pcal_8,R_8}
        =\zeta(\Xcal_{\Pcal_\U,R_\U})
        =\rho(\Omega_\U)
        =\Omega_{12}.
    \]

    From Theorem~\ref{thm:XP0R0-subst-structure},
    $\Xcal_{\Pcal_i,R_i}=\overline{\beta_i(\Xcal_{\Pcal_{i+1},R_{i+1}})}^\sigma$
    for each $i\in\{0,1,\dots,7\}$.
    Thus using Theorem~\ref{thm:article2}
    and
    Theorem~\ref{thm:substitutionequality}, we obtain
\begin{align*}
        \Xcal_{\Pcal_0,R_0}
          &=\overline{\beta_0\beta_1\beta_2\beta_3\beta_4\beta_5\beta_6\beta_7
           (\Xcal_{\Pcal_8,R_8})}^\sigma\\
          &=\overline{\omega_{0}\omega_{1}\omega_{2}\omega_{3}\jmath
            \eta\,\omega_6 \omega_7 \omega_8 \omega_9 \omega_{10} \omega_{11}
            (\Omega_{12})}^\sigma
           =X_{0}\subsetneq\Omega_{0}.
\end{align*}
Similarly for intermediate subshifts, we have
\begin{align*}
        \Xcal_{\Pcal_1,R_1}
          =\overline{\beta_1\beta_2\beta_3\beta_4\beta_5\beta_6\beta_7
           (\Xcal_{\Pcal_8,R_8})}^\sigma
          &=\overline{\jmath \eta\,\omega_6 \omega_7 \omega_8 \omega_9 \omega_{10}
           \omega_{11} (\Omega_{12})}^\sigma
           =\jmath(\Omega_5)\subsetneq\Omega_{4}\\
        \Xcal_{\Pcal_3,R_3}
          =\overline{\beta_3\beta_4\beta_5\beta_6\beta_7
           (\Xcal_{\Pcal_8,R_8})}^\sigma
          &=\overline{\omega_7 \omega_8 \omega_9 \omega_{10}
           \omega_{11} (\Omega_{12})}^\sigma
           =\Omega_7\\
        \Xcal_{\Pcal_4,R_4}
          =\overline{\beta_4\beta_5\beta_6\beta_7
           (\Xcal_{\Pcal_8,R_8})}^\sigma
          &=\overline{\omega_8 \omega_9 \omega_{10} \omega_{11} 
            (\Omega_{12})}^\sigma
           =\Omega_8\\
        \Xcal_{\Pcal_5,R_5}
          =\overline{\beta_5\beta_6\beta_7(\Xcal_{\Pcal_8,R_8})}^\sigma
          &=\overline{\omega_9 \omega_{10} \omega_{11}(\Omega_{12})}^\sigma
           =\Omega_9\\
        \Xcal_{\Pcal_6,R_6}
          =\overline{\beta_6\beta_7(\Xcal_{\Pcal_8,R_8})}^\sigma
          &=\overline{\omega_{10} \omega_{11}(\Omega_{12})}^\sigma
           =\Omega_{10}\\
        \Xcal_{\Pcal_7,R_7}
          =\overline{\beta_7(\Xcal_{\Pcal_8,R_8})}^\sigma
          &=\overline{\omega_{11}(\Omega_{12})}^\sigma
           =\Omega_{11}
\end{align*}
which ends the proof.
\end{proof}

We may now prove the main result.

\begin{proof}[Proof of Theorem~\ref{thm:XP0R0-markov-partition}]
    In Corollary~\ref{cor:equality-X0=XP0R0},
    we proved that $\Xcal_{\Pcal_0,R_0}=X_0$.
    In \cite[Lemma 6.2]{MR4226493}, we proved that $X_0$
    is a shift of finite type.
    In \cite{labbe_markov_2021}, we proved that the partition $\Pcal_0$ gives a
    symbolic representation of $(\R^2/\Gamma_0,\Z^2,R_0)$.
    Thus $\Pcal_0$ is a Markov partition for the dynamical system $(\R^2/\Gamma_0, \Z^2, R_0)$.
\end{proof}

\begin{proof}[Proof of Corollary~\ref{cor:isomorphic}]
    The same statement was proved in 
    \cite{labbe_markov_2021} for $\Xcal_{\Pcal_0,R_0}$
    and we proved in Corollary~\ref{cor:equality-X0=XP0R0}
    that $X_0=\Xcal_{\Pcal_0,R_0}$.
\end{proof}

\part*{Conclusion}

\section{Concluding remarks}

This concludes a study of an aperiodic subshift of finite type defined by 11
Wang tiles which was discovered by Jeandel and Rao
\cite{jeandel_aperiodic_2021}. Jeandel-Rao proved that 11 is minimal as Wang
shifts defined with fewer Wang tiles are either empty or contain a periodic
configuration. Our work ended up being split into four articles in a
publication ordering which does not correspond to the order in which the
discoveries were made.
Before discussing possible extensions, it may be worth to take a
step back and explain the research process leading to this work and the relations
between the articles.

Indeed, 
the partition $\Pcal_0$ and $\Z^2$-rotation $R_0$ shown in
Figure~\ref{fig:2d-walk} discovered after few months of experimentations in 2017 and
published in \cite{labbe_markov_2021} was the starting point. 
Surprisingly, the weird polygonal
partition $\Pcal_0$ and toral $\Z^2$-rotations $R_0$
allows to easily build valid configurations
with the 11 Jeandel-Rao tiles, i.e, $\Xcal_{\Pcal_0,R_0}\subseteq\Omega_0$.
The next question was whether $\Omega_0\subseteq\Xcal_{\Pcal_0,R_0}$.
To answer this question,
it was
natural to study the substitutive structure of the symbolic dynamical system
$\Xcal_{\Pcal_0,R_0}$ using a higher-dimensional version of Rauzy induction as
done in the current article through computations of the $\beta_i$'s. 
Then, we needed to check whether the Wang shift $\Omega_0$ also has the same
substitutive structure given by the $\beta_i$'s,
including the shear-conjugacy $\beta_1$.
Polygonal partitions and the $\beta_i$'s are not mentioned in 
\cite{MR4226493}, but they served as beacons for the
author to describe the substitutive structure of the Jeandel-Rao Wang shift
\textit{independently} of the polygonal partitions
and uniquely from the Wang tiles themselves and the
desubstitution of Wang shifts with the notion of marker tiles. 
This process leads to the self-similar structure hidden in the Jeandel-Rao Wang
shift which was considered separately as a first step \cite{MR3978536}.

The idea of writing the substitutions in a canonical way independent of the
partitions and of the Wang tiles from which they are computed, as mentioned in
Remark~\ref{rem:military-ordering}
and Definition~\ref{def:definition-radix-order}, came after the
publication of \cite{MR3978536}. This is why the bijections $\rho$ and $\zeta$
are not the identity map because the ordering of the tiles defining $\Omega_\U$
was chosen according to another convention.

It turns out that $\Omega_0\setminus\Xcal_{\Pcal_0,R_0}$ is non-empty due to
the existence of horizontal fault lines in some configurations of $\Omega_0$ 
($\Xcal_{\Pcal_0,R_0}$ is minimal, but $\Omega_0$ is not).
We believe that the difference is a null set, see
Conjecture~\ref{conjecture:uniquely-erogodic}.  Some time was spent on the
question with Jennifer McLoud-Mann and Casey Mann during their sabbatical year
2019--2020 in Bordeaux, but unfortunately, proving it seems a challenge beyond reach.

\subsection*{Some open questions}
This article and the three others were dedicated to the study of a single Wang shift.
The next step is to find other examples or even families of examples hoping
that the tools developed here will simplify their description.
It seems reasonable that there is a characterization of the toral
$\Z^2$-rotations which admit symbolic dynamical systems that are subshifts of
finite type or more generally sofic subshifts.
In the spirit of $\beta$-expansion of real numbers in real bases,
the condition could be expressed algebraically
involving for instance Parry numbers, see \cite[Theorem 2.3.15]{MR2742574}.

Of course, some impossibility results are expected. Since there are only countably
many $\Z^2$-SFTs, we can not encode all toral $\Z^2$-rotations into SFTs.
For instance, a $\Z^2$-rotation given by non-computable real numbers can not be
encoded into a finite set of Wang tiles.

Observe that the $19\times19$ incidence matrix of $\beta_8\beta_9\tau$ is not hyperbolic
but, as shown by its characteristic polynomial, it is hyperbolic on a
8-dimensional subspace:
{\footnotesize
\begin{verbatim}
sage: (beta8*beta9*tau).incidence_matrix().charpoly().factor()
x^3 * (x - 1)^4 * (x + 1)^4 * (x^2 - 3*x + 1) * (x^2 + x - 1)^3
\end{verbatim}}
\noindent 
One question is whether the polygonal Markov partition $\Pcal_\U$
for the $\Z^2$-rotation $R_\U$ on $\torusI^2$ is related or is the projection of
some 8-dimensional Markov partition of the hyperbolic automorphism on $\torusI^8$
associated to the restriction of the action of the incidence matrix of the
self-similarity $\omega_\U$ to a subspace.

It is known from the work of Bowen \cite{MR474415} that 
for hyperbolic automorphisms of $\torusI^3$,
the boundary of the atoms of the Markov partitions are typically fractal.
Later, Cawley \cite{MR1145614} proved that
the only hyperbolic toral automorphisms $f$ for which there
exist Markov partitions with piecewise smooth boundary are those for which a
power $f^k$ is linearly covered by a direct product of automorphisms of the 2-torus.
These results give the impression that smooth polyhedral Markov partitions for
$\Z^d$-rotations on $\torusI^d$ exist only in the case of $\Z^d$-rotations
involving quadratic integers. Also it would be interesting to find an example of
a fractal Markov partition of $\torusI^2$ for a $\Z^2$-rotation involving algebraic
integers of degree $\geq3$, or maybe it does not exist?

As pointed out by the referee, the examples provided in this article are such
that the lattice $\Gamma_0$ and $\Z^2$ (the orbit of zero under the
$\Z^2$-action $R_0$) have parallel elements (e.g., $(1,0)$ is parallel to
$(\varphi,0)$). We do not know if similar results could be proven in the case
when $\Gamma_0$ has no vectors parallel to a vector of $\Z^2$.

\bibliographystyle{myalpha} 
\bibliography{biblio}


\end{document}